\documentclass[12pt]{amsart}

\usepackage[margin = 1in]{geometry}
\usepackage{amsfonts, amsmath,amscd, amssymb, cite, color, latexsym, mathrsfs, mathtools,
slashed, stmaryrd, verbatim, wasysym , tensor }
\usepackage[all]{xy}
\usepackage{graphicx}
\usepackage{hyperref}
\usepackage[applemac]{inputenc}

\newtheorem{theorem}{Theorem}
\newtheorem{corollary}[theorem]{Corollary}

\newtheorem{lemma}[theorem]{Lemma}
\newtheorem{proposition}[theorem]{Proposition}
\newtheorem{remark}{Remark}

\numberwithin{equation}{section}
\numberwithin{theorem}{section}

\let\oldsqrt\sqrt
\def\sqrt{\mathpalette\DHLhksqrt}
\def\DHLhksqrt#1#2{%
\setbox0=\hbox{$#1\oldsqrt{#2\,}$}\dimen0=\ht0
\advance\dimen0-0.2\ht0
\setbox2=\hbox{\vrule height\ht0 depth -\dimen0}%
{\box0\lower0.4pt\box2}}

\newcommand{\bhs}[1]{\mathfrak B_{#1}}
\renewcommand{\tilde}{\widetilde}
\renewcommand{\bar}{\overline}
\renewcommand{\Re}{\operatorname{Re}}
\renewcommand{\hat}[1]{\widehat{#1}}

\newcommand{\wt}[1]{\widetilde{#1}}
\newcommand{\rest}[1]{\big\rvert_{#1}} 

\newcommand{\Rint}{\sideset{^R}{}\int}

\newcommand\lra{\longrightarrow}
\newcommand\xlra[1]{\xrightarrow{\phantom{x} #1 \phantom{x}}}

\newcommand\pa{\partial}

\newcommand\eps\varepsilon
\renewcommand\epsilon\varepsilon

\newcommand\ephi{\operatorname{\eps, \phi}}
\newcommand\ed{\operatorname{\eps, d}}
\newcommand\ehc{\operatorname{\eps, hc}}

\newcommand\hc{\operatorname{hc}}

\newcommand\Ephi{{}^{\ephi}}
\newcommand\Ed{{}^{\ed}}
\newcommand\Ehc{{}^{\ehc}}

\newcommand\Hc{{}^{\hc}}

\newcommand\CI{{\mathcal{C}}^{\infty}}

\newcommand{\lrpar}[1]{\left( #1 \right)}
\newcommand{\lrspar}[1]{\left[ #1 \right]}
\newcommand\ang[1]{\langle #1 \rangle}
\newcommand{\lrbrac}[1]{\left\lbrace #1 \right\rbrace}
\newcommand{\norm}[1]{\lVert #1 \rVert}

\DeclareMathOperator*{\btimes}{\times} 
\newcommand\fib{\operatorname{---}} 

\renewcommand\det{\operatorname{det}}

\newcommand\dR{\operatorname{dR}}

\newcommand\ev{\operatorname{even}}

\DeclareMathOperator*{\FP}{\operatorname{FP}}

\newcommand\Id{\operatorname{Id}}
\renewcommand\Im{\operatorname{Im}}

\newcommand\lAT{\operatorname{LAT}} 

\newcommand\odd{\operatorname{odd}}

\newcommand\pt{\operatorname{pt}}

\renewcommand\Re{\operatorname{Re}}
\DeclareMathOperator*\Res{\operatorname{Res}}

\newcommand\RTr[1]{{}^R\!\operatorname{Tr}\left( #1 \right)}
\newcommand\RStr[1]{{}^R\!\operatorname{Str}\left( #1 \right)}

\newcommand\sign{\operatorname{sign}}
\newcommand\tsmall{\mathrm{small}}

\newcommand\Spec{\operatorname{Spec}}

\newcommand\Str{\operatorname{Str}}

\newcommand\Tr{\operatorname{Tr}}

\newcommand\pr{\operatorname{pr}}

\newcommand\IH{\operatorname{IH}}
\newcommand\GL{\operatorname{GL}}
\newcommand\hM{\widehat{M}}

\newcommand\im{\operatorname{im}}
\newcommand\sma{\operatorname{small}}
\newcommand\bm{\overline{m}}

\newcommand\tdelta{\widetilde{\delta}}

\newcommand\bs{\bhs{sb}}
\newcommand\sm{\bhs{sm}}

\newcommand\mf{\bhs{mf}}

\newcommand\Mand{\text{ and }}
\newcommand\Mas{\text{ as }}
\newcommand\Mat{\text{ at }}

\newcommand\Mforall{\text{ for all }}

\newcommand\Mif{\text{ if }}

\newcommand\Mor{\text{ or }}

\newcommand\Mwith{\text{ with }}

\newcommand\paperintro%
        {%
         }
\newcommand\paperbody%
        {%
         }

\newcommand\bbB{\mathbb{B}}
\newcommand\bbC{\mathbb{C}}

\newcommand\bbN{\mathbb{N}}

\newcommand\bbR{\mathbb{R}}
\newcommand\bbS{\mathbb{S}}

\newcommand\bbZ{\mathbb{Z}}

\newcommand\bN{\mathbf{N}}

\newcommand\cC{\mathcal{C}}

\newcommand\cH{\mathcal{H}}

\newcommand\cK{\mathcal{K}}

\newcommand\cM{\mathcal{M}}

\newcommand\cO{\mathcal{O}}
\newcommand\cP{\mathcal{P}}

\newcommand\cR{\mathcal{R}}

\newcommand\cV{\mathcal{V}}

\newcommand\sE{\mathscr{E}}

\newcommand\sT{\mathscr{T}}

\newcommand\td{\widetilde{d}}

\newcommand\tH{\mathrm{H}}

\DeclareMathAlphabet{\mathpzc}{OT1}{pzc}{m}{it}

\begin{document}

\title[Analytic torsion and R-torsion on manifolds with cusps]{Analytic torsion and R-torsion of Witt representations on manifolds with cusps}
\author{Pierre Albin}
\address{Department of Mathematics, University of Illinois at Urbana-Champaign}
\email{palbin@illinois.edu}
\author{Fr\'ed\'eric Rochon}
\address{D\'epartement de Math\'ematiques, UQ\`AM}
\email{rochon.frederic@uqam.ca }
\author{David Sher}
\address{Department of Mathematics, University of Michigan}
\email{dsher@umich.edu}

\begin{abstract}
We establish a Cheeger-M\"uller theorem for unimodular representations satisfying a Witt condition on a noncompact manifold with cusps.
This class of spaces includes all non-compact hyperbolic spaces of finite volume, but we do not assume that the metric has constant curvature nor that the link of the cusp is a torus. We use renormalized traces in the sense of Melrose to define the analytic torsion and we relate it to 
the intersection R-torsion of Dar of the natural compactification to a stratified space. Our proof relies on our recent work on the behavior of the Hodge Laplacian spectrum on a closed manifold undergoing degeneration to a manifold with fibered cusps.
\end{abstract}

\maketitle

\tableofcontents

\paperintro
\section*{Introduction}

The celebrated theorem of Cheeger and M\"uller establishes the equality of Reidemeister and analytic torsion on an odd-dimensional closed manifold equipped with a flat Euclidean bundle. This was originally conjectured by Ray-Singer \cite{Ray-Singer}, proven by Cheeger and M\"uller \cite{Cheeger1979, Muller1978} and subsequently extended by M\"uller \cite{Muller1993} and Bismut-Zhang \cite{Bismut-Zhang}.
The importance and usefulness of the theorem stems from the fact that the Reidemeister torsion, or R-torsion, is a combinatorial invariant of simplicial complexes while the analytic torsion is a smooth invariant defined via the spectrum of the Hodge Laplacian. This connection is behind many applications in topology, number theory, and mathematical physics.

One particularly interesting aspect of this theorem is that it allows us to use analysis to study the size of the torsion in cohomology. For example, if
\begin{equation*}
	F^0 \xlra{d^0} F^1 \xlra{d^1} \ldots \lra F^n \xlra{d^n} 0
\end{equation*}
is a complex of free abelian groups and $K^i = F^i \otimes \bbR$ then the Reidemeister torsion (after some canonical choices, see \cite[Example 1.3]{Cheeger1979}) is given by 
\begin{equation*}
	\text{$R$-torsion} = \prod \frac{ |H^{2k+1}(F^{\bullet})_{\text{torsion}}| }{ |H^{2k}(F^{\bullet})_{\text{torsion}}| }.
\end{equation*}
This relationship has been recently exploited to study the growth of torsion in group homology by studying the analytic torsion of locally symmetric spaces \cite{BV,Calegari-Venkatesh, Muller:AsympRSTHyp3Mfds,Marshall-Muller,Raimbault:Asymp, Raimbault:ARHtorsion,Muller-Pfaff:OnAsympRSATCmptHypMfds, Muller-Pfaff:ATL2TorsionCmptLocSymSpaces, Muller-Pfaff:GrowthTorsionCohoArithGps,BMZ, BSV, Menal-Ferrer-Porti:HigherDRTCuspedHyp3Mfds}. \\

Since locally symmetric spaces often have a natural compactification to a stratified space, it is natural to look for an analogue of the the Cheeger-M\"uller theorem in this context.  On such a space, the natural cohomology to consider is the intersection cohomology of Goresky-MacPherson \cite{GM1980,GM1983}.  In 1987, Dar \cite{Dar} introduced the intersection $R$-torsion on stratified spaces, an analogue of the $R$-torsion defined in terms of intersection cohomology and a choice of perversity.  Dar also proposed that the intersection $R$-torsion should be related to the analytic torsion of some appropriately chosen incomplete iterated edge metric adapted to the singularities of the stratified spaces.  There are many recent advances on this question \cite{Hartmann-Spreafico2010, Vertman, Mazzeo-Vertman, Hartmann-Spreafico2011, Lesch2013, Sher:ConicDeg, Guillarmou-Sher, Dai-Huang}, but still no relation obtained even in the simplest case where the stratified space has only isolated conical singularities.

In \cite{ARS1}, we  proposed instead to relate the intersection $R$-torsion on a stratified space of depth one  to the analytic torsion of a  fibred cusp metric  using the geometric microlocal analysis methods of Melrose.  Specifically, let $N$ be the interior of a manifold with boundary $\bar N$ and assume that the boundary participates in a fiber bundle of closed manifolds
\begin{equation*}
	Z \fib \pa N \xlra{\phi} Y.
\end{equation*}
Let $x$ be a boundary defining function for $\pa N,$ that is, a smooth function on $\bar N$ that vanishes precisely at $\pa N$ and with non-vanishing differential there. A metric $g_d$ on $N$ is a fibered cusp metric, or $d$-metric, if it is asymptotically of the form
\begin{equation*}
	g_d \sim \frac{dx^2}{x^2} + x^2g_Z + \phi^*g_Y
\end{equation*}
where $g_Z + \phi^*g_Y$ is a submersion metric on $\pa M.$ We say $g_d$ is an `even' $d$-metric if $g_Z$ and $g_Y$ are functions of $x^2,$ see \cite[\S 7.3]{ARS1} for more details.
Let $F \lra \bar N$ be a vector bundle with flat connection $\nabla^F$ induced by a unimmodular representation $\alpha:\pi_1(N)\lra \mathrm{GL}(k,\bbR)$, that is, such that $|\det \alpha|=1.$  Endow $F$ with a  bundle metric $g_F,$ not necessarily compatible with $\nabla^F,$ but smooth all the way down to $\pa N.$ In fact we assume that $g_F$ is {\em even}, meaning that it extends smoothly to the double of $N$ across $\pa N.$

The analytic torsion of $(N,g_d,F,g_F)$ is defined in \cite[\S~10]{ARS1} following \cite{MelroseAPS} by means of the renormalized trace of the heat kernel, since the usual operator trace is not defined. 
However, if we suppose that the flat vector bundle $F$ is `strongly acyclic at infinity' in that 
\begin{equation*}
	\mathrm{H}^*(\pa N/Y;F)=0,
\end{equation*}
then there is no continuous spectrum and the heat kernel is in fact trace class, so there is no need to renormalized the trace and the analytic torsion of $(N,g_d,F,g_F)$ can then be defined directly.  In this setting, we prove in  \cite[Corollary 12.2]{ARS1} that this analytic torsion can be expressed in terms of the Reidemeister torsion of $\bar N$ relative to $\partial N$.

In the present paper, we specialize to cusp metrics, which are fibered cusps for which $Y$ is a point, and we replace the strong acyclicity condition at infinity on $\alpha$ with a much weaker `Witt condition' for which small eigenvalues do appear.  With this weaker condition the Hodge Laplacian typically has continuous spectrum (albeit bounded away from $0$) and the heat kernel is no longer trace class. \\

We can describe exactly how we will manage this extension by recalling the proof of \cite[Corollary 12.2]{ARS1} in the case where $Y$ is a point. Let $M$ be a smooth closed manifold obtained by doubling $\bar N$ across $\pa N=Z,$ and $F$ a flat bundle over $M.$
We consider a family of metrics $\eps \mapsto g_{\ehc}$ that in a tubular neighborhood of $Z$ has the form
\begin{equation*}
	g_{\ehc} = \frac{dx^2}{x^2+\eps^2} + (x^2+\eps^2) g_Z.
\end{equation*}
This can be visualized as stretching the manifold $M$ in the direction normal to the hypersurface $Z$ until it has two infinite cusp ends in place of the hypersurface.  Thus, in the limit $\eps\searrow 0$, we obtain a cusp metric $g_{\hc}$ on the disjoint union of two copies of $N$. 

In fact, while for $\eps>0$ the metrics $g_{\ehc}$ are smooth Riemannian metrics on $M,$
as $\eps\to0$ the metric degenerates along $Z$ and becomes a cusp metric on $M \setminus Z$.
In the limit, the de Rham operator $\eth_{\dR} = d+\delta$ associated to $g_{\ehc}$ has two model operators. On $M\setminus Z,$ we obtain $\eth_{\dR,\hc},$ the de Rham operator of the limiting cusp metric. The other model operator relates the two sides of $Z$ and is actually on $\bbR.$ It is the de Rham operator $D_b$ of a metric with cylindrical ends, but twisted by a weight and the `vertical cohomology bundle',
\begin{equation*}
	\mathrm H^*(Z;F)\lra \bbR,
\end{equation*}
see equation \eqref{eq:DefHOp} below for the precise definition of $D_b$. 
In \cite{ARS1} we carried out a careful analysis of the spectrum of $\eth_{\dR, \ehc}$ as $\eps \to 0$ by describing the precise asymptotics of the Schwartz kernels of the resolvent and heat kernel.
In particular we proved that there are finitely many eigenvalues of $\eth_{\dR,\ehc}$ that converge to zero as $\eps\to0.$ We call these the \emph{small eigenvalues} and denote the product of the non-zero small eigenvalues by $\det (\eth_{\dR})_{\sma}$ (and the square of this product by $\det (\eth_{\dR}^2)_{\sma}$).  If $\log \det (\eth_{\dR}^2)_{\sma}$ is polyhomogeneous in $\eps,$ the metric $g_{\ehc}$ is of `product-type' and the flat bundle is Witt in that
\begin{equation*}
	\mathrm H^{\dim Z/2}(Z;F) =0,
\end{equation*}
we show in \cite[Theorem 11.2]{ARS1} that 
the the determinant of the Laplacian satisfies
\begin{equation}\label{eq:DetAsymp}
	\FP_{\eps=0} \log\det \eth_{\dR,\ehc}^2
	= \log\det \eth_{\dR,\hc}^2
	+ \log\det D_b^2
	- \FP_{\eps=0} \log\det (\eth_{\dR}^2)_{\sma},
\end{equation}
where again the $b$-operator $D_b$ is defined in \eqref{eq:DefHOp} and the determinants of $D_b^2$ and the Hodge Laplacian $\eth_{\dR,\hc}^2$ of the metric $g_{\hc}$ are defined in terms of a renormalized trace of their respective heat kernels.

The strong acyclicity condition at infinity that  we imposed in the fibred cusp setting of \cite{ARS1}  greatly simplifies this formula, since then $D_b$ is trivial and there are no small eigenvalues, so that the last two terms in \eqref{eq:DetAsymp} do not contribute to the analytic torsion. To remove this condition in the cusp setting and compute the limit of analytic torsion as $\eps\to0$, we will establish that $\log \det (\eth_{\dR})_{\sma}$ is polyhomogeneous in $\eps$ and compute its finite part as $\eps\to0$ (Corollary \ref{poly.1}), and we will compute the determinant of the model operator $D_b^2$ (\S\ref{sec:ATDb}, especially \eqref{eq:FinalATDb}). Since analytic torsion should be thought of not as a number, but as a function that assigns a number to each basis of the cohomology $H^*(M;F),$  we will also compute the behavior of a basis of harmonic forms as $\eps\to0$ (\eqref{ff.16}).
These pieces together determine the limit of analytic torsion as $\eps\to0.$ 

On the topological side, we need to determine what happens to the Reidemeister torsion of $M$ in the limit.  In turns out it is not related to the Reidemeister torsion of the manifold with boundary $N$, but instead, the intersection $R$-torsion \cite{Dar} of the stratified space $\widehat{N}$ obtained from $N$ by collapsing $\pa N$ to a point.  Notice that for this quantity to be well-defined, we need to make the extra assumption that $F$ in fact descends to be a flat vector bundle on the the stratified space $\widehat{N}$.  
With this understood,  in Theorem \ref{rt.10}, we relate the Reidemeister torsion of $M$ with the intersection $R$-torsion
$I\tau^{\bm}(\widehat{N}, \mu_{N},F)$ associated to the upper middle perversity intersection cohomology $\mathrm{IH}^k_{\bar m}(\widehat{N};F)$ of  $\widehat{N}$ taking values in $F$.  Here, $\mu_{N}$ is a basis of orthonormal $L^2$-harmonic forms with respect to $g_F$ and $g_{\hc}$, which via the Hodge decomposition \cite[Corollary~9.4]{ARS1} induces a basis of  $\mathrm{IH}^*_{\bar m}(\widehat{N};F)$.    All together we establish the following theorem in Corollary \ref{ff.20}.

\begin{theorem}[A Cheeger-M\"uller theorem for manifolds with cusps] \label{thm:IntrohcCheegerMuller}
Let $(N,g_{\hc})$ be an odd dimensional Riemannian manifold with $g_{\hc}$ an even cusp metric.  Let $F\lra N$ be a flat Witt bundle with unimodular holonomy $\alpha: \pi_1(\widehat{N})\to \GL(k,\bbR)$ endowed with a bundle metric $g_F$ that extends smoothly to the double of $N$ across $\pa N.$
Let $\mu_{N}$ and $\mu_Z$ be bases of  $\mathrm{IH}^k_{\bar m}(\widehat{N};F)$ and $\tH^k(Z;F)$ respectively, consisting of $L^2$-harmonic forms orthonormal with respect to  $g_F$, $g_{\hc}$ and $g_Z$. The canonical identification \eqref{rt.5a} gives a basis $\mu_{\cC Z}$ for $\mathrm{IH}^k_{\bar m}(\cC Z).$  Using these bases to define the corresponding $R$-torsions, we have the following formula:
\begin{multline*}
	 \lAT(N, g_{\hc}, g_F, F) = 
	\log \lrpar{ \frac{I\tau^{\bm}(\widehat{N}, \mu_{N},F)\tau(Z;F)^{\frac12}}{I\tau^{\bm}(\cC Z,\mu_{\cC Z},F)} }  -\sum_{q>\frac{m-1}2} (-1)^q  \frac{\dim H^q(Z;F)}{4} \log 2\\
	- \sum_{\substack{0\leq q\leq m-1 \\ q\neq \frac{m-1}2}} (-1)^q \frac{\dim \tH^q(Z;F)}{4} |m-1-2q|\log|m-1-2q|.\end{multline*}
where $\cC Z= (Z\times [0,1])/(Z\times \{0\})$ is the cone over $Z$. 
\end{theorem}
\begin{remark}
If $F$ is in fact a flat Euclidean vector bundle, then by Poincar\'e duality the formula simplifies to
\begin{multline*}
	 \lAT(N, g_{\hc}, g_F, F)
	= \log \lrpar{ \frac{I\tau^{\bm}(N, \mu_{N},F)}{I\tau^{\bm}(\cC Z,\mu_{\cC Z},F)} } -\frac{\chi(Z;F)}{8}\log 2\\
	- \frac12 \sum_{q=0}^{\frac{m-1}2-1} (-1)^q \dim \tH^q(Z;F)(m-1-2q)\log(m-1-2q).
\end{multline*}
\end{remark}

Since we require that $F$ be defined on $\widehat{N}$ for the intersection $R$-torsion to be well-defined, notice that $F$ is automatically trivial on $Z$, so $F$ is Witt if and only if 
$\mathrm H^{\frac{m-1}2}(Z)=0$.  Thus, this results does not have implication on  hyperbolic $3$-manifolds with cusps.  However, in this case and more generally on odd dimensional hyberbolic manifolds with cusps with $Z$ a disjoint union of tori, notice that as pointed out in \cite[Remark~1.2]{ARS1}, if the holonomy representation of $F$ is orthogonal or more generally if it becomes a direct sum of irreducible representations when restricted to each connected component of $Z$, then by \cite[Remarks~3.5 (3) of Chapter VII]{Borel-Wallach}, the Witt condition is satisfied if and only if the strong acyclicity condition at infinity $H^*(Z;F)=0$ is satisfied, so that  \cite[Corollary 12.2]{ARS1} applies in this case.  \\

On hyperbolic manifolds with cusps, analytic torsion  was first studied by Park \cite{Park}, who proved that a relation discovered by Fried \cite{Fried1986} between analytic torsion and Ruelle zeta functions continues to hold on noncompact hyperbolic spaces. Combined with \cite[Corollary~12.2]{ARS1}, this gives a description of Ruelle zeta functions in terms of intersection $R$-torsion.  Recently there has been also an impressive sequence of papers by M\"uller and Pfaff \cite{Muller-Pfaff:ATComHypMfdsFinVol, Pfaff:ExpGrowthHomTorTowerCongSubgpsBianchi, Muller-Pfaff:ATAsympBhvSeqHypMfdsFinVol, Pfaff:SelbergZetaFunOddDHypMfdsFinVol, Pfaff:ATRTHyp3MFdsCusps, Pfaff:GluingFormATHypMfdsCusps}, see also \cite{Raimbault:Asymp, Raimbault:ARHtorsion},  in which the Selberg trace formula is used to great effect in analyzing analytic torsion. The methods in these papers are closely tied to the algebraic structure of locally symmetric spaces.

Calegari and Venkatesh \cite{Calegari-Venkatesh} study relationships between the torsion in the homology of arithmetic groups, for certain incommensurable groups, by a careful study of noncompact arithmetic three-manifolds with cusps. They define a Reidemeister torsion as the `regulator' of the homology groups divided by the size of torsion in the first homology groups, thus extending the R-torsion of compact arithmetic three-manifolds. They define the analytic torsion by a renormalized trace of the heat kernel.
In \cite[Theorem 6.8.3]{Calegari-Venkatesh} they prove a relative Cheeger-M\"uller theorem for trivial coefficients comparing a ratio of Reidemeister torsions to a ratio of analytic torsions, for two manifolds with isometric cusp structure. Their proof requires a careful study of small eigenvalues, which in their context refers to near-zero eigenvalues on a truncated hyperbolic manifold. We have heard from Venkatesh that one can in principle deduce a formula for the ratio of Reidemeister torsion and analytic torsion from the proof of their relative Cheeger-M\"uller theorem.

A theorem close to ours in the hyperbolic setting is an interesting Cheeger-M\"uller theorem due to Pfaff \cite{Pfaff:GluingFormATHypMfdsCusps}. 
This theorem applies to noncompact hyperbolic manifolds with cusps $N$ of odd dimension $m$ and flat vector bundles $F$ induced by the irreducible representations of $\mathrm{SO}^0(m,1)$ or $\mathrm{Spin}(m,1)$ that are not invariant with respect to the Cartan involution. Pfaff uses constructions of Harder \cite{Harder} to define a canonical Reidemeister torsion $\tau_{Eis}(\bar N;F)$ (similar to that used in \cite{Calegari-Venkatesh}). 
Let $\cC$ be a neighborhood of the cusps.
Pfaff uses the renormalized trace of Melrose to define analytic torsion and is then able to compute the difference 
\begin{equation*}
	\log \tau_{Eis}(\bar N;F) - \log\lrpar{\frac{AT(N;F)}{AT(\cC,\pa\cC;F)}}
\end{equation*}
in terms of the rank of $F,$ the Betti numbers and volume of $\pa C$ and some weights associated to the holonomy representation of $F.$
Notice that in this setting, the Witt condition is never satisfied. Moreover, the bundle metric used by Pfaff does not extend to a bundle metric on the double.

Finally we mention a preprint of Boris Vertman  \cite{Vertman:CheegerMuller}. 
In his paper,  Vertman investigates the Cheeger-M\"uller theorem for flat unitary bundles over odd-dimensional manifolds with product-type cusps satisfying the Witt condition. His approach, by gluing methods as in \cite{Lesch2013, Pfaff:GluingFormATHypMfdsCusps}, is completely different to ours.\\

For hyperbolic surfaces, cusp formation corresponds to converging to the boundary of Teichm\"uller space and so has been the subject of much study. For example, Seeley and Singer \cite{Seeley-Singer} studied the $\bar\partial$ operator as a cusp is formed. 
We can apply our analysis to study this situation as well. In Propositions \ref{lem:BurgerSmallEigen} and \ref{lem:WolpertDet} we recover results of Wolpert and Burger \cite{wol87,wol90,wol07,bur} on the asymptotics of small eigenvalues and the determinant.\\

The paper is organized as follows. Section \ref{sec:CuspMets} recalls our conventions for cusp metrics and analytic torsion. 
In \S\ref{sec:HorOp} we analyze the model operator $D_b$ on $\bbR$ and compute its contribution to the asymptotics of analytic torsion.
Section \ref{sec:Small} is devoted to the study of the small eigenvalues, including their polyhomogeneity in $\eps$, and culminates in the computation of the corresponding determinant. This section also includes an analysis of the asymptotics of an appropriately chosen basis of harmonic forms.
These results are collected in \S\ref{sec:AT} and yield the asymptotics of analytic torsion under degeneration to a manifold with cusp ends.

In section \ref{sec:RT} we show how the R-torsion of the closed manifold $M$ relates to the R-torsion of $N,$ the manifold with cusp ends. Finally, in \S\ref{sec:CheegerMuller}, we combine this study with our analysis of analytic torsion to obtain our Cheeger-M\"uller theorem. In the last section, \S\ref{wolpert}, we specialize to dimension two and explain the relevance of our results to families of hyperbolic metrics approaching the boundary of Teichm\"uller space.\\

{\bf Acknowledgements.}
P. A. was supported by NSF grant DMS-1104533 and Simons Foundation grant \#317883.
F. R. was supported by a Canada Research Chair, NSERC and FRQNT.
D. S. was supported by a CRM postdoctoral fellowship and by NSF EMSW21-RTG 1045119.  The authors are grateful to an anonymous referee for many helpful suggestions.
The authors are also happy to acknowledge useful conversations with Steven Boyer, Dan Burghelea, Nathan Dunfield, Rafe Mazzeo, Richard Melrose, Werner M\"uller and Jonathan Pfaff.
They are also happy to acknowledge the hospitality of the BIRS workshop ``Geometric Scattering Theory and Applications" held in November 2014.

\paperbody
\section{Cusp metrics and analytic torsion}\label{sec:CuspMets}

In this section, we recall the definition of cusp metrics and a very useful replacement for the tangent bundle that is adapted to the geometry. We also recall the definition of analytic torsion on closed manifolds and manifolds with ends  asymptotic to cusps.\\

\subsection{Analytic torsion}\label{subsec:AT}

On a closed Riemannian manifold $(M,g)$ of dimension $m$, the heat kernel of any Laplace-type operator satisfies
\begin{equation*}
	\Tr(e^{-t\Delta}) \sim t^{-m/2} \sum_{k\geq 0} a_k t^k \Mas t \to 0, \quad
	\Tr(e^{-t\Delta}) - \dim \ker \Delta = \cO(e^{-t\lambda_1}) \Mas t \to \infty,
\end{equation*}
with $\lambda_1>0.$ 
Hence its zeta function
\begin{equation*}
	\zeta(s) = \zeta(s;\Delta) = \frac1{\Gamma(s)}\int_0^\infty t^s \Tr(e^{-t\Delta} - \cP_{\ker \Delta}) \; \frac{dt}t
\end{equation*}
extends from a holomorphic function on $\Re s>m/2$ to a meromorphic function on all of $\bbC$ which has at worst simple poles and is regular at the origin.
If $F \lra M$ is a flat vector bundle endowed with a fiber metric $g_F$ (not necessarily compatible with the flat connection), and $\Delta_q$ is the Hodge Laplacian on $F$-valued differential forms of degree $q,$ then
\begin{equation*}
	\lAT(M,g,F,g_F) = \frac12\sum_q (-1)^q q \zeta'(0;\Delta_q)
\end{equation*}
is the logarithm of the analytic torsion of $(M,g,F,g_F).$

If $F$ is acyclic and its holonomy is unimodular, that is if $\tH^*(M;F) = 0,$ then the analytic torsion is independent of the choice of metrics $g,g_F.$ 
For a general flat bundle with unimodular holonomy, we choose a basis $\{ \mu_j^q\}$ of each $\tH^q(M;F)$ and let $\omega$ be an orthonormal basis of harmonic representatives with respect to the metrics $g, g_F;$ then we define
\begin{equation}\label{eq:IndepLAT}
	\bar{\lAT}(M,\{\mu_j^q\},F) = \lAT(M,g,F,g_F)  - \log \left(  \Pi_{q=0}^{n} [\mu^q |\omega^q]^{(-1)^q} \right),
\end{equation}
where $[\mu^q|\omega^q ] =|\det W^q|$ with $W^q$ the matrix such that $$\displaystyle \mu^q_i= \sum_j W^q_{ij} \omega^q_j.$$ It is this quantity that is independent of the choice of metrics.

\subsection{Cusp metrics}

Let $L$ be a smooth manifold with boundary $Z.$ Let $x$ be a smooth, non-negative function on $L$ that vanishes precisely on $Z$ and such that $dx$ does not vanish anywhere on $Z.$ We call such a function a `boundary defining function' for $Z,$ or `bdf' for short. We fix a choice of bdf, and our constructions will depend (mildly) on this choice.

Let us single out a subset of the vector fields on $L,$
\begin{equation*}
	\cV_{\phi}(L) = \lrbrac{ V \in \CI(L;TL) : V \text{ is tangent to $Z,$ and } Vx \in \cO(x^2) }
\end{equation*}
and point out that there is a vector bundle over $L$ whose space of sections is $\cV_{\phi}(L).$ 
We denote this bundle
\begin{equation*}
	{}^\phi TL \lra L
\end{equation*}
and refer to it as the `$\phi$-tangent bundle' of $L.$ 
(The $\phi$ more generally denotes a fibration on the boundary of $L;$ in our present context the fibration is $Z \fib Z \lra \pt.$)
The $\phi$-tangent bundle is isomorphic to the usual tangent bundle of $L,$ but not in a canonical way. 
The dual bundle 
\begin{equation*}
	{}^\phi T^*L \lra L
\end{equation*}
is called the `$\phi$-cotangent bundle' of $L.$
Note that $\frac{dx}{x^2}$ is a section of ${}^\phi T^*L$ that is non-degenerate at $Z = \{x=0\}.$

We can use $x$ to rescale the $\phi$-tangent bundle at $Z$ (see \cite[Chapter 8]{MelroseAPS}), and we refer to the bundle
\begin{equation*}
	\Hc TL = \frac1x {}^{\phi} TL
\end{equation*}
as the $\hc$-tangent bundle or `hyperbolic cusp tangent bundle'.  Its dual bundle 
\begin{equation*}
	\Hc T^*L \lra L
\end{equation*}
is the $\hc$-cotangent bundle of $L,$ and we point out that the one form $\frac{dx}x,$ as a section of $\Hc T^*L,$ is non-degenerate at $Z.$
Similarly if $z$ is a local coordinate on $Z$ then $x dz,$ as a local section of $\Hc T^*L,$ is {\em non-vanishing} at $x=0.$

An $\hc$-metric is a bundle metric on the $\hc$-tangent bundle. The simplest $\hc$-metrics are those that in some collar neighborhood of $Z$ of the form $[0,1]_x\times Z$ take the form
\begin{equation*}
	g_{\hc,pt} = \frac{dx^2}{x^2} + x^2 g_Z
\end{equation*}
with $g_Z$ a metric on $Z$ independent of $x.$ We refer to such metrics as {\bf product-type} $\hc$-metrics.
An $\hc$-metric $g_{\hc}$ is {\bf product-type to order $\ell$} if there is a product-type metric $g_{\hc,pt}$ such that
\begin{equation*}
	g_{\hc} - g_{\hc, pt} \in x^\ell \CI(L; S^2(\Hc T^*L))
\end{equation*}
where $S^2(\Hc T^*L)$ denotes the bundle of symmetric bilinear forms on $\Hc T^*L.$ In this paper our results will hold for $\hc$-metrics that are product-type to order $2.$\\

The heat kernel of a Laplace-type operator associated to an $\hc$-metric is not as well-behaved as the corresponding object on a closed manifold (\cite{v}, \cite[\S 7]{ARS1}).
First, the heat kernel is possibly not trace class. Fortunately it is well-behaved enough that we can make sense of its renormalized trace
\begin{equation*}
	\RTr{e^{-t\Delta}} = \FP_{z=0} \Tr(x^z e^{-t\Delta}).
\end{equation*}
Moreover, from \cite[\S~7]{ARS1} and the appendix of \cite{Albin-Rochon:ModSpace}, the asymptotics of the renormalized trace of the heat kernel are more complicated as $t \to 0:$
\begin{equation*}
	\RTr{e^{-t\Delta}} \sim t^{-m/2} \sum_{k\geq 0} a_{k/2} t^{k/2} + t^{-1/2} \sum_{k\geq 0} b_{k/2} t^{k/2} \log t.
\end{equation*}
Furthermore, one does not always have exponential convergence of $\RTr{e^{-t\Delta}}$ to $\dim \ker \Delta$ as $t\to\infty.$
We will deal with these differences by adding appropriate additional assumptions.

Let us say that {\bf a flat bundle $F$ is Witt }if, upon restricting to $Z,$ we have
\begin{equation*}
	\tH^{v/2}(Z;F) = 0
\end{equation*}
where $v = \dim Z = m-1.$
If $\Delta$ is a Hodge Laplacian associated to a Witt bundle, then it might have some continuous spectrum, but we know from \cite{ARS1} that  there is no continuous spectrum in a neighbourhood of zero.    Moreover, the heat kernel of $\Delta$ does not have to be trace class, in which case we can instead consider its renormalized trace.  This renormalized trace behaves very much like the usual trace.  In particular, it is shown in \cite[Proposition~7.3]{ARS1} that 
\begin{equation*}
	\RTr{e^{-t\Delta}} - \dim \ker \Delta = \cO(e^{-t\lambda_1}) \Mas t\to \infty \text{ for some }\lambda_1 >0.
\end{equation*}
If $g_{\hc}$ is product-type to order two and $m$ is odd, then $a_{m/2}=b_{1/2} =0.$ 
Again, if $\Delta$ is a Hodge Laplacian associated to a Witt bundle, the zeta function
\begin{equation*}
	\zeta(s;\Delta) = \frac1{\Gamma(s)}\int_0^{\infty} t^s \; \RTr{e^{-t\Delta} - \cP_{\ker \Delta}} \frac{dt}t
\end{equation*}
is a holomorphic function on $\Re s>m/2$ that extends to a meromorphic function on $\bbC,$ with at worst double poles, but regular at the origin. Thus for flat Witt bundles we may define analytic torsion for a cusp manifold just as for a closed manifold. If $m$ is even, there may be a pole at the origin, but we may still define analytic torsion by taking as a replacement for $\zeta'(0)$ the coefficient of $s$ in the Laurent series expansion of $\zeta(s)$ at the origin.

\subsection{Cusp degeneration}

We say that a closed Riemannian manifold $(M,g)$ with a two-sided hypersurface $Z$ is undergoing {\bf cusp degeneration }if the metric is degenerating from a smooth metric to a cusp metric on $M\setminus Z.$ We will carry out these degenerations in a controlled fashion by studying `cusp surgery metrics'.

Let us start by performing the `radial blow-up' of $Z \times \{ 0\}$ in $M \times [0,1]_{\eps}.$ Recall that this is a smooth manifold with corners,
\begin{equation*}
	X_s = [M\times [0,1]_{\eps}, Z \times \{ 0 \}],
\end{equation*}
obtained by replacing $Z \times \{ 0 \}$ with its inward pointing spherical normal bundle (see \cite{MelroseAPS}). 
Figure \ref{fig:singlespace} represents the space $X_s.$ There is a natural map, known as the blown-down map,
\begin{equation*}
	\beta:X_s \lra M \times [0,1]_\eps,
\end{equation*}
obtained by collapsing the new boundary hypersurface of $X_s$ back to $Z \times \{ 0\}.$
\begin{figure}
	\centering
	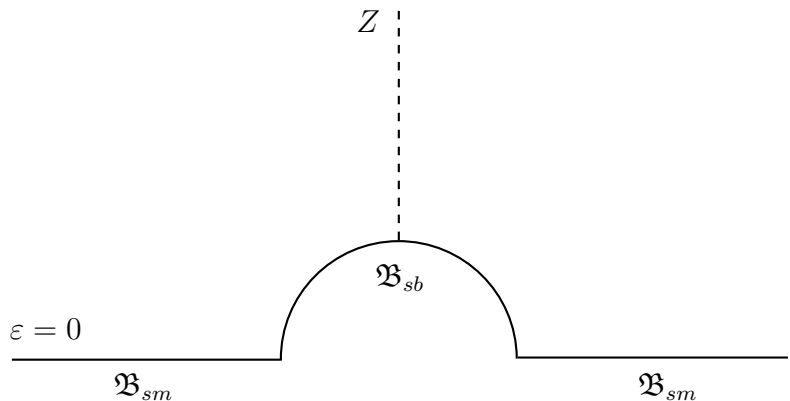
	\caption{The single surgery space $X_s.$}
	\label{fig:singlespace}
\end{figure}

The manifold $X_s$ has three boundary hypersurfaces. One, $\beta^{-1}(\{\eps=1\})$, will not be relevant to our studies and will be cheerfully ignored.
The other two are $\beta^{-1}(Z \times \{ 0 \}),$ known as the surgery boundary and denoted $\bhs{sb},$ and
\begin{equation*}
	\bhs{sm} = \bar{ \beta^{-1}( M \times \{ 0 \} \setminus Z \times \{ 0 \} ) },
\end{equation*}
where the $m$ in the subscript recalls that this is where most of $M\times \{ 0 \}$ ended up. 
Given any blow-down map, the `interior lift of a set' is equal to the closure of the lift of that set minus the set being blown-up; thus $\bhs{sm}$ is the interior lift of $M\times \{ 0\},$ which we denote
\begin{equation*}
	\bhs{sm} = \beta^{\sharp}(M \times \{ 0\}).
\end{equation*}

There is a natural choice of boundary defining function for $\bhs{sb},$ which we fix once and for all:
\begin{equation*}
	\rho_{sb} = \sqrt{ x^2 + \eps^2}.
\end{equation*}
When there is no possibility of confusion, we will denote this simply as $\rho.$

The interior of $\bhs{sb}$ can be identified with the normal bundle to $Z$ in $M$; $\bhs{sb}$ corresponds to its fiberwise compactification. The normal bundle to $Z$ is trivial by assumption, and so we have
\begin{equation*}
	\bhs{sb} \cong Z \times [-\pi/2, \pi/2].
\end{equation*}
(Of course any closed interval would serve, but our usual choice of coordinates will correspond to $[-\pi/2,\pi/2],$ so we use this interval throughout.)
We endow $\bhs{sb}$ with a trivial fibration
\begin{equation*}
	Z \fib \bhs{sb} \xlra{\phi_+} [-\pi/2,\pi/2].
\end{equation*}

Analogously to the $\hc$-tangent bundle, we will define a `cusp surgery tangent bundle' or $\ehc$-bundle.
First let $\pi_{\eps}:X_{s} \lra [0,1]_{\eps}$ be the composition of $\beta$ with the obvious projection and define
\begin{equation*}
	{}^\eps TX_s = \ker \pi_{\eps *} \subseteq TX_s.
\end{equation*}
Next let
\begin{equation*}
	\cV_{\ephi} = \{ V \in \CI(X_s; {}^{\eps} TX_s) : V\rest{\bhs{sb}} \text{ tangent to fibers of } \phi_+ \Mand V\rho \in \cO(\rho^2) \}
\end{equation*}
and define $\Ephi TX_s$ so that $\cV_{\ephi}$ is its space of sections. Finally, let
\begin{equation*}
	\Ehc TX_s = \frac1\rho \Ephi TX_s,
\end{equation*}
by which we mean that $\Ehc TX_s$ naturally isomorphic to $\Ephi TX_s$ away from $\bhs{sb}$, while near $\bhs{sb}$, if $\nu_1,\ldots, \nu_n\in \CI(X_s;\Ephi TX_s)$ is a local basis of sections smooth up then $\bhs{sb}$, then $\frac{\nu_1}{\rho}, \ldots, \frac{\nu_n}{\rho}$ are \emph{declared} smooth sections  $\Ehc TX_s$ up to $\bhs{sb}$, though of course, as sections of $\Ephi TX_s$, these blow up at $\bhs{sb}$.  Similarly,  we let $\Ehc T^*X_s$ denote the dual bundle, so that 
the one-forms
\begin{equation*}
	\frac{dx}{\rho}, \quad \rho \; dz,
\end{equation*}
where $z$ denotes a coordinate along $Z,$ lift from the interior of $X_s$ to a spanning set of sections of $\Ehc T^*X_s.$ Again, seen as sections of $\Ehc T^*X_s,$ these do not degenerate at $\bhs{sb}.$

A cusp surgery metric is a bundle metric on $\Ehc TX_s.$ We say that {\bf an $\ehc$-metric is of product type} if there is a tubular neighborhood $\mathrm{Tub}(Z) \cong [-1,1]_x \times Z \subseteq M$ around $Z$ in which the metric takes the form
\begin{equation*}
	g_{\ehc, pt} = \frac{dx^2}{x^2 + \eps^2} + (x^2 + \eps^2)g_Z
\end{equation*}
where $g_Z$ is a metric on $Z$ that is independent of both $x$ and $\eps.$
We say that {\bf an $\ehc$-metric $g_{\ehc}$ is of product type to order $\ell$} if 
\begin{equation*}
	g_{\ehc}-g_{\ehc,pt} \in \rho^\ell \CI(X_s; S^2(\Ehc T^*X_s))
\end{equation*}
for some product type metric $g_{\ehc, pt},$ where $S^2(\Ehc T^*X_s)$ denotes the bundle of symmetric two-tensors on $\Ehc T^*X_s.$\\

Let $F\lra X_s$ be a flat vector bundle endowed with a bundle metric $g_F,$ not necessarily compatible with the flat connection, and let 
\begin{equation*}
	\eth_{\dR} = d + \delta
\end{equation*}
be the corresponding de Rham operator. We will consider this as an operator on the bundle
\begin{equation*}
	E = \Lambda^* \Ehc T^*X_s \otimes F.
\end{equation*}
One of the advantages of using the $\ehc$-cotangent bundle, as opposed to the usual cotangent bundle of $X_s,$ is that the leading order behavior of $\eth_{\dR}$ will be described by tractable model operators, discussed below.
We are interested in the action of $\eth_{\dR}$ as an unbounded operator on $L^2_{\ehc}(M;E),$ the natural $L^2$ space associated to an $\ehc$-metric $g_{\ehc}$ and the bundle metric on $F.$ However, for some constructions it will be easier to work with
\begin{equation}
	L^2_{\eps,b}(M; E) = \rho^{v/2} L^2_{\ehc}(M;E),
\label{l2.1}\end{equation}
where $v = \dim Z = m-1.$ Thus our main object of interest is the operator
\begin{equation*}
	D_{\dR} = \rho^{v/2} \eth_{\dR} \rho^{-v/2}
\end{equation*}
acting as an unbounded operator on $L^2_{\eps,b}(M;E).$\\

If $g_{\ehc}$ is of product-type to order two, then we have simple expressions for the model operators of $D_{\dR}.$
First let us write
\begin{equation*}
	\Lambda^\ell \lrpar{ \Ehc T^* \mathrm{Tub}(Z) } \cong
	\rho^{\ell}\Lambda^\ell T^*Z \oplus 
	\frac{dx}{\rho} \wedge \lrpar{ \rho^{\ell-1} \Lambda^{\ell-1} T^*Z };
\end{equation*}
this splitting distinguishes between forms with a $dx$ and forms without a $dx$. With respect to this splitting, a direct computation tells us that $D_{\dR}$ is given near $\bs$ by
\begin{equation}
	D_{\dR} = 
	\begin{pmatrix}
	\tfrac1\rho \eth_{\dR}^Z & -\rho\pa_x + (\bN_Z - \tfrac12 v)\tfrac x\rho \\
	\rho\pa_x + (\bN_Z - \tfrac12 v)\tfrac x\rho & -\tfrac1\rho\eth_{\dR}^Z
	\end{pmatrix} + \text{higher order terms},
\label{moop.1}\end{equation}
up to higher order terms in $\rho$ as $\ehc$-differential operators. Here $\bN_Z$ is the number operator on $Z$ that multiplies a differential form by its degree.

The first model operator, known as the {\bf vertical operator}, is 
\begin{equation}\label{eq:DefVOp}
	D_v u= \rho D_{\dR}\widetilde{u}\rest{\bhs{sb}}= 
	\begin{pmatrix}
	\eth_{\dR}^Z & 0 \\
	0 & -\eth_{\dR}^Z
	\end{pmatrix}u,
\end{equation}
where $u$ is a section of $E$ on $\bhs{sb}$ and $\widetilde{u}$ is any smooth extension of $u$ to $X_s$.
Its null space forms a vector bundle over $\bhs{sb}$ which is just the space of scaled harmonic forms on $Z$, thought of as a trivial vector bundle over $[-\pi/2,\pi/2]$ and then pulled-back along $\phi_+.$ We will denote this bundle by
\begin{equation*}
	\rho^{\bN} \cH^*(Z;F) \lra \bhs{sb}.
\end{equation*}

The second model operator, known as the {\bf horizontal operator}, is defined by
$$
     D_b u= \Pi_h D_{\dR}\widetilde{u},
$$
where $\Pi_h$ denotes the projection onto $\ker D_v$, $u$ is a section of $\ker D_v$ and $\widetilde{u}$ is any choice of extension off $\bs$.  In terms of \eqref{moop.1}, the operator $D_b$ is given by

\begin{equation}\label{eq:DefHOp}
	D_b u=
	 \Pi_h 
	\begin{pmatrix}
	0 & -\rho\pa_x + (\bN_Z - \tfrac12 v)\tfrac x\rho \\
	\rho\pa_x + (\bN_Z - \tfrac12 v)\tfrac x\rho & 0
	\end{pmatrix} \widetilde{u},
\end{equation}
where $\Pi_h$ denotes the projection onto $Z$-harmonic forms.
In projective coordinates near $\bhs{sb},$
\begin{equation*}
	X = \frac x\eps, \quad z, \quad \eps, 
\end{equation*}
in which $\bhs{sb} = \{ \eps =0\},$ we let $\ang X=\sqrt{1+X^2}$ so that 
\begin{equation*}
	D_b = 
	\begin{pmatrix}
	0 & -\ang X\pa_X + (\bN_Z - \tfrac12 v)\tfrac X{\ang X} \\
	\ang X\pa_X + (\bN_Z - \tfrac12 v)\tfrac X{\ang X} & 0
	\end{pmatrix},
\end{equation*}
as an operator acting on $\CI(\bbR_X; \rho^{\bN_Z} \cH^*(Z;F) \oplus \frac{dX}{\ang X} \wedge \rho^{\bN_Z} \cH^*(Z;F)),$ where the restriction of $\rho^{q} \cH^q(Z;F)$ to $\bs$ is well defined as a section of $\Lambda^q(\Ehc T^*X_s)$.  Thus, in the sense of Melrose \cite{MelroseAPS}, $D_b$ is a $b$-operator on the radial compactification of $\bbR$. If $F$ is a Witt bundle, then $D_b$ is Fredholm \cite[Lemma 2.1]{ARS1}.  By analogy with the de Rham operator, notice also that $D_b$ is naturally the sum of two $b$-operators,  namely $D_b= d_b+\delta_b$ with 
\begin{equation} 
\begin{aligned}
d_b&:=\begin{pmatrix}
	0 & 0 \\
	\ang X\pa_X + (\bN_Z - \tfrac12 v)\tfrac X{\ang X} & 0
	\end{pmatrix},  \\
\delta_b&:=\begin{pmatrix}
	0 & -\ang X\pa_X + (\bN_Z - \tfrac12 v)\tfrac X{\ang X} \\
	0 & 0
	\end{pmatrix}.
\end{aligned}
\label{dbdb.1}\end{equation}
as operators acting on $\CI(\bbR_X; \rho^{\bN_Z} \cH^*(Z;F) \oplus \frac{dX}{\ang X} \wedge \rho^{\bN_Z} \cH^*(Z;F)).$

Finally, $D_{\dR}$ induces an operator on $\bhs{sm}.$ This face is the manifold with boundary $M_0 = [M;Z]$ and 
\begin{equation}
	D_d = D_{\dR}\rest{\bhs{sm}}
\label{dd.1}\end{equation}
is the twisted de Rham operator corresponding to the $\hc$-metric $g_0 = g_{\ehc}\rest{\bhs{sm}}$ and the flat bundle $F\rest{\bhs{sm}}.$

\section{Analysis of the model operator} \label{sec:HorOp}

Let $M$ be a closed manifold of dimension $m,$ $Z$ a two-sided hypersurface with fixed boundary defining function $x$ and $g_{\ehc}$ a cusp surgery metric, product-type to second order. If $F\lra X_s$ is a flat vector bundle of Witt type, then we have seen that there is a model $b$-operator
\begin{equation*}
	D_b = 
	\begin{pmatrix}
	0 & -\ang X\pa_X + (\bN_Z - \tfrac12 v)\tfrac X{\ang X} \\
	\ang X\pa_X + (\bN_Z - \tfrac12 v)\tfrac X{\ang X} & 0
	\end{pmatrix},\quad
	\ang X = \sqrt{1+X^2}
\end{equation*}
acting on $\CI(\bbR_X; \rho^{\bN} \cH^*(Z;F) \oplus \frac{dX}{\ang X} \wedge \rho^{\bN} \cH^*(Z;F)).$  In this section we study this operator and its contribution to the asymptotics of analytic torsion.\\

\subsection{Null space of the horizontal operator}\label{sec:NullDb}
First let us compute its null space, and that of its square.
Note that if
\begin{equation*}
	\begin{pmatrix}
	0 & -\ang X\pa_X + (\bN_Z - \tfrac12 v)\tfrac X{\ang X} \\
	\ang X\pa_X + (\bN_Z - \tfrac12 v)\tfrac X{\ang X} & 0
	\end{pmatrix}
	\begin{pmatrix}
	f(X) \\ h(X)
	\end{pmatrix}
	=
	\begin{pmatrix}
	0 \\ 0
	\end{pmatrix}
\end{equation*}
then the projections of $f$ and $h$ onto the spaces of forms of fixed \emph{vertical} degree $k$ (that is, having degree $k$ in $Z$), which we denote $f_k$ and $h_k$ respectively, are also in the null space of $D_b.$
More generally, for $a\in \bbR$, consider the operator
\begin{equation*}
	P(a) 
	:= \ang X \pa_X + a \tfrac X{\ang X}
	= \ang X^{-a} ( \ang X \pa_X ) \ang X^{a}. 
\end{equation*}
Taking $a=k-\frac{v}2$,  we see that
\begin{equation*}
	\ang X^{-(k-v/2)}  (\ang X \pa_X ( \ang X^{(k-v/2)}f_k(X))) =0,
	\implies f_k(X) = C \ang X^{v/2-k}
\end{equation*}
while taking $a=\frac{v}2-k$ yields
\begin{equation*}
	-\ang X^{(k-v/2)}  (\ang X \pa_X ( \ang X^{-(k-v/2)}h_k(X))) =0
	\implies h_k(X) = C \ang X^{k-v/2}.
\end{equation*}
Thus we have found that 
\begin{equation*}
	\ker D_b = \mathrm{span}\lrbrac{
	\begin{pmatrix} u \ang X^{v/2-k} \\ v \ang X^{k-v/2} \end{pmatrix} :
	u, v \in \rho^k\cH^k(Z;F), k \in \bbN_0 }.
\end{equation*}
We are interested in $D_b$ as an unbounded operator on $L^2_b,$ i.e., with the measure $\frac{dX}{\ang{X}}$ on $\bbR.$
With respect to this measure, $\ang X^a$ is in $L^2$ iff $a<0,$ and hence the $L^2$ kernel of $D_b$ is
\begin{multline}\label{eq:DbNullSpace}
	\ker_{L^2} D_b = \\
	\mathrm{span}\lrbrac{
	\begin{pmatrix} u \ang X^{v/2-k} \\ v \ang X^{k-v/2} \end{pmatrix} :
	u \in \rho^k\cH^k(Z;F), k >v/2,  \quad
	v \in \rho^k\cH^k(Z;F), k <v/2  }.
\end{multline}

Next consider
\begin{equation*}
	D_b^2 = 
	\begin{pmatrix}
	-P(\tfrac12v - \bN_Z)P(\bN_Z - \tfrac12v) & 0 \\
	0 & -P(\bN_Z - \tfrac12v)P(\tfrac12v - \bN_Z)
	\end{pmatrix}
\end{equation*}
and note that, for $f \in \CI(\bbR),$ 
\begin{multline*}
	P(-a)P(a) f = 0 \implies P(a) f = C \ang{ X}^{a}
	\implies f =  C \ell_{a}(X) + C' \ang {X}^{-a}, \\
	 \Mwith \ell_a(X) = \ang X^{-a}\int_{0}^X \ang s^{2a-1} \; ds.
\end{multline*}
Notice that as $X\to\pm\infty$, $|\ell_a(X)|$ is of order $|X|^{|a|}$ for $a\ne 0$, while for $a=0$,  $\ell_0(X)= \sinh^{-1}(X)$, so $\ell_a$ is never in $L^2$ with respect to the density $\frac{dX}{\langle X\rangle}.$   
Hence the null space of $D_b^2$ acting on smooth sections of $\ang X^{\bN}\cH^*(Z)$ is 
\begin{equation}\label{eq:SmoothKerDb2}
	\ker D_b^2 = 
	\mathrm{span}\lrbrac{
	\begin{pmatrix} u \ang X^{v/2-k} + u' \ell_{k-v/2}(X) \\ v \ang X^{k-v/2} + v' \ell_{v/2-k}(X) \end{pmatrix} :
	u, u', v, v' \in \ang X^k\cH^k(Z;F), k \in \bbN_0 }.
\end{equation}
Now, since the function $\ell_a(X)$ is never in $L^2_b$, we see that when $a<0$,  the only way that the linear combination $c_1 \ang X^{a} + c_2 \ell_{a}(X)$ be in $L^2_b$ is if $c_2=0$.  If instead $a\ge0$, then one can choose the constants $c_i$ in the linear combination $c_1 \ang X^{a} + c_2 \ell_{a}(X)$ in such a way that it is 
$o(|X|^a)$ in one end of $\bbR$.  However, since $\ell_a$ is an odd function and $\ang X^a$ is an even function, unless $c_1=c_2=0$, there is at least one end where the norm of the linear combination grows like $c|X|^a$ for some positive constant. In other words, when $a\ge 0$, the only way that $c_1 \ang X^{a} + c_2 \ell_{a}(X)$ be in $L^2_b$ is if $c_1=c_2=0$.  This implies that
\begin{equation*}
	\ker_{L^2} D_b^2 = \ker_{L^2} D_b,
\end{equation*}
which could also have been deduced from the formal self-adjointness of $D_b$.  
Let us emphasize in particular that $D_b^2$ has no $L^2$-kernel on forms of {\em total} degree (i.e., degree in $\frac{dX}{\ang{X}}$ plus degree in $Z$) equal to zero or $m.$

\subsection{Analytic torsion contribution of the horizontal operator} \label{sec:ATDb}

Let 
\begin{equation*}
	(D_b^2)_{j,k} = D_b^2\rest{\Lambda^j\bbR_{X} \wedge \ang X^k \cH^k(Z;F)}, \quad j \in \{ 0 ,1\}, k \in \{ 0, \ldots, v \}
\end{equation*}
where $v=\dim Z = m-1$.
Each of these is a Laplace-type operator on $\bbR_X$ and we denote the corresponding zeta function by $\zeta_{j,k}(s).$
From \cite[Theorem 11.2]{ARS1}, we know that the contribution of the horizontal operator $D_b$ to the asymptotics of analytic torsion is through
\begin{equation}\label{eq:OverallContrib}
	\frac12\sum (-1)^{j+k}(j+k)\zeta'_{j,k}(0).
\end{equation}
In this subsection we will compute this contribution.\\

From the previous subsection we see that the heat kernel of $D_b$ satisfies
\begin{equation*}
	e^{-tD_b^2}\rest{\rho^q\cH^q(Z;F) \oplus \frac{dX}{\ang X} \wedge \rho^{q-1}\cH^{q-1}(Z;F)} 
	= 
	\begin{pmatrix}
	e^{t P(\frac {v-2q}2)P(\frac{2q-v}2)} & 0 \\
	0 & e^{t P( \frac{2(q-1)-v}2)P( \frac{v-2(q-1)}2 ) }
	\end{pmatrix}.
\end{equation*}
We introduce the abbreviation
\begin{equation*}
	F^{\bbR}_a = {}^R\Tr(e^{t P(a)P(-a)}) 
\end{equation*}
and note that
\begin{equation*}
	{}^R\Tr(e^{-tD_b^2}\rest{\rho^q\cH^q(Z;F) \oplus \frac{dX}{\ang X} \wedge \rho^{q-1}\cH^{q-1}(Z;F)} )
	= b_qF^{\bbR}_{(v-2q)/2} + b_{q-1}F^{\bbR}_{(2(q-1)-v)/2},
\end{equation*}
where $b_q = \dim \cH^q(Z;F)$ with the convention that $b_{-1}=b_{v+1}=0.$
So we can write 
\begin{multline}\label{eq:DhTrContribution}
	\sum (-1)^{j+k}(j+k) \RTr{e^{-t(D_b^2)_{j,k}}}
	= \sum_{q=0}^{v+1} (-1)^{q} q \lrpar{ b_q F^{\bbR}_{(v-2q)/2} + b_{q-1}F^{\bbR}_{(2(q-1)-v)/2} } \\
	= \sum_{q=0}^{v} (-1)^qb_q \lrpar{ qF^{\bbR}_{(v-2q)/2} - (q+1)F^{\bbR}_{(2q-v)/2} }
\end{multline}
and hence \eqref{eq:OverallContrib} is equal to 
\begin{equation}\label{eq:SimplifiedContrib}
	\frac12\sum_{q=0}^v (-1)^{q}b_q
	\left[ q \lrpar{ -\log\det -P(\tfrac{v-2q}2)P(\tfrac{2q-v}2)} -(q+1)
	\lrpar{  -\log\det -P(\tfrac{2q-v}2)P(\tfrac{v-2q}2)} \right].
\end{equation}
%
%
It thus suffices to compute the determinant of $-P(-a)P(a)$ on $\bbR$ (endowed with the metric $dX^2/\ang X^2$ and bdf $\varrho = \ang X^{-1}$),
\begin{equation*}
	\det \left(-P(-a)P(a)\right) = e^{-\zeta_{-P(-a)P(a)}'(0)}.
\end{equation*}
Our strategy will be to compute the variation in $a$ of the renormalized trace (see \eqref{eq:UltimateVariation}) and use this to compute the determinant (see \eqref{eq:LogDet1D}). Once we have computed these one-dimensional determinants, we return to \eqref{eq:SimplifiedContrib} in \eqref{eq:FinalATDb}.\\

Let us start with the two cases we can compute directly.
\begin{lemma}\label{lem:2cases}
When $a=0,$ we have
\begin{equation*}
\begin{gathered}
	e^{tP(0)^2}(X,X') = \frac1{\sqrt{4\pi t}} \exp\lrpar{-\frac{| \sinh^{-1}(X)- \sinh^{-1}(X')|^2}{4t}},\\
	\RTr{e^{tP(0)^2}} = \frac{\log 2}{\sqrt{\pi t}}, \quad
	\zeta_{-P(0)^2}(s) = 0, \quad 
	\log\det -P(0)^2 = 0.
\end{gathered}
\end{equation*}
When $a=-1,$ we have
\begin{equation*}
\begin{gathered}
	e^{tP(1)P(-1)}(X,X') = e^{-t}e^{tP(0)^2}(X,X'), \quad
	\RTr{e^{tP(1)P(-1)}} = \frac{e^{-t}\log 2}{\sqrt{\pi t}}, \\
	\int_0^{\infty} t^s \; \RTr{e^{tP(1)P(-1)}} \; \frac{dt}t = \frac{\Gamma(s-1/2)\log 2}{\sqrt\pi}, \quad
	\log\det\left( -P(1)P(-1)\right) = 2\log 2.
\end{gathered}
\end{equation*}
\end{lemma}
\begin{proof}
In the coordinate $u = \sinh^{-1}(X)$,
\begin{equation*}
	P(a) = \pa_u + a \tanh u, \quad -P(-a)P(a) = -\pa_u^2 +a^2-(a^2+a)\mathrm{sech}^2 u.
\end{equation*}
Hence $-P(0)^2$ is the Euclidean Laplacian and $-P(1)P(-1)$ is the Euclidean Laplacian plus one.
In the former case the restriction of the heat kernel to the diagonal is $(4 \pi t)^{-1/2}$ and in the latter $e^{-t}(4\pi t)^{-1/2};$ as these are independent of $u,$ the renormalized trace is computed by multiplying these expressions by the renormalized volume.
The renormalized volume is, for our choice of measure $\frac{dX}{\ang{X}}$ and of boundary defining function $\varrho =\ang{X}^{-1},$ 
\begin{multline*}
	\Rint_{\bbR} \frac{dX}{\ang X}
	=2 \; \Rint_{\bbR^+} \frac{dX}{\ang X}
	=2 \; \sideset{^R}{_0^1}\int \frac{d\varrho}{\varrho\sqrt{1-\varrho^2}}
	= 2 \FP_{\eps = 0} \int_{\eps}^1 \frac{d\varrho}{\varrho\sqrt{1-\varrho^2}} \\
	= 2 \FP_{\eps=0} \lrpar{ \log ( \sqrt{1-\eps^2} +1 ) - \log \eps }
	= 2 \FP_{\eps=0} \lrpar{ \log 2 - \log \eps + \cO(\eps^2) }
	= 2 \log 2.
\end{multline*}
This proves that $\RTr{ e^{tP(0)^2} } = \frac{\log 2}{\sqrt{\pi t}}$. It is easy to see that the renormalized Mellin transform over $\bbR^+_t$ of a power of $t$ is equal to zero; indeed, let us define, for any function $f(t)$ with an asymptotic expansion in $t$ as $t \to 0$ and $t^{-1}$ as $t \to \infty:$
\begin{equation}\label{eq:DefM0infty}
	\cM_0(f, s) = \int_0^1 t^s f(t) \frac {dt}t, \quad
	\cM_\infty(f, s) = \int_1^{\infty} t^s f(t) \frac {dt}t.
\end{equation}
Each of these extends to a meromorphic function on $\bbC,$ which we denote by the same symbol, and the renormalized Mellin transform of $f$ is
\begin{equation}
	\cM(f,s) = \cM_0(f,s) + \cM_{\infty}(f,s).
\label{mero.1}\end{equation}
If $f(t) = t^\nu$, then $\cM_0(f,s) = \frac1{s+\nu}$ and $\cM_{\infty}(f,s) = -\frac1{s+\nu},$ and so $\cM(f,s)=0.$
Hence we see that the zeta function of $-P(0)^2$ is identically zero.

For $a=-1$ we have, for any $s>\frac12,$
\begin{equation}
	\int_0^{\infty} t^s \Tr(e^{tP(1)P(-1)}) \frac{dt}t
	= \frac{\log 2}{\sqrt \pi} \int_0^{\infty} t^{s-1/2} e^{-t} \frac{dt}t
	= \frac{\Gamma(s-1/2)\log 2}{\sqrt \pi}
	\sim -2\log 2 + \cO(s);
\label{mero.2}\end{equation}
hence
\begin{equation*}
	\log \det -P(1)P(-1) 
	= -\frac{\pa}{\pa s}\zeta(s) \rest{s=0}
	= -\frac{\pa}{\pa s} \lrpar{
	\frac{\Gamma(s-1/2)\log 2}{\Gamma(s)\sqrt \pi} } \rest{s=0}
	= 2\log 2.
\end{equation*}

\end{proof}

Next let us compute the variation of the renormalized trace for arbitrary $a.$
First note that
\begin{equation*}
	\pa_a P(a)
	= \pa_a \lrpar{ \ang{X}^{-a}(\ang X\pa_X)\ang X^a } 
	= - \log \ang X P(a) + P(a) \log \ang X = [P(a), \alpha]
\end{equation*}
where $\alpha$ denotes $\log \ang X.$ Similarly
\begin{equation*}
	\pa_a P(-a) 
	= [\alpha, P(-a)], \quad
	\pa_a (-P(-a)P(a)) 
	= -\alpha P(-a)P(a) + 2P(-a)\alpha P(a) - P(-a)P(a)\alpha.
\end{equation*}
Next, by Duhamel's formula, we have
\begin{equation*}
	\frac{\pa}{\pa a} {}^R\!\Tr(e^{tP(-a)P(a)}) = 
	-{}^R\!\Tr\lrpar{ \int_0^t e^{\tau P(-a)P(a)}\pa_a(-P(-a)P(a))e^{(t-\tau)P(-a)P(a)} } \; d\tau
	= T_1+R_1
\end{equation*}
where
\begin{equation*}
\begin{gathered}
	T_1 = -t \,{}^R\!\Tr(e^{tP(-a)P(a)}\pa_a(-P(-a)P(a))),\\
	R_1 = 
	-\int_0^t {}^R\!\Tr\lrspar{  e^{\tau P(-a)P(a)}\pa_a(-P(-a)P(a)),e^{(t-\tau)P(-a)P(a)} } \; d\tau.
\end{gathered}
\end{equation*}
Note that since the \emph{renormalized} trace does not vanish on commutators, $R_1$ is not automatically zero.

Let us focus first on $T_1.$ We can rewrite it as
\begin{multline*}
	T_1 
	= -t \; {}^R\!\Tr(e^{tP(-a)P(a)} (-\alpha P(-a)P(a) + 2P(-a)\alpha P(a) - P(-a)P(a)\alpha) ) \\
	= -t \; {}^R\!\Tr \left( -P(-a)P(a)e^{tP(-a)P(a)}\alpha +2 P(a)e^{tP(-a)P(a)}P(-a)\alpha - e^{tP(-a)P(a)}P(-a)P(a)\alpha \right.\\ \left.
	+\lrspar{ P(-a)P(a), e^{tP(-a)P(a)} \alpha} - 2 \lrspar{P(a), e^{tP(-a)P(a)}P(-a)\alpha} \right).
\end{multline*}
In turn let us write this as $T_2 + R_2$ where $R_2$ consists of the summands involving commutators.
From the uniqueness of the solution to the heat equation, we note that $P(-a)e^{tP(a)P(-a)}=e^{tP(-a)P(a)}P(-a),$ and so we can write $T_2$ as
\begin{equation*}
\begin{gathered}
	T_2 = -2t \; {}^R\! \Tr( P(a)P(-a)e^{tP(a)P(-a)}\alpha) + 2t\;  {}^R\! \Tr( P(-a)P(a) e^{tP(-a)P(a)}\alpha) \\
	= 2t\pa_t \lrpar{ {}^R\! \Tr(e^{tP(-a)P(a)}\alpha) - {}^R\!\Tr( e^{tP(a)P(-a)}\alpha )}
\end{gathered}
\end{equation*}
which we can rewrite  as $2t\pa_t {}^R\!\Str(e^{t\hat{P}^2(a)}\alpha),$ where
$$
    \hat{P}(a) =\left( \begin{array}{cc} 0 & P(-a) \\ P(a) & 0  \end{array} \right) \quad \mbox{and} \quad  {}^R\!\Str \left( \begin{array}{cc} A & B \\ C & D  \end{array} \right) =  {}^R\! \Tr(A)- {}^R\! \Tr(D). 
$$

Thus we have
\begin{equation}\label{firstexpression}
	\frac{\pa}{\pa a} {}^R\!\Tr(e^{tP(a)P(-a)}) = 2t\pa_t {}^R\!\Str(e^{t\hat{P}^2(a)}\alpha) + R_1 + R_2
\end{equation}
where $R_1$ and $R_2$ involve renormalized traces of commutators. \\

To address these terms we will make use of appropriate trace defect formul\ae{}.
We will use the same conventions as in \cite{MelroseAPS} regarding Mellin transform and indicial operators, namely:
\begin{equation*}
\begin{gathered}
	\cM(u)(\lambda) = \int_0^{\infty} \varrho^{-i\lambda}u(\varrho)\; \frac{d\varrho}{\varrho}, \quad
	(\cM^{-1}v)(\varrho) = \frac1{2\pi}\int_{\Im \lambda = \eta} \varrho^{i\lambda} v(\lambda) \; d\lambda \\
	I(A, \lambda) = \lrspar{ \varrho^{-i\lambda}A\varrho^{i\lambda} }_{\pa}.
\end{gathered}
\end{equation*}
The latter means that $I(A,\lambda)$ acts on a section $u$ over the boundary by choosing an extension $\wt u$ off of the boundary, applying $x^{-i\lambda}Ax^{i\lambda}$ to $\wt u,$ and then restricting back to the boundary.  The result is independent of the choice of extension.  Recall moreover that the operation of taking the indicial operator is an homomorphism in the sense that 
$$
      I(AB,\lambda)= I(A,\lambda) I(B,\lambda).         
$$

\begin{lemma}\label{lem:renormcommut}
On any manifold with boundary $M$ with a fixed choice of bdf $\varrho,$
\begin{itemize}
\item [a)] \cite[Lemma 5.10]{MelroseAPS} If $A$ is a $b$-pseudodifferential operator and $B$ is a smoothing $b$-pseudodifferential operator, then
\begin{equation*}
	{}^R\! \Tr([A,B]) 
	= \frac i{2\pi}\int_{\bbR} \Tr_{\pa}\lrpar{ \pa_{\lambda}I(A,\lambda) I(B,\lambda) }\; d\lambda,
\end{equation*}
where the trace in the integrand is the trace $\Tr_{\pa}:\Psi^{-\infty}(\pa M)\lra \bbC;$
\item [b)] With $A$ and $B$ as above,
\begin{equation*}
	{}^R\!\Tr([A,B\log \varrho]) 
	= -\frac1{4\pi} \int_{\bbR} \Tr_{\pa} ( \pa_{\lambda}^2I(A,\lambda)I(B,\lambda)) \; d\lambda.
\end{equation*}
%
%
%
\end{itemize}
\end{lemma}

\begin{proof}
Following \cite{Melrose-Nistor} (cf. \cite{Melrose-Rochon:Cusp, Albin:RenInt}), let us use Riesz renormalization to define the renormalized trace,
\begin{equation*}
	{}^R\! \Tr(B) = \FP_{z=0} \Tr(\varrho^z B)
\end{equation*}
for any operator $B$ such that $\varrho^z B$ is trace-class for large enough $\Re(z).$
We have
\begin{equation}
	{}^R\!\Tr([A,B]) 
	= \FP_{z=0} \Tr(\varrho^z[A,B])
	= \FP_{z=0} \Tr([\varrho^z,A]B)
	= \FP_{z=0} \Tr(z\varrho^z \wt A(z) B)
	= \Res_{z=0} \Tr(\varrho^z \wt A(z) B)
\label{rtr.1}\end{equation}
where $\wt A(z) = \frac{A - \varrho^{-z} A \varrho^{z}}z.$
Note that this is a holomorphic function of $z$ and has indicial operator
\begin{equation*}
	I(\wt A(z),\lambda) = \frac{ I(A,\lambda) - I(A, \lambda + \tfrac1iz)}z
\end{equation*}
with Taylor expansion at $z=0$ given by
\begin{equation*}
	-\sum_{k \geq 1} \frac{ z^{k-1} }{k! i^k} \pa_{\lambda}^k I(A,\lambda).
\end{equation*}
In particular, we point out that
\begin{equation}\label{eq:IndicialDerivative}
	I(\wt A(0),\lambda)= i\pa_{\lambda}I(A,\lambda) = - I([A,\log\varrho],\lambda).
\end{equation}
To compute the residue $\Res_{z=0} \Tr(\varrho^z \wt A(z) B)$ note that
\begin{equation*}
	\int_0^\delta \varrho^z (\varrho^k) \; \frac{d\varrho}{\varrho} = \frac{\delta^{z+k}}{z+k}
\end{equation*}
so only the term with $k=0$ contributes to the residue at $z=0.$ Hence
$\Tr(\varrho^z \wt A(z) B)$ extends to a meromorphic function with a simple pole at $z=0$ and residue equal to 
\begin{equation*}
	\Tr_{\pa}(\wt A(0)B\rest{\varrho=0}) = 
	\frac1{2\pi} \int_{\bbR} \Tr_{\pa}( I( \wt A(0)B,\lambda) ) \; d\lambda
	= \frac i{2\pi} \int_{\bbR} \Tr_{\pa}( \pa_{\lambda} I(A,\lambda) I(B,\lambda))\; d\lambda
\end{equation*}
as required.

Now replace $B$ with $B\log\varrho$. Proceeding as in \eqref{rtr.1}, we see that 
\begin{equation*}
	{}^R\!\Tr([A,B\log \varrho]) 
	= \Res_{z=0} \Tr(\varrho^z \wt A(z) B\log\varrho).
\end{equation*}
To compute this residue, note that 
\begin{equation*}
	\int_0^\delta \varrho^z (\varrho^k \log \varrho) \; \frac{d\varrho}\varrho =
	 \frac{ \delta^{z+k} \log \delta}{z+k} 
	- \frac{\delta^{(z+k)}}{(z+k)^2},
\end{equation*}
so only the terms with $k=0$  in the expansion of $\wt A(z) B\log\varrho$, , that is, of order  $\log \varrho$, contribute to the residue. Furthermore, for small $z$ we have $\delta^z = 1+ z\log\delta + \cO(z^2),$ and so
\begin{equation*}
	 \frac{ \delta^{z} \log \delta}{z} 
	- \frac{\delta^{z}}{z^2}
	\sim \frac{\log\delta}z - \frac{1+z\log\delta}{z^2} + \cO(1) = -\frac1{z^2}+\cO(1).
\end{equation*}
Thus, in terms of the expansion 
$$
\wt A(z)B\log\varrho= \sum_{j=0} z_j \wt A_j B\log\varrho,
$$
this means that we only need (minus one times) deal with the term of order $\log\varrho$ in $z \wt A_1 B\log\varrho$ to compute the residue,
$$
\Res_{z=0} \Tr(\varrho^z \wt A(z) B\log \varrho)=- \Tr_{\pa} (\wt A_1B\rest{\varrho=0}).
$$   
Since the term of order $z$ in the expansion of $I(\wt A(z),\lambda)$ is $\frac z2\pa_{\lambda}^2I(A,\lambda),$ we finally find that 
\begin{equation*}
	\Res_{z=0} \Tr(\varrho^z \wt A(z) B\log \varrho)=- \Tr_{\pa} ( \wt A_1B\rest{\varrho=0})
	= -\frac1{4\pi} \int_{\bbR} \Tr_{\pa} ( \pa_{\lambda}^2I(A,\lambda)I(B,\lambda)) \; d\lambda.
\end{equation*}
%

%
%
%
%
%
%
%
%
\end{proof}

To compute the indicial operator of $P(a)$ let us recall that our bdf is $\varrho = \ang X^{-1},$ so that 
\begin{equation*}
	X = \sign(X)\sqrt{\varrho^{-2}-1}, \quad
	\ang X \pa_X = -\sign(X)\sqrt{1-\varrho^2}\varrho\pa_\varrho 
\end{equation*}
and hence
\begin{equation}
\begin{gathered}
	P(a) = \varrho^a \lrpar{  -\sign(X)\sqrt{1-\varrho^2}\varrho\pa_\varrho }\varrho^{-a}
	= -\sign(X) \sqrt{1-\varrho^2}(\varrho\pa_{\varrho} -a), \\
	I(P(a),\lambda) = 
	\begin{cases}
	-i\lambda+a & \Mat X \to +\infty, \\
	i\lambda -a & \Mat X \to -\infty.
	\end{cases}
\end{gathered}
\label{forpm.1}\end{equation}
It follows that $I(-P(-a)P(a), \lambda) = \lambda^2+a^2$ and $I(e^{tP(-a)P(a)},\lambda) = e^{-t(\lambda^2+a^2)}$ at both ends of $\bbR.$ 

We can use this observation and Lemma \ref{lem:renormcommut} to compute $R_1$ and $R_2.$ Indeed, note from \eqref{eq:IndicialDerivative} and the fact that $\alpha=-\log\varrho$ that
\begin{equation*}
\begin{gathered}
	I(\pa_a(-P(-a)P(a)), \lambda) 
	=I(-[P(-a),\log \varrho]P(a) + P(-a)[P(a),\log\varrho]), \lambda) \\
	=-(\tfrac1i \pa_{\lambda}I(P(-a),\lambda))I(P(a),\lambda) 
	+P(-a)(\tfrac1i\pa_{\lambda}I(P(a),\lambda)) = 2a
\end{gathered}
\end{equation*}
(the same at both ends of $\bbR$). Hence
\begin{equation*}
\begin{gathered}
	R_1 = 
	-\int_0^t {}^R\!\Tr\lrspar{  e^{\tau P(-a)P(a)}\pa_a(-P(-a)P(a)),e^{(t-\tau)P(-a)P(a)} \; d\tau} \\
	= -2 \int_0^t \frac i{2\pi}\int_{\bbR}\pa_{\lambda}\lrpar{e^{-\tau(a^2+\lambda^2)}2a}e^{-(t-\tau)(a^2+\lambda^2)}\; d\lambda d\tau
\end{gathered}
\end{equation*}
where we multiply by two in applying the trace defect formula since we have the same contribution from each end of $\bbR$. Since the integrand is odd in $\lambda$, we see that $R_1=0$.

Next for $R_2$, taking into account both ends of $\bbR$ using \eqref{forpm.1} and recalling that $\alpha=-\log\varrho$,  we have that
\begin{multline*}
	R_2 = -t \; {}^R\!\Tr \left(\lrspar{ P(-a)P(a), e^{tP(-a)P(a)} \alpha} - 2 \lrspar{P(a), e^{tP(-a)P(a)}P(-a)\alpha} \right) \\
	= \frac t{2\pi} \int_{\bbR} \lrpar{ -(\pa_{\lambda}^2I(P(-a)P(a),\lambda))e^{-t(a^2+\lambda^2)} 
		+ 2 (\pa_{\lambda}^2I(P(a),\lambda))e^{-t(a^2+\lambda^2)}I(P(-a),\lambda)} \; d\lambda \\
	= \frac t{\pi} \int_{\bbR} e^{-t(a^2+\lambda^2)} \; d\lambda = \sqrt{\frac t\pi}e^{-ta^2},
\end{multline*}
and so altogether
\begin{equation}\label{eq:UltimateVariation}
	\frac{\pa}{\pa a}{}^R\!\Tr(e^{tP(-a)P(a)}) =  \sqrt{\frac t\pi}e^{-ta^2} -2t\frac{\pa}{\pa t}{}^R\!\Str(e^{t\hat{P}^2(a)}\log \varrho).
\end{equation}
$ $\\

\begin{lemma}\label{lem:Det1}
When $a=1,$ we have
\begin{equation*}
	\log\det \left(-P(-1)P(1)\right) = 0.
\end{equation*}
\end{lemma}

\begin{proof}
Note that since $\RStr{ e^{t\hat{P}^2(a)}\log \varrho }$ is an odd function of $a$ for each fixed $t,$ as are its derivatives in $t,$ we have
\begin{equation*}
	\frac{\pa}{\pa a}\lrpar{
	{}^R\!\Tr(e^{tP(-a)P(a)}) - {}^R\!\Tr(e^{tP(a)P(-a)}) }.
	= 2\sqrt{\frac t\pi} e^{-ta^2}
\end{equation*}
Since $\RStr{ e^{t\hat{P}^2(0)} }=0,$ integrating from $0$ to $a$ yields
\begin{equation}\label{eq:StrInt}
	{}^R\!\Tr(e^{tP(-a)P(a)}) - {}^R\!\Tr(e^{tP(a)P(-a)}) = 2\int_0^a \sqrt{\frac t\pi} e^{-tb^2} \; db.
\end{equation}

Hence from Lemma \ref{lem:2cases} we have
\begin{equation*}
	\RTr{ e^{tP(-1)P(1)} }
	= \RTr{ e^{tP(1)P(-1)} } + 2\int_0^1 \sqrt{\frac t\pi} e^{-tb^2} \; db
	= \frac{e^{-t}\log 2}{\sqrt{\pi t}} + 2\int_0^1 \sqrt{\frac t\pi} e^{-tb^2} \; db.
\end{equation*}
Consider this integral as a function of $t.$ Writing it alternately as
\begin{equation*}
	2\int_0^{1} \sqrt{\frac t\pi} e^{-tb^2} \; db 
	= \frac{2}{\sqrt{\pi}} \int_0^{\sqrt t}  e^{-v^2} \; dv
\end{equation*}
we see that it is $\cO(t^{1/2})$ as $t\to0$ and $1 + \cO(t^{-1/2}e^{-t})$ as $t\to \infty.$
It follows that the integral
\begin{equation*}
	2\int_0^{\infty} t^s 
	\int_0^{1} \sqrt{ \frac t\pi } e^{-tb^2} \; db \frac{dt}t
\end{equation*}
exists for all $s$ with real part in $(-1/2,0).$
For these $s$ we can use Fubini's theorem to find
\begin{equation*}
	2\int_0^{\infty} t^s 
	\int_0^{1} \sqrt{ \frac t\pi } e^{-tb^2} \; db \frac{dt}t
	= \frac2{\sqrt\pi}\int_0^1 \int_0^{\infty} t^{s+1/2} e^{-tb^2} \frac{dt}t db
	= -\frac{\Gamma(s+1/2)}{s\sqrt\pi}.
\end{equation*}
From \eqref{mero.2}, we also know that $\cM(\RTr {e^{tP(1)P(-1)} },s) = \frac{\Gamma(s-1/2)\log 2}{\sqrt\pi},$ where we recall that $\cM$ is defined in \eqref{mero.1}.   Thus, altogether it follows that 
the zeta function of $-P(-1)P(1)$ is equal to
\begin{equation*}
	\zeta_1(s) = \frac1{\Gamma(s)} \cM\lrpar{ \RTr{e^{tP(-1)P(1)}}, s }
	= \frac{\Gamma(s-1/2)\log 2}{\sqrt\pi \Gamma(s)}
	 -\frac{\Gamma(s+1/2)}{s\sqrt\pi \Gamma(s)}
	\sim 
	- 1  + \cO(s^2)
\end{equation*}
and hence
\begin{equation*}
	\log \det -P(-1)P(1) = -\pa_s \zeta_1(s)\rest{s=0} = 0.
\end{equation*}
\end{proof}

Next, using the notation of \eqref{eq:DefM0infty}, we point out that
\begin{equation*}
\begin{gathered}
	\frac{\pa}{\pa a} \cM_0\lrpar{ \RTr{ e^{tP(-a)P(a)} }, s }
	= \cM_0\lrpar{ \frac{\pa}{\pa a} \RTr{ e^{tP(-a)P(a)} }, s }, \\
	\frac{\pa}{\pa a} \cM_\infty\lrpar{ \RTr{ e^{tP(-a)P(a)} }, s }
	= \cM_\infty\lrpar{ \frac{\pa}{\pa a} \RTr{ e^{tP(-a)P(a)} }, s }, 
\end{gathered}
\end{equation*}
since this interchange is justified in the particular regions of $s \in \bbC$ where these functions are holomorphic.
Hence, for $a\neq 0,$ writing $\zeta_a(s) = \frac1{\Gamma(s)}\cM\lrpar{ \RTr{ e^{tP(-a)P(a)} }, s}$, we have by  equation \eqref{eq:UltimateVariation} that 
\begin{equation*}
	\frac{\pa}{\pa a} \zeta_a(s) 
	= 
	\frac1{\Gamma(s)}\cM\lrpar{ \sqrt{\frac t\pi} e^{-ta^2} - 2t\pa_t \; \RStr{ e^{t\hat{P}^2(a)}\log \varrho } }
	=: \bar \zeta_a(s) + \hat \zeta_a(s).
\end{equation*}
We now examine $\bar\zeta_a(s)$ and $\hat \zeta_a(s)$ in a neighborhood of $s=0.$

For $\bar\zeta_a(s)$, note that
\begin{equation*}
	\int_0^{\infty} t^{s+1/2} e^{-ta^2} \; \frac{dt}t = |a|^{-2s-1} \int_0^\infty y^{s+1/2} e^{-y} \frac{dy}y = |a|^{-2s-1}\Gamma(s+1/2)
\end{equation*}
and hence,
\begin{equation*}
	\bar\zeta_a(s) = \frac{|a|^{-2s-1}\Gamma(s+1/2)}{\sqrt\pi\Gamma(s)} \sim \frac{1}{|a|} s + \cO(s^2) \Mas s \to 0.
\end{equation*}

For $\hat\zeta_a(s),$ we start by integrating by parts to find
\begin{equation*}
	\hat\zeta_a(s) = \frac{2s}{\Gamma(s)} \cM\lrpar{ \RStr{ e^{t\hat{P}^2(a)}\log \varrho }, s }.
\end{equation*}
Indeed, the integration by parts is justified for each of $\cM_0$ and $\cM_\infty$ in the region where it is holomorphic, and the resulting boundary terms cancel out when we add together the meromorphically continued $\cM_0$ and $\cM_{\infty}.$
As above, $\cM_0\lrpar{ \RStr{ e^{t\hat{P}^2(a)}\log \varrho }, s }$ extends meromorphically from $\Re s>1/2$ to the complex plane with simple poles at a subset of $\{ \frac12-\bbN \}.$ For $a\neq 0,$ $\cM_{\infty}\lrpar{ \RStr{ e^{t\hat{P}^2(a)}\log \varrho }, s}$ extends meromorphically from $\Re s<0$ to the complex plane with a single, simple pole at $s=0$ and residue
\begin{equation*}
	\hat R_a = -\lrpar{ \Tr( \Pi_{\ker_{L^2} P(a)}\log \varrho) - \Tr(\Pi_{\ker_{L^2} P(-a)}\log \varrho)},
\end{equation*}
where $\Pi_{\ker_{L^2} P(b)}$ is the orthogonal projection onto the $L^2$-null space of $P(b).$

We will need a more explicit formula for this residue. 
We have shown that
\begin{equation*}
	\ker_{L^2} P(a) = 
	\begin{cases}
	\mathrm{span} \{ \ang X^{-a} \} & \Mif a>0  \\
	\{ 0 \} &\Mif a\leq 0.
	\end{cases}
\end{equation*}
Let us write
\begin{multline}\label{eq:ck}
	c_q = \norm{ \ang X^{-q} }_{L^2}^2 =  \int_\bbR \ang X^{-2q} \frac{dX}{\ang X}
	= 2 \int_0^1 \varrho^{2q} \frac{d\varrho}{\varrho\sqrt{1-\varrho^2}} \\
	= \int_0^1 r^{q-1} (1-r)^{1/2-1} \; dr = \mathrm B(q,1/2) = \frac{\Gamma(q)\Gamma(1/2)}{\Gamma(q+1/2)}.
\end{multline}
So, for $b>0,$  the Schwartz kernel of the projection onto $\ker P(b)$ is given by
\begin{equation*}
	\cK_{\Pi}(X, X') 
	= \frac1{c_b} \ang{ X }^{-b}\ang {X'}^{-b} \frac{dX'}{\ang {X'}},
\end{equation*}
and hence 
\begin{equation*}
	\Tr(\Pi_{\ker P(b)}\log\varrho)
	= \frac1{c_b} \int_{\bbR} \ang{ X }^{-2b} \log\left( \frac{1}{\ang X}\right) \frac{dX}{\ang {X}}
	= \frac1{2c_b} \pa_b(c_{b}).
\end{equation*}
Thus, for $a\neq 0,$ the residue equals
\begin{equation*}
	\hat R_a = -\lrpar{ \Tr( \Pi_{\ker P(a)}\log \varrho) - \Tr(\Pi_{\ker P(-a)}\log \varrho)}
	=\begin{cases}
	-\frac1{2c_a}\pa_ac_a & \Mif a>0 \\
	-\frac1{2c_{|a|}}\pa_ac_{|a|} & \Mif a<0,
	\end{cases}
\end{equation*}
which determines the behavior of $\hat\zeta_a(s)$ near $s=0.$ Indeed, since $\frac s{\Gamma(s)} = s^2 + \cO(s^3)$ and 
$\cM(\RStr{ e^{tP^2(a)}\log \varrho }, s) = \frac1s\hat R_a + \cO(1),$ we have
\begin{equation*}
	\hat \zeta_a(s) = 2\hat R_a s + \cO(s^2).
\end{equation*}

Thus altogether we have
\begin{equation*}
	\frac{\pa}{\pa a}\zeta_a(s) = \bar\zeta_a(s) + \hat\zeta_a(s) \sim \lrpar{\frac{1}{|a|} +2\hat R_a}s + \cO(s^2),
\end{equation*}
and, interchanging $\pa_a$ and $-\pa_s\rest{s=0}$, this shows that
\begin{equation*}
	\frac{\pa}{\pa a}\log \det\left( -P(-a)P(a)\right) = -\frac{1}{|a|} - 2\hat R_a
	= \begin{cases}
	\frac{\pa}{\pa a}\lrpar{ -\log a + \log c_a } & \Mif a>0 \\
	\frac{\pa}{\pa a}\lrpar{ \log |a| + \log c_{|a|} } & \Mif a<0.
	\end{cases}
\end{equation*}
Hence there exist constants $C_\pm$ such that
\begin{equation*}
	\log\det\left( -P(-a)P(a)\right) = 
	\begin{cases}
	C_+ - \log a + \log c_a & \Mif a>0 \\
	C_- +\log |a| +\log c_{|a|} & \Mif a<0
	\end{cases}.
\end{equation*}
Note that $c_1 = 2$, so from Lemmas \ref{lem:2cases} and \ref{lem:Det1} we have
\begin{equation*}
\begin{gathered}
	\log \det\left( -P(1)P(-1)\right) = 2\log 2 = C_- +\log 2\\
	\log \det \left(-P(-1)P(1)\right) = 0 = C_++\log 2,
\end{gathered}
\end{equation*}
which shows that $C_-=\log 2$ and $C_+ = -\log 2.$
Thus altogether we have shown that
\begin{equation}\label{eq:LogDet1D}
\begin{aligned}
	\log \det \left(-P(-a)P(a)\right) &=
	\begin{cases}
	\log (2|a|c_{|a|}) & \Mif a<0 \\
	0 & \Mif a=0 \\
	\log \lrpar{ \frac{c_a}{2a} } & \Mif a>0
	\end{cases} \\
	&= \begin{cases}
	\log c_{|a|} - \sign(a) \log (2|a|) & \Mif a \neq 0 \\
	0 &\Mif a=0
	\end{cases}
\end{aligned}	
\end{equation}
$ $\\

Finally, let us compute the contribution to the analytic torsion we have been looking for. 
Using the computation of the determinant, we have for all $q\neq v/2,$
\begin{multline*}
	\left[ q \lrpar{ -\log\det -P(\tfrac{v-2q}2)P(\tfrac{2q-v}2)} -(q+1)
	\lrpar{  -\log\det -P(\tfrac{2q-v}2)P(\tfrac{v-2q}2)} \right] \\
	= -q\lrpar{ \log c_{|v/2-q|} - \sign(2q-v) \log |v-2q|}
	+(q+1) \lrpar { \log c_{|v/2-q|} - \sign(v-2q) \log |v-2q|} \\
	=  \log c_{|v/2-q|} + (2q+1) \sign(2q-v) \log |v-2q|.
\end{multline*}
Substituting this into \eqref{eq:SimplifiedContrib}, we find
\begin{multline}\label{eq:FinalATDb}
	\lAT([-\pi/2,\pi/2], D_b, \cH^*(Z;F))= \\ 
	\frac12\sum_{\substack{0\leq q\leq v \\ q\neq v/2}} (-1)^qb_q
	\lrpar{ \log c_{|v/2-q|} + (2q+1) \sign(2q-v) \log |v-2q| }.
\end{multline}

\begin{remark}
If we assume that the metric $g_F$ is compatible with the flat connection on $F$ (which implies orthogonal holonomy), then we can use Poincar\'e duality to rewrite \eqref{eq:FinalATDb} as a sum over $q<v/2$:
\begin{equation}\label{eq:OrthATDb}
	\begin{cases}
	\displaystyle \sum_{q<v/2} (-1)^qb_q \lrpar{  \log c_{v/2-q} + (v-2q) \log (v-2q) } & \Mif v \text{ even}, \\
	\displaystyle \sum_{q<v/2} (-1)^qb_q \lrpar{ - (v+1) \log (v-2q) } & \Mif v \text{ odd}.
	\end{cases}
\end{equation}
\end{remark}

\section{Cusp degeneration and small eigenvalues} \label{sec:Small}

Let $g_{\ehc}$ be an $\ehc$-metric, product-type to order two, and $F \lra X_s$ a flat vector bundle with bundle metric $g_F,$ not necessarily compatible with the flat connection.
In \cite[\S\S 4-5]{ARS1} we showed that, as long as $F|_Z$ is Witt, there exists $\delta>0$ and $\eps_0>0$ such that 
\begin{equation*}
	\Spec(D_{\dR})\cap \bbS_\delta(0) = \emptyset \Mforall \eps <\eps_0,
\end{equation*}
and such that every eigenvalue in $\bbB_{\delta}(0)$ for sufficiently small $\eps$ converges to zero as $\eps\to0.$
We call the eigenvalues of $D_{\dR}$ in $\bbB_{\delta}(0)$ the {\bf small eigenvalues}
and the sum of their eigenspaces the {\bf small eigenforms}, $\Omega_{\tsmall}^q(M;\eps);$ note that the space of harmonic forms $\ker D_{\dR}^2=\ker D_{\dR}$ is a subspace of $\Omega_{\tsmall}^q(M;\eps)$. Write $\tilde d_{\eps}=\rho^{v/2} d \rho^{-v/2}$, $\tilde\delta_{\eps}=\rho^{v/2}\delta_{\eps}\rho^{-v/2}$, so that $\eth_{\dR}=d+\delta_{\eps}$ and $D_{\dR}=\tilde d_{\eps}+\tilde\delta_{\eps}$, where the $\eps$ subscripts are used to remind us that all of the operators except $d$ depend on $\eps$. Note that the small eigenvalues of $D_{\dR}$ and of $\eth_{\dR}$ are the same, and the eigenforms differ by a factor of $\rho^{v/2}$, so we may usually speak without ambiguity. As we will see in Corollary~\ref{poly.1} below, the product of the positive eigenvalues in a given degree is polyhomogeneous in $\eps$. The quantity
\begin{equation*}
	\log \tau_{\sma}(\Delta_q) =
	\FP_{\eps=0} \log \prod_{\substack{\lambda \in \Spec_{\sma}(\Delta_q)\setminus \{0\} \\ \text{repeated with multiplicity}}} \lambda
\end{equation*}
is thus well-defined, 
where $\Delta_q$ denotes the action of $D_{\dR}^2$ on forms of degree $q.$
In this section we will compute the contribution to the analytic torsion coming from these small eigenvalues, which by \cite[Theorem 11.2]{ARS1} is given by
\begin{equation}\label{eq:ATSmall}
	-\frac12 \sum_{q=0}^m (-1)^q q \log \tau_{\sma}(\Delta_q).
\end{equation}

\subsection{Surgery long exact sequence}
From \cite[\S 5]{ARS1}, the dimension of the space of small eigenforms is equal to
\begin{equation*}
	\dim \Omega_{\tsmall}^* = \dim \ker_{L^2} D_d + \dim \ker_{L^2} D_b,
\end{equation*}
where $D_d$ is defined in \eqref{dd.1}.
We can express these dimensions directly in terms of the topology of $M,$ $Z,$ and $M_0=\bhs{sm}=[M;Z].$ Let us write $\hat M_0$ for the singular space obtained from $M_0$ by coning off $\partial M_0$, which is two copies of $Z$, so that we have $\hat M_0 = M_0 \cup \cC Z\cup \cC Z,$ where 
$$
    \cC Z:= \left(Z\times [0,1]\right)/\left(Z\times \{0\}\right)
$$
is the cone over $Z$.  Assume further that $\left. F \right|_{Z}$ has trivial holonomy so that the intersection cohomology of $\hat M_0$ with value in $F$ makes sense.  
 Then it follows from work of Hausel, Hunsicker, Mazzeo \cite{hhm} and Lemma~9.5 of \cite{ARS1} that 
\begin{equation*}
	\tH_{(2)}^q(M_0;F) 
	\cong \IH_{\bar m}^q(\hat M_0;F) 
	\cong
	\begin{cases}
	\tH^q(M_0;F) & q \leq \frac{m-1}2\\
	\mathrm{Im}\lrpar{ \tH^q(M_0, \pa M_0;F) \lra \tH^q(M_0;F) } & q=\frac m2\\
	\tH^q(M_0,\pa M_0;F) & q>\frac{m-1}2
	\end{cases}
\end{equation*}
	where $\tH_{(2)}^q(M_0;F)$ denotes the $q^{\text{th}}$ $L^2$-cohomology group of $(M_0,g_0)$ with coefficients in $F$ and $\IH_{\bar m}^q(\hat M_0;F)$ is the upper middle perversity intersection cohomology of $\hat M_0$ with coefficient in $F$ as defined in \cite[\S~9]{ARS1} .  On the other hand, from the computation of the $L^2$-null space of the horizontal model operator in  \S\ref{sec:NullDb}, we see that for $q \leq \frac{v}2=\frac{m-1}2,$ 
\begin{equation*}
	\dim \Omega^q_{\tsmall} = \dim \tH^q(M_0;F) + \dim \tH^{q-1}(Z;F).
\end{equation*}
For positive $\eps,$ a subspace of dimension $\dim \ker \Delta_{\ehc} = \dim \ker D_{\dR}$ of these eigenforms will correspond to the eigenvalue zero, and the rest will correspond to positive small eigenvalues.

This suggests that, to understand the small eigenforms corresponding to non-zero small eigenvalues, we look for a long exact sequence linking $\tH^q(M_0;F),$ $\tH^q(M;F),$ and $\tH^{q-1}(Z;F).$ Observe first that $M$ is homeomorphic to the union of $M_0$ with $Z\times(-1,1)$, with an overlap region homotopic to $Z\sqcup Z$. The associated Mayer-Vietoris sequence is
\begin{multline}\label{se.4a}
	\ldots \lra \tH^{q-1}(Z;F)\oplus \tH^{q-1}(Z;F) \xlra{\tilde\pa_{q-1}} \tH^q(M;F) \lra \tH^q(M_0;F)\oplus \tH^{q}(Z;F)\\ \xlra{\tilde j_q} \tH^{q}(Z;F)\oplus \tH^{q}(Z;F) \lra \ldots,
\end{multline}
where $\tilde \pa_q$ is the boundary homomorphism of the long exact sequence and the map $\tilde j_q$ is given by $\tilde j_q(\mu,\lambda)=(\iota^*_+\mu-\lambda,\iota^*_-\mu-\lambda)$ and $\iota_{\pm}: Z\times\{\pm1\}\hookrightarrow M_0$ is the natural inclusion. Note that $\tilde j_q$ is injective when restricted to $\tH^q(Z;F)$. By using the identification
\begin{equation*}\begin{array}{ccl}
              \left(  H^q(Z,F)\oplus H^q(Z,F) \right) / \tilde{j}_q(\{0\}\times H^q(Z,F))& \to & H^q(Z,F)  \\
                        \left[\lambda_1,\lambda_2\right] & \mapsto & \lambda_1-\lambda_2,
              \end{array}
\end{equation*}
we obtain from \eqref{se.4a} a new long exact sequence
\begin{equation}\label{se.4}
	\ldots \lra \tH^{q-1}(Z;F) \xlra{\pa_{q-1}} \tH^q(M;F) \xlra{i_q} \tH^q(M_0;F) \xlra{j_q} \tH^{q}(Z;F) \lra \ldots
\end{equation}
with $j_q(\mu)=\iota_+^*\mu-\iota_-^*\mu$. 

Moreover, for $q\leq\frac{m-1}2$, we can replace $\tH^q(M_0;F)$ with $\tH^q_{(2)}(M_0;F).$ Replacing singular cohomology with Hodge cohomology (notationally by replacing $\tH$ with $\cH$) endows these spaces with inner products, so for $q\leq\frac{m-1}2$, we have short exact sequences
\begin{equation*}
	0 \lra (\ker \pa_{q-1})^{\perp} \lra \cH^q(M;F) \lra \ker j_q \lra 0,
\end{equation*}
and hence the dimension of the space of eigenforms corresponding to positive small eigenvalues is equal to
\begin{equation*}
	\dim \Omega_{\tsmall}^q(M;F) - \dim \cH^q(M;F) = \dim \ker \pa_{q-1} + \dim (\ker j_q)^{\perp}.
\end{equation*}
This in turn suggests that we relate the vector spaces
\begin{equation*}
	\cH^q_+(Z;F) := \ker \pa_q, \quad
	\cH^q_+(M_0;F) := (\ker j_q)^{\perp}
\end{equation*}
to the space of small eigenforms with positive eigenvalue, and that we relate their respective orthocomplements, which we denote $\cH^q_H(Z;F)$ and $\cH^q_H(M_0;F)$, to the space of harmonic forms on $M$. Indeed, as the notation suggests, we will see in the next subsection that these vector spaces are restrictions of the subspaces of small eigenforms of $D_{\dR}^2$ with positive eigenvalue, or with zero eigenvalue, to the two boundary faces of $X_s$. 

Finally, from \eqref{se.4}, we deduce the decomposition
\begin{equation}\label{se3.a}
\tH^q(M;F)= \cH^q_H(M_0;F)\oplus \cH^{q-1}_H(Z;F) \quad \mbox{for} \quad q\le \frac{m-1}2.
\end{equation}

To obtain a similar decomposition for $q>\frac{m-1}2$ it is convenient to use Poincar\'e duality, which exchanges $F$ with the dual flat bundle $F^*$ to which the above reasoning applies equally well.
So let us consider the dual of the long exact sequence \eqref{se.4} for $F^*.$ Notice that this sequence could have alternately been obtained by looking at the long exact sequence associated to the pair $(M, M_0)$ and using the Thom isomorphism $\tH^{q+1}_c(Z \times (-1,1);F^*) \cong \tH^q(Z;F^*)$ along with the identification $\tH^q_{(2)}(M_0;F^*) \cong \tH^q(M_0;F^*)$ for $q\leq \frac{m-1}2.$
This means that the long exact sequence dual to \eqref{se.4} for $F^*$ can be obtained by looking at the long exact sequence associated to the pair $(M, Z)$:
\begin{equation}
\xymatrix{
        \cdots \ar[r] & \tH^q_c(M_0;F) \ar[r] & \tH^q(M;F) \ar[r] & \tH^q(Z;F) \ar[r] & \tH^{q+1}_c(M_0;F) \ar[r] & \cdots.
}
\label{se.3b}\end{equation}      
By using the identification $\tH^q_c(M_0;F)\cong \tH^q_{(2)}(M_0;F)$ for $q>\frac{m-1}2$, we get
\begin{equation}
\xymatrix{
        \cdots \ar[r] \tH^q_{(2)}(M_0;F) \ar[r]^-{\hat{i}_q} & \tH^q(M;F) \ar[r]^-{\hat{\pa}_q} & \tH^q(Z;F) \ar[r]^-{\hat{j}_q} & \tH^{q+1}_{(2)}(M_0;F) \ar[r] & \cdots
}\label{se.3c}\end{equation}
where $\hat{i}_q,$ $\hat{j}_q$ and $\hat{\pa}_q$ are Poincar\'e duals of the maps $i_{m-q},$ $j_{m-1-q}$ and $\pa_{m-1-q}$ in \eqref{se.4} for $F^*.$ 
Using the Hodge $*$-operators of $g_F,$ $g_Z$ and $g_0,$ we can define for $q>\frac{m}2$  the spaces of harmonic forms
\begin{equation}
\begin{gathered}
   \cH^q_+(Z;F):= *_Z \cH^{m-1-q}_+(Z;F^*), \quad \cH^q_H(Z;F):= *_Z \cH^{m-1-q}_H(Z;F^*), \\ L^2\cH^q_+(M_0;F):= *_{g_0} L^2\cH^{m-q}_+(M_0;F^*), \quad L^2\cH^q_H(M_0;F):= *_{g_0} L^2\cH^{m-q}_H(M_0;F^*),
\end{gathered}
\label{se.3e}\end{equation} 
so that for $q>\frac{m}2$ we have the decompositions
\begin{equation}
\begin{gathered}
    \tH^q(Z;F)= \cH^q_+(Z;F)\oplus \cH^q_H(Z;F), \\
    \tH^q_{(2)}(M_0;F)= L^2\cH^q_{+}(M_0;F)\oplus L^2\cH^q_{H}(M_0;F), \\
    \tH^q(M;F)= L^2\cH^q_H(M_0;F)\oplus \cH^{q}_H(Z;F).
\end{gathered}         
\label{se.3d}\end{equation}
This last decomposition is of course related to the long exact sequence \eqref{se.3c} via the natural identifications 
\begin{equation}
    \cH^q_H(Z;F)\cong \ker \hat{j}_q \quad \mbox{and} \quad  L^2\cH^q_{+}(M_0;F)\cong \ker \hat{i}_q \quad \mbox{for} \; q> \frac{m}2.
\label{se.3g}\end{equation}

If $m$ is even, then we also have to discuss the case $q=\frac{m}2,$ where we have the decompositions
\begin{equation}
\begin{gathered}
    L^2\cH^{\frac{m}2}_{H}(M_0;F):= L^2\cH^{\frac{m}2}(M_0;F)\cong  \tH^{\frac{m}2}_{(2)}(M_0;F), \\
    \tH^{\frac{m}2}(M;F)= L^2\cH^{\frac{m}2}_H(M_0;F)\oplus \cH^{\frac{m}2-1}_H(Z;F)\oplus \cH^{\frac{m}2}_H(Z;F).
\end{gathered}     \label{se.3f}\end{equation}
Indeed, using the natural identification $\tH^{\frac{m}2}_{(2)}(M_0;F)\cong \Im\left[ \tH^{\frac{m}2}_c(M_0;F)\to \tH^{\frac{m}2}(M_0;F) \right],$ this can be checked directly using the commutative diagram 
$$
\xymatrix{ 
   \tH^{\frac{m}2-1}(Z;F) \ar[r] & \tH^{\frac{m}2}_c(M_0;F) \ar[r]\ar[d] &  \tH^{\frac{m}2}(M;F) \ar[r]^-{\hat{\pa}_{\frac{m}2}}\ar[d] & \tH^{\frac{m}2}(Z;F) \\
   \tH^{\frac{m}2-1}(Z;F) \ar[r]^-{\pa_{\frac{m}2-1}} & \tH^{\frac{m}2}(M;F) \ar[r] & \tH^{\frac{m}2}(M_0;F) \ar[r] & \tH^{\frac{m}2}(Z;F),
}
$$ 
where the top and bottom rows come from the long exact sequences \eqref{se.3c} and \eqref{se.4}.  Alternatively, we could proceed more analytically and deduce the decompositions \eqref{se.3f} from Theorem~\ref{se.11} below.

\subsection{Positive small eigenvalues} As discussed in the previous subsection, the surgery long exact sequences \eqref{se.4} and \eqref{se.3c} suggest that the vector spaces $\cH_{+}^q(Z;F)$ and $\cH_{+}^q(M_0;F)$ correspond to restrictions to the boundary faces of $X_s$ of the eigenforms corresponding to positive small eigenvalues. In this subsection, we will both see that this is indeed the case and compute the rate at which these eigenvalues approach zero as $\epsilon\to 0$. Introducing the operators
$$
           \td_{\eps}= \rho^{\frac{v}2}d \rho^{-\frac{v}2} \quad \mbox{and} \quad \tdelta_{\eps}= \rho^{\frac{v}2}\delta \rho^{-\frac{v}2}
$$
so that $D_{\dR}= \td_{\eps} + \tdelta_{\eps}$, we first make the following simple observation.

\begin{lemma}
If $\lambda_{\epsilon}$ is a positive small eigenvalue with eigenform $u_{\epsilon}\ne 0,$
$D_{\dR}^2u_{\epsilon}=\lambda_{\epsilon} u_{\epsilon},$ then $\td_{\eps} u_{\epsilon}$ and $\tdelta_{\epsilon}u_{\epsilon}$ are also eigenforms with the same small eigenvalue $\lambda_{\epsilon}.$  Furthermore, at least one of these two eigenforms is non-zero.  
\label{se.7}\end{lemma}
\begin{proof}
The first statement is immediate since both $\td_{\eps}$ and $\tdelta_{\epsilon}$ commute with $D_{\dR}$. To see that $\td u_{\epsilon}$ and $\tdelta_{\epsilon}u_{\epsilon}$ cannot be both zero, it suffices to notice that for $\epsilon>0,$
$$
       0\ne \lambda_{\epsilon}u_{\epsilon}= \td_{\eps}( \tdelta_{\epsilon}u_{\epsilon}) + \tdelta_{\epsilon}( \td_{\eps} u_{\epsilon}).
$$\end{proof}

The key step for the computation of the decay rates of the small eigenvalues is the following lemma.
\begin{lemma}
Suppose $u_{\epsilon}$ is a section of $\Lambda^q(\Ehc T^*X_s)\otimes F$, with $q\le \frac{m-1}{2}$, such that 
\begin{itemize}
\item $\Pi_{\sma} u_{\epsilon}= u_{\epsilon}$ and $\|u_{\epsilon}\|_{L^2_b}=1$;
\item $u_{\epsilon}$ is polyhomogeneous on $X_s$;
\item $\left. u_{\epsilon} \right|_{\bs}=0$ and $u_0:= \left. u_{\epsilon} \right|_{\bhs{sm}}$ is such that $\left\| u_0\right\|_{L^2_b}=1$;
\item $\tdelta_{\epsilon}u_{\epsilon}=0$;
\item $j_q([u_0])\ne 0,$ where $[u_0]\in \tH^q_{(2)}(M_0; F)$ is the cohomology class associated to $u_0.$
\end{itemize}
Then 
\begin{equation}
\| \td_{\eps} u_{\epsilon}\|^{2}_{L^2_b}= \frac{1}{4c_{(m-1)/2-q}}||j_q([u_0])||_{L^2}^2\epsilon^{m-1-2q} + o(\epsilon^{m-1-2q}),
\label{lest.1}\end{equation}
 where $c_q = \mathrm B(q,1/2)$ is the constant from  \eqref{eq:ck} and
$j_q$ is understood as a map into the harmonic forms on $Z$ (identified with the cohomology by the de Rham isomorphism). Moreover, the $(q+1)$-form $v_{\epsilon}=\frac{\td_{\epsilon}u_{\epsilon}}{\|\td_{\epsilon} u_{\epsilon}\|_{L^2_b}}$ is well-defined and
\begin{itemize}
\item $\Pi_{\sma} v_{\epsilon}= v_{\epsilon}$ and $\|v_{\epsilon}\|_{L^2_b}=1$;
\item $v_{\epsilon}$ is polyhomogeneous on $X_s$;
\item $\left. v_{\epsilon} \right|_{\bhs{sm}}=0$ and $v_b:= \left. v_{\epsilon} \right|_{\bs}$ is such that $\left\| v_b\right\|_{L^2_b}=1$;
\item $\td_{\epsilon}v_{\epsilon}=0$;
\item $v_b\in \ker\pa_q\subset\cH^q(Z; F)\cong\ker_{L^2} D^{q+1}_b$ is a non-zero multiple of $j_q([u_0]) \in \cH^q_{+}(Z;F)$.
\end{itemize}
In particular, if $u_{\epsilon}$ is an eigenform associated to a small eigenvalue $\lambda_{\epsilon},$ then 
$$
\lambda_{\epsilon}= \| \td_{\epsilon}u_{\epsilon}\|^{2}_{L^2_b}= \frac{1}{c_{(m-1)/2-q}}||j_q([u_0])||_{L^2}^2\epsilon^{m-1-2q} + o(\epsilon^{m-1-2q}),
$$
and $v_{\epsilon}$ is also an eigenform for the small eigenvalue $\lambda_{\epsilon}.$
\label{se.8}\end{lemma}

\begin{proof} 
Since $j_{(m-1)/2}([u_0])$ is always 0 by the Witt condition, the theorem statement is empty for $q=(m-1)/2$, so we may assume that $q<(m-1)/2.$
With this understood, let $u_0$ be the restriction of $u_{\epsilon}$ to $\sm$.   Since $\Pi_{\sma}u_{\epsilon}=u_{\epsilon}$, we know from \cite{ARS1} that $u_0$ is in the $L^2$-kernel of $\left.D_{\dR}\right|_{\sm}.$  This is helpful in describing the  expansion of $u_0$ near $\sm\cap\bs,$ which has two components, which we write as $\pa_+$ and $\pa_{-}.$  Near $\pa_{\pm},$ using the splitting into tangential and normal parts, we see from \eqref{moop.1} that 
$$
  u_0=  \left(  \begin{array}{c} |x|^{\ell} u_{\pm} \\ |x|^{\ell'}v_{\pm}  \end{array}\right) +  \left(\begin{array}{c}o(|x|^{\ell}) \\ o(|x|^{\ell'})\end{array}\right), \quad \mbox{for some} \; u_{\pm}\in |x|^q\cH^q(Z; F),v_{\pm}\in \frac{dx}{|x|}\wedge|x|^{q-1}\cH^{q-1}(Z; F).
$$
Indeed, clearly, the coefficients $u_{\pm}$ and $v_{\pm}$ must be in $|x|^q\cH^q(Z; F)$ and $\frac{dx}{|x|}\wedge|x|^{q-1}\cH^{q-1}(Z; F)$ for $u_0$ to be in the kernel of $D_d=\left.D_{\dR}\right|_{\sm}$.  Then the powers $\ell$ and $\ell'$ are obtained by solving the equation
$$
   \begin{pmatrix}
	0 & -|x|\pa_{|x|} + (q - \tfrac12 v) \\
	|x|\pa_{|x|} + (q - \tfrac12 v) & 0
	\end{pmatrix} \left(\begin{array}{c}|x|^{\ell} \\ |x|^{\ell'} \end{array}\right) =0,
$$ 
which gives $\ell= \frac{v}2-q,$  $\ell'= q-\frac{v}2$.  Since $\ell'<0$ and $u_0$ is in $L^2_{b}(M_0; E) = |x|^{v/2} L^2_{\hc}(M_0;E),$ this means $v_{\pm}=0$, so that in fact we have that
\begin{equation}
       u_0= |x|^{\frac{m-1}2-q} \left(  \begin{array}{c} u_{\pm} \\ 0  \end{array}\right) +  o(|x|^{\frac{m-1}2-q}), \quad \mbox{for some} \; u_{\pm}\in |x|^q\cH^q(Z; F).    
\label{se.9}\end{equation} 

Now, recall that $u_0$ is in the $L^2_b$-kernel of $D_d=\left.D_{\dR}\right|_{\sm}$ if and only if
$\widetilde{u}_0:=|x|^{-\frac{v}2}u_0$ is in the $L^2$-kernel of $\left. \eth_{\dR} \right|_{\sm}$, which by the Hodge decomposition \cite[Corollary~9.4]{ARS1} is identified with $\tH_{(2)}^q(M_0;F)\cong \tH^q(M_0;F)$.  Thus, in terms of these identifications, the cohomology class $[u_0]$ associated to $u_0$ is represented by $\widetilde{u}_0$.    
Moreover, in terms of this identification, we have that $j_q([u_0])=\widetilde{u}_+-\widetilde{u}_-$ with $\widetilde{u}_{\pm}= |x|^{-q}u_{\pm}\in \cH^q(Z;F)$.  Thus, since $j_q[\widetilde{u}_0]\neq 0,$ we must have $u_+\neq u_-,$ so at least one must be nonzero. Therefore, as a section of $\Lambda^q(\Ehc T^*X_s)$, the expansion of $u_{\epsilon}$ at $\bs$ must have a nonzero term of order $\epsilon^{\frac{m-1}2-q}.$ Moreover, since $u_+\neq u_-,$ the coefficient of this term cannot be of the form 
\begin{equation}
                  \omega(1+X^2)^{\frac{m-1}4-\frac{q}2} \quad \mbox{for some}  \; \;\omega\in \rho^{q}\cH^q(Z; F),
\label{se.8a}\end{equation} and hence cannot be in the kernel of the operator $d_b$ introduced in equation \eqref{dbdb.1}.

A priori, $u_{\epsilon}$ could have lower order terms at $\bs$.  Notice however that $u_{\epsilon}$ cannot have terms of order less than $\epsilon^{\frac{m-1}2-q}$ at $\bs$ that are in the kernel of $d_b.$  Indeed, if there were such a term of order $\epsilon^{\alpha}(\log\epsilon)^p$ with $\alpha<\frac{m-1}2-q,$ or with $\alpha=\frac{m-1}2-q$ and $p>0,$ the coefficient $u_{\alpha,p}$ of this term would be in the kernel of $d_b,$ so of the form \eqref{se.8a}.    
Let $\rho_{sm}=\frac{\epsilon}{\rho}$; it is a boundary defining function for $\sm.$ Since $\sqrt{1+X^2}= \frac{\rho}{\epsilon},$ near $\sm,$ $\epsilon^{\alpha}(\log\epsilon)^p u_{\alpha,p}$ would be of order  $\rho_{sm}^{\alpha-\frac{m-1}2+q}(\log \rho_{sm})^p.$ Since $u_{\epsilon}$ is bounded, this forces $u_{\alpha,p}=0$.

Recall that the exterior differential $d$ does not depend on $\epsilon,$ and that $\td_{\epsilon}= \rho^{\frac{m-1}2}d \rho^{-\frac{m-1}2}.$ Therefore, to understand $u_{\epsilon},$ we must examine the first term in the expansion of $u_{\epsilon}$ at $\sm$ (resp. $\bs$) which is not in the kernel of $d_0$ (resp. $d_b$); the discussion above indicates that such a term exists.  Let us denote it by $u_{lead}$ and suppose for contradiction that $u_{lead}$ occurs at order $\epsilon^{\alpha}(\log \epsilon)^p$ with either $\alpha< \frac{m-1}{2}-q$ or $\alpha=\frac{m-1}2-q$ and $p>0,$ and with $(\alpha,p)$ minimal. Then consider $\td_{\epsilon}u_\epsilon$; it is polyhomogeneous on $X_s$ with a nontrivial term of order either $\epsilon^{\alpha-1}(\log\epsilon)^p$ or $\epsilon^{\alpha}(\log\epsilon)^p$ (the loss of an order happens if and only if if the term is at $\bs$ and is not a section of $ \ker D_v$). Let $w_{\epsilon}$ be $\td_{\epsilon}u_{\epsilon}$ divided by its leading-order coefficient in $\eps$, so that $w_{\epsilon}$ has expansions at $\bs$ and $\sm$  with restriction $w_0$ to $\sm$ and/or restriction $w_b$ to $\bs$ being nontrivial.

Since $\td_{\epsilon} \Pi_{\sma}= \Pi_{\sma}\td_{\epsilon},$ $w_\epsilon$ is also in the range of $\Pi_{\sma},$ and so $w_0$ and $w_b$ are in $|x|^{(m-1)/2}\cH^q_{L^2}(M_0; F)$ and $\ker_{L^2}D_b$ respectively. In particular, we immediately see that $w_b$ must be a section of $ \ker D_v.$ Suppose that $u_{lead}$ is at $\bs$ but is not a section of the
\begin{equation*}
	 \ker D_v = \rho^{\bN}\cH^*(Z;F) \lra \bhs{sb};
\end{equation*}
then the Hodge decomposition on $Z$ would imply that $w_b$ is not a section of $ \ker D_v,$ which is a contradiction. So $u_{lead}$ must either be at $\sm$ or a section of $ \ker D_v$ at $\bs.$

Next suppose it is at $\sm$ at order $(\alpha,p)$; then by the Hodge decomposition on $\sm=M_0,$ the coefficient of $u_{lead}$ at $\sm$ cannot be in $L^2_b(M_0;E),$ so must have a term of order zero or smaller at $\bs\cap \sm$, which would contradict our assumption that $u_{\epsilon}$ is bounded with $\left. u_{\epsilon}\right|_{\bs}=0$.  Thus $w_0=0$, $w_b\neq 0$, and $u_{lead}$ occurs at $\bs$ and is a section of $ \ker D_v$. Furthermore, the same logic as in the previous paragraph tells us that no term at $\bs$ within one order of $u_{lead}$, inclusive, is not a section of $ \ker D_v$; if it were, $w_b$ would have a contribution from $\td_{\epsilon}$ applied to that term and would not be a section of $ \ker D_v$ either. We conclude that $u_{lead}$ occurs at $\bs$, is equal to $u_{\alpha,p}\eps^{\alpha}(\log\eps)^{p}$, and that $d_bu_{\alpha,p}=w_b$.

Since $w_b\in \ker_{L^2}(d_b+\delta_b)$ and $d_b^2=0$, we must have that $u_{\alpha,p}\in\ker(\delta_bd_b)$ but is not contained in $\ker d_b$.   By examining \eqref{eq:SmoothKerDb2} and using the fact that $\ell_a$ is odd in $X$ while $\langle X\rangle^a$ is even, we see that there must be a nontrivial term of order $\langle X\rangle^{\frac{m-1}2-q}$ in the expansion of the first component of $u_{\alpha,p}$ for at least one of the two  boundary compoments of $\bs.$   Since $\langle X\rangle=\frac{\rho}{\epsilon}$, this means that $u_{\alpha,p}\epsilon^{\alpha}(\log\epsilon)^p$ is of order $\rho^{\frac{m-1}2-q}\epsilon^{\alpha-(\frac{m-1}2-q)}(\log\epsilon)^p$ at $\bhs{\bs}\cap\bhs{\sm}$.  
However, unless $\alpha\geq(m-1)/2-q$ and $p=0$ if $\alpha=\frac{m-1}2-q,$ this would contradict the assumption that $u_{\epsilon}$ is bounded at $\sm.$ 
And since there is in fact a nontrivial term of order $(m-1)/2-q$ at $\bs,$ we must in fact have $\alpha=(m-1)/2-q$ and $p=0.$ 

Thus let $u_b= u_{(m-1)/2-q,0}$ be the coefficient of the term of order $\frac{m-1}2-q$ of $u_{\epsilon}$ at $\bs.$  By the argument above, $u_b$ must be an element of the form \eqref{eq:SmoothKerDb2}. To be consistent with \eqref{se.9}, this means that
\begin{equation}
u_b= \left( \begin{array}{c} u_+ -u_-  \\ 0\end{array} \right) \frac{(X^2+1)^{\frac{m-1}4-\frac{q}2}}{2}f_q(X)  +  \left( \begin{array}{c} u_+ +u_-  \\ 0\end{array} \right) \frac{(X^2+1)^{\frac{m-1}4-\frac{q}2}}{2} + \nu
\label{se.10a}\end{equation}
with $\nu\in \ker_{L^2_b}D_b$ and where 
$$
     f_q(X)= \frac{2}{c_{\frac{m-1}2-q}} \int_0^X \langle s\rangle^{2q-m}ds \quad \mbox{is such that} \; \lim_{X\to \pm \infty} f_q(X)= \pm 1.
$$

Now a simple computation using the explicit form of $d_b$ shows that the leading order term $w_b$ is precisely the vector with zero in the first factor and
\begin{equation}
\frac{1}{c_{(m-1)/2-q}}(u_+-u_-)(X^2+1)^{\frac{q}{2}-\frac{m-1}{4}}
\label{map.1}\end{equation}
in the second factor. Squaring and then integrating with respect to b-surgery densities, this finally gives
\begin{equation}
    \|\td_{\epsilon} u_{\epsilon}\|_{L^2_b}^2=\left( \frac{1}{c_{(m-1)/2-q}}\|u_+-u_-\|^2_{L^2} \right) \epsilon^{m-1-2q} + o(\epsilon^{m-1-2q}).
\label{se.10b}\end{equation}

From there, the properties of $v_{\epsilon}$ are easily obtained. The only slightly tricky part is to show that $v_{\epsilon}$ is polyhomogeneous. However, we may write \[v_{\epsilon}=\epsilon^{-(m-1)/2+q}\td_{\epsilon}u_{\epsilon}/||\epsilon^{-(m-1)/2+q}\td_{\epsilon}u_{\epsilon}||_{L^2_b}.\]  Since $||\epsilon^{-(m-1)/2+q}\td_{\epsilon}u_{\epsilon}||^2_{L^2_b},$  is bounded and has limit a positive constant as $\epsilon$ approaches zero, it follows from a direct series expansion construction that  its inverse is also polyhomogeneous in $\epsilon$ with limit a positive constant as $\epsilon$ approaches zero. Since the product of a polyhomogeneous function of $\epsilon$ and a polyhomogeneous function on $X_s$ is polyhomogeneous on $X_s,$ the result follows.
\end{proof}

We can now consider the projections.  The following theorem shows that their leading-order behavior at $\sm$ and $\bs$ is what we expect.  In the statement of the result,  we will tacitly make the following natural identification following from \eqref{eq:DbNullSpace},
\begin{equation}
       \left. \ker_{L^2}D_b^2 \right|_{\left.\Lambda^q(\Ehc T^*X_s)\right|_{\bs}} \cong \left\{ \begin{array}{ll} \cH^{q-1}(Z;F), &q\le\frac{m-1}2, \\
        \cH^{q-1}(Z;F)\oplus \cH^q(Z;F), & q=\frac{m}2, \, m \, \mbox{even}, \\
       \cH^q(Z;F), & q\ge \frac{m+1}2.  \end{array} \right.
\label{deco.1}\end{equation}

\begin{theorem} Recall from \cite{ARS1} that for some index family $\mathcal K\geq 0$, $\Pi_{\sma}\in \Psi^{-\infty,\cK}_{b,s}(X_s;E)$ is the projection onto eigenforms of $D_{\dR}$ corresponding to small eigenvalues, where $E=\Lambda^*({}^{\epsilon, d}T^*X_s)\otimes F$ and where $\Psi^{-\infty,\cK}_{b,s}$ is the space of operators whose kernels are smooth and polyhomogeneous on $X_s^2$ with index family $\cK$. Then let $\Pi^q_{\sma}$ be the subspace of $\Pi_{\sma}$ consisting of forms of pure degree $q$. We have:
\begin{itemize}
\item[(i)] $\Pi_{\sma}^q=\Pi^q_H+\Pi_+^q,$ where $\Pi^q_H, \Pi^q_+ \in  \Psi^{-\infty,\cK}_{b,s}(X_s;E)$ are projections onto the harmonic forms and onto the eigenforms associated to positive small eigenvalues respectively. 

\item[(ii)] Each of the projections $\Pi_H^q$ and $\Pi^q_{+}$ can itself be written as a sum of two projections in $\Psi^{-\infty,\cK}_{b,s}(X_s;E),$
$$
     \Pi_H^q= \Pi^q_{H, \sm}+ \Pi^q_{H,b},\ \ \ \ \Pi^q_{+}= \Pi^q_{+,\sm}+ \Pi^{q}_{+,b},
 $$
where $ \Pi^q_{H, \sm}$ restricts to the projection onto $L^2\cH^q_{H}(M_0;F)$ on $\sm$ and restricts to zero on $\bs,$ while $\Pi^q_{H,b}$ vanishes on $\sm$ and on $\bs$ restricts to the projection onto $\cH^{q-1}_H(Z;F)$ for $q\le\frac{m-1}2$ and onto $\cH^{q}_H(Z;F)$ for $q\ge\frac{m+1}2.$  Similarly, $\Pi^q_{+,\sm}$ restricts to the projection onto $L^2\cH^q_+(M_0;F)$ on $\sm$ and vanishes on $\bs,$ while $ \Pi^q_{+,b}$ vanishes on $\sm$ and on $\bs$ restricts to the projection onto $\cH^{q-1}_+(Z;F)$ for $q\le \frac{m-1}2$ and onto $\cH^q_+(Z;F)$ for $q\ge\frac{m+1}2.$  Furthermore, the image of each of these projections admits a basis of polyhomogeneous forms on $X_s.$  
\item[(iii)] For $q\le \frac{m-1}2,$  
$$\tdelta_{\epsilon} \Pi^q_{+,\sm}=0, \; D_{\dR}^2\Pi^q_{+,\sm}= \mathcal{O}(\epsilon^{m-1-2q}), \;\td_{\epsilon} \Pi^q_{+,b}=0,\; \mbox{and} \; \; D_{\dR}^2 \Pi^q_{+,b}=\mathcal{O}(\epsilon^{m+1-2q}),$$ 
and the maps
\begin{equation}\label{isoprop}
\epsilon^{-\frac{m-1}2+q-1}\td_{\epsilon}: \Im \Pi^{q-1}_{+,\sm}\to \Im \Pi^{q}_{+,b}, \quad
  \epsilon^{-\frac{m-1}2+q-1}\tdelta_{\epsilon}: \Im \Pi^{q}_{+,b}\to \Im \Pi^{q-1}_{+,\sm}\quad
\end{equation}
 are isomorphisms.
\item[(iv)]  For $q>\frac{m-1}2,$ 
 $$\td_{\epsilon} \Pi^q_{+,\sm}=0, \; D_{\dR}^2\Pi^q_{+,\sm}= \mathcal{O}(\epsilon^{2q-1-m}), \;\tdelta_{\epsilon} \Pi^q_{+,b}=0,\; \mbox{and} \; \; D_{\dR}^2 \Pi^q_{+,b}=\mathcal{O}(\epsilon^{2q+1-m}),$$ and the maps
\begin{equation}\label{isoprop.2}
\epsilon^{\frac{m+1}2-q-1}\tdelta_{\epsilon}: \Im \Pi^{q+1}_{+,\sm}\to \Im \Pi^{q}_{+,b}, \quad
  \epsilon^{\frac{m+1}2-q-1}\td_{\epsilon}: \Im \Pi^{q}_{+,b}\to \Im \Pi^{q+1}_{+,\sm}\quad
\end{equation}
 are isomorphisms.
 \item[(v)] For $q=\frac{m}2$ when $m$ is even, 
 $$
     \Pi^{\frac{m}2}_{+,\sm}=0 \quad \mbox{and} \quad \Im \Pi^{\frac{m}2}_{+,b}= \left(\epsilon^{\frac{m-1}2}\td_{\epsilon}\Im \Pi^{\frac{m}2-1}_{+,\sm}\right)    \oplus  \left(\epsilon^{\frac{m-1}2}\tdelta_{\epsilon}: \Im \Pi^{\frac{m}2+1}_{+,\sm}\right).  $$ 
 \end{itemize}
\label{se.11}\end{theorem}

\begin{proof}   
In degree $q=0,$ $\Pi_{H}^0$ is just the projection onto $\ker D_{\dR}^2|_{q=0}$, which is $\rho^{\frac{m-1}2}\Gamma_{\operatorname{flat}}(M;F)$, where $\Gamma_{\operatorname{flat}}(M;F)$ is the space of flat sections of $F.$  Since $\Gamma_{\operatorname{flat}}(M;F)$ does not depend on $\epsilon,$ we clearly see that $\Pi_{H}^0\in \Psi^{-\infty,\cK}_{b,s}(X_s;E)$ and that its range admits a basis of polyhomogeneous sections of $F$ on $X_s.$  Since $\ker_{L^2}(d_b+\delta_b)$ is trivial in degree zero, $\Pi^0_{+,b}=\Pi^0_{H,b}=0$ and $\Pi_H^0=\Pi^0_{H,\sm}.$  Thus, we can take $\Pi_{+,\sm}^0=(\Pi_{H,\sm}^0)^{\perp}\subset\Pi^0_{\sma}$; this projection is polyhomogeneous on $X_{b,s}^2$ and its image restricts to the image of the projection onto $(\ker j_0)^{\perp}$ at $\sm$ and to $0$ at $\bs.$ We claim the image has a basis which is polyhomogeneous on $X_s$ and restricts to a basis of harmonic forms corresponding to $(\ker j_0)^{\perp}$ at $\sm$ and to 0 at $\bs.$  Indeed, the argument is standard and proceeds as follows. Take a basis of harmonic forms corresponding to $(\ker j_0)^{\perp}$ at $\sm\subset X_s$; each is polyhomogeneous on $\sm,$ so we can extend each basis element to a polyhomogeneous form on $X_s.$ Then applying the projection $\Pi^0_{+,\sm}$ to each element yields a polyhomogeneous basis for $\epsilon$ small. (This argument to go from polyhomogeneity of the projection to polyhomogeneity of a basis works for any of the projections involved in this proof, as all are polyhomogeneous on $X_{b,s}^2$ and have restrictions at $\mf$ and $\bhs{bf}$ to projections which have polyhomogeneous bases on $\sm$ and $\bs$ respectively.) By Lemma~\ref{se.8}, we see that for $u_{\epsilon}$ and $v_{\epsilon}$ in the range of $\Pi^0_{+,\sm},$ 
$$
        \langle v_{\epsilon}, D_{\dR}^2 u_{\epsilon}\rangle_{L^2_b} =
          \langle \td_{\epsilon} v_{\epsilon}, \td_{\epsilon} u_{\epsilon}\rangle_{L^2_b}= \mathcal{O}(\epsilon^{m-1}).
$$
Since $D_{\dR}$ and $\Pi_{\sma}$ commute, this means that $D_{\dR}^2\Pi^0_{+,\sm}=\mathcal{O}(\epsilon^{m-1}).$  

From here, we now proceed inductively to prove the rest of the theorem for $q\leq\frac{m-1}2$. Suppose it is true for degree $q-1$; we must show it is true for degree $q$.
First, we take $\Pi^{q}_{+,b}$ to be the projection onto the range of $\epsilon^{-\frac{m-1}2+q-1}\td_{\epsilon} \Pi^{q-1}_{+,\sm}$ with respect to the inner product induced by the metrics $g_{\ehc}$ and $g_F$.
We claim that it has all the required properties. Indeed, it is clear that $\td_{\eps}\Pi^{q}_{+,b}=0$.  If $\{u_i\}$ is a polyhomogeneous basis of $\Im \Pi^{q-1}_{+,\sm},$  then by Lemma~\ref{se.8},   $v_i=\eps^{-\frac{m-1}2+q-1}\td_{\epsilon} u_i$ is a polyhomogeneous basis of $\Im\Pi^q_{+,b}$ such that  $v_i=0$ on $\sm$ and $v_i\rest{\bs}$ is a basis of  $\ker_{L^2}D_b^{q}\cong\cH_+^{q-1}(Z;F)$.
Moreover, the first map in \eqref{isoprop} is by construction an isomorphism, whereas the second map is an isomorphism thanks to the estimate \eqref{lest.1} in the statement of Lemma~\ref{se.8}.  Thus, we see that \eqref{isoprop} holds for $\Pi^{q}_{+,b}$.

Next we construct $\Pi^q_{H,b}.$  If $(\ker \pa_{q-1})^{\perp}=0,$ then $\Pi^q_{H,b}=0.$  If not, let $\omega\in \rho^q(\ker \pa_{q-1})^{\perp}$ be a non-zero element and let 
$$
   u_b= (1+X^2)^{\frac{q}2-\frac{m-1}4} \left( \begin{array}{c} 0 \\ \omega \end{array} \right)
$$  
be the corresponding element in $\ker_{L^2}D_b.$ Let $\chi\in \CI_c(\bbR)$ be a cutoff function taking values between $0$ and $1$, with $\chi(t)=1$ for $|t|<\delta$ and $\chi(t)=0$ for $|t|> 2\delta,$ where $\delta>0$ is chosen small enough, and consider $u_{\epsilon}:=\chi(x) u_b(x/\epsilon)$. Since the form $u_b$ is closed and involved a $dX= \frac{dx}{\varepsilon}$, we see that $\td_{\epsilon}u_{\epsilon}=0.$  Moreover, we have that $\left. u_{\epsilon}\right|_{\sm}=0$ and that $\tdelta_{\epsilon}u_{\epsilon}$ is polyhomogeneous and bounded on $X_s,$ vanishing at both $\sm$ and $\bs.$ 

First set
$$
    v_{\epsilon}:= u_{\epsilon}- \Pi^q_{+,b} u_{\epsilon}.
$$        
 Since $\td_{\epsilon}\Pi_{+,b}=0,$ we still have that $\td_{\epsilon} v_{\epsilon}=0.$  By the properties of $\Pi^q_{+,b},$ we also have that $\left.v_{\epsilon}\right|_{\bs}=\left. u_{\epsilon}\right|_{\bs}=u_b$ and $\left. v_{\epsilon}\right|_{\bhs{sm}}=0.$     By construction, $\Pi^q_{+,b}v_{\epsilon}=0,$ which implies by \eqref{isoprop} that $\Pi^{q-1}_{+,\sm}\tdelta_{\epsilon}v_{\epsilon}=0.$ However, by the Hodge decomposition $\tdelta_{\epsilon}v_{\epsilon}$ is orthogonal to the harmonic forms, and since $\tdelta_{\epsilon}\tdelta_{\epsilon}v_{\epsilon}=0,$ the inductive hypothesis implies that $\tdelta_{\epsilon}v_{\epsilon}$ is also orthogonal to the image of $\Pi^{q-1}_{+,b}.$ Therefore $\Pi^{q-1}_{\sma}\tdelta_{\epsilon}v_{\epsilon}=0.$ Finally, note that $\Pi^{q}_{+,b}u_{\epsilon}$ is also polyhomogeneous and bounded on $X_s,$ vanishing to positive order at both $\sm$ and $\bs.$ Therefore $\tdelta_{\epsilon}\Pi^{q}_{+,b}u_{\epsilon}$ is polyhomogeneous on $X_s,$ and \eqref{isoprop} implies that it is in $\epsilon^{\frac{m-1}{2}-q+1}L^2$ and thus bounded, and in fact vanishing at both $\sm$ and $\bs.$ The same is therefore true for $\tdelta_{\epsilon}v_{\epsilon}.$

Now set
\[\mu_{\epsilon}=-(\td_{\epsilon}+\tdelta_{\epsilon}-\Pi_{\sma})^{-1}(\tdelta_{\epsilon} v_{\epsilon}).\]
By the resolvent construction \cite[Theorem~4.5, Corollary~5.2]{ARS1} and the mapping properties of surgery operators \cite[Theorem~3.3]{ARS1}, $\mu_{\epsilon}$ is polyhomogeneous and bounded on $X_s$.  Moreover, since  $\tdelta_{\epsilon}v_{\epsilon}$ vanishes to positive order at  $\sm$ and $\bs$, the same is true for $u_{\epsilon}.$
Now, by construction, we have that $\Pi^{q}_{\sma}\mu_{\eps}=0$ and 
$$
          (\td_{\epsilon}+\tdelta_{\epsilon})\mu_{\epsilon}= -\tdelta_{\epsilon} v_{\epsilon} \; \Longrightarrow \; \td_{\epsilon} \mu_{\epsilon}= -\tdelta_{\epsilon} (\mu_{\epsilon} +v_{\epsilon}).
 $$
Since $\td_{\epsilon}$ and $\tdelta_{\epsilon}$ have orthogonal range by  the Hodge decomposition, this means that $\td_{\epsilon}\mu_{\epsilon}=0$ and that $\tdelta_{\epsilon} \mu_{\epsilon}=-\tdelta_{\epsilon}v_{\epsilon}$.   In particular, since $v_{\varepsilon}$ is of degree $q$, we see that  $\mu_{\epsilon}$ is also of degree $q$.  
Finally, set $w_{\epsilon}= v_{\epsilon} +\mu_{\epsilon}$; we have that
$$      
D_{\dR}^2w_{\epsilon}= \td_{\epsilon}\tdelta_{\epsilon}v_{\epsilon} + \td_{\epsilon}\tdelta_{\epsilon}\mu_{\epsilon}= \td_{\epsilon}\tdelta_{\epsilon}v_{\epsilon}
+\td_{\epsilon}(\td_{\epsilon}+\tdelta_{\epsilon}-\Pi^{q}_{\sma})\mu_{\epsilon}= \td_{\epsilon}\tdelta_{\epsilon}v_{\epsilon}- \td_{\epsilon}\tdelta_{\epsilon}v_{\epsilon}=0. 
$$
So $w_{\epsilon}$ is harmonic. Since $w_{\epsilon}=u_{\epsilon}-\Pi^q_{+,b} u_{\epsilon}+\mu_{\epsilon},$ $w_{\epsilon}$ has the same restrictions at $\sm$ and $\bs$ as $u_{\epsilon},$ namely 0 and $u_b.$

Thus, taking a basis $\omega_{1},\ldots,\omega_p$ of $(\ker\pa_{q-1})^{\perp},$ we can find harmonic forms $w_1,\ldots, w_p$ on $X_s,$ polyhomogeneous and bounded on $X_s,$ such that
$$
    \left.w_i\right|_{\sm}=0, \quad \left. w_i\right|_{\bs}= (1+X^2)^{\frac{q}2-\frac{m-1}4} \left( \begin{array}{c} 0 \\ \omega_i \end{array} \right)
$$
and with $[\left.w_i\right|_{\epsilon=c}]\in \tH^q(M; F)$ a positive multiple of $\pa_{q-1}[\omega_i]$ for $c>0.$  We can thus define $\Pi^q_{H,b}$ to be the projection on the span of $w_1, \ldots, w_p.$  

To construct $\Pi^q_{H,\sm}$ we can proceed in a similar fashion.  If $\ker j_q=\{0\},$ then we can just pick $\Pi_{H,\sm}^q=0.$  Otherwise, take a class $\mu\in \ker j_q$ and choose a class $\tau\in \tH^q(M; F)$ such that $i_q(\tau)= \mu.$  Represent $\tau$ by a smooth form $v$ of degree $q$ on $M$; without loss of generality, we can assume that in a tubular neighborhood $Z\subset M,$ $v$ is of the form
$$
  \left. v\right|_{Z\times(-1,1)_x}= \omega
$$
with $\omega\in \cH^q(Z; F)$ independent of $x.$  In particular, if $q=\frac{m-1}2,$ this means by the Witt condition that $\omega=0$ so that $\left. v\right|_{Z\times(-1,1)_x}=0.$  If instead $q<\frac{m-1}2,$ then the norm of $v\in L^2\Lambda^q(M_{\epsilon}; F)$ is uniformly bounded as $\epsilon$ approaches zero.  Thus, for any $q\le \frac{m-1}2,$ we have that the form $v_{\epsilon}:=\rho^{\frac{m-1}2}v$ on $X_s$ is in $L^2_b.$  Since $dv=0,$ it is such that $\td_{\epsilon}v_{\epsilon}=0.$  Moreover, we have that $\left. v_{\epsilon}\right|_{\bs}=0,$ while $\left.v_{\epsilon}\right|_{sm}$ represents the class $\mu\in \ker j_q \subset \tH^2_{(2)}(M_0; F).$  Consider then
$$
  u_{\epsilon}:= v_{\epsilon}- (\Pi^q_{H,b}+ \Pi^q_{+,b}) v_{\epsilon}.
$$
Then we still have that $\td_{\epsilon}u_{\epsilon}=0,$ $\left.u_{\epsilon}\right|_{\bs}=0$ and $\left. u_{\epsilon}\right|_{\sm}$ represents the class $\mu.$  We also have that
$\Pi^{q-1}_{\sma}\tdelta_{\epsilon}u_{\epsilon}=0,$ so the form
$$
\mu_{\epsilon}= -(\td_{\epsilon}+\tdelta_{\epsilon}-\Pi_{\sma})^{-1}\tdelta_{\epsilon} u_{\epsilon}
$$  
is a well-defined $q$-form with $\Pi^{q}_{\sma}\mu_{\epsilon}=0.$  As in the construction of 
$\Pi^q_{H,b},$ the form $w_{\epsilon}=u_{\epsilon}+\mu_{\epsilon}$ is harmonic with $\left. w_{\epsilon}\right|_{\bs}=0$ and with $\left.w_{\epsilon}\right|_{\sm}$ the harmonic representative of the class $\mu.$  

Thus, starting with with a basis $\mu_1,\ldots,\mu_{\ell}$ of $\ker j_q,$ we can construct harmonic forms $w_1,\ldots, w_{\ell}$ such that $\left. w_i\right|_{\bs}=0,$ $\left. w_i\right|_{\bhs{sm}}$ represents $\mu_i$ and $\Pi^q_{H,b} w_i=0.$  Then we define $\Pi^q_{H,\sm}$ to be the projection on the range of $w_1,\ldots,w_{\ell}.$  

Finally, we set
$$
\Pi^q_{+,\sm}= (\Pi^q_{+,b}\oplus \Pi^q_{H,b}\oplus \Pi^q_{H,\sm})^{\perp}\subset\Pi^{q}_{\sma}.
$$
As before, it is polyhomogeneous with a basis which is polyhomogeneous on $X_s$. Since we understand the restrictions of every space on the right-hand side, we conclude that it has the appropriate restrictions at $\sm$ and $\bs$. Let $u_{\eps}\in\Pi^q_{+,\sm}$, then consider $\epsilon^{-\frac{m-1}{2}+q}\tdelta_{\eps}u_{\eps}$, which we claim is equal to zero. Suppose not. Then it is certainly in the image of $\Pi^{q-1}_{\sma}$, and as it is in the image of $\tdelta_{\eps}$, it must be in the image of one of the $+$ projections. By duality and the fact that $d_{\eps}\Pi^{q-1}_{+,b}=0$, it is orthogonal to everything in the image of $\Pi^{q-1}_{+,b}$, so it must be in the image of $\Pi^{q-1}_{+,\sm}$.
However, $\Pi^{q}_{+,b}u_{\eps}=0$; therefore, by the inductive hypothesis, $\Pi^{q-1}_{+,\sm}\tdelta_{\eps}u_{\eps}=0$. Hence $\tdelta_{\eps}\Pi^{q}_{+,\sm}=0$ as required. Then Lemma ~\ref{se.8} immediately gives the required estimates for $D_{\dR}^2\Pi^{q}_{+,\sm}$. This completes the inductive step, and the proof for $q\leq\frac{m-1}2$.

For $q\geq\frac{m+1}2$, we obtain the results by applying Poincar\'e duality and the corresponding results for $F^*.$ 

Finally, if $m$ is even, then, applying Lemma~\ref{se.8} as in the case $q=0$ to $\Pi^{\frac{m}2-1}_{+,\sm}$, we obtain the part of $\Pi^{\frac{m}2}_{+,b}$ projecting on $\epsilon^{\frac{m-1}2}\td_{\epsilon}\Im \Pi^{\frac{m}2-1}_{+,\sm}.$  Using the Hodge $*$-operator and the corresponding result for $F^*$ we get the other part of $\Pi^{\frac{m}2}_{+,b}.$  Since each eigenspace of positive small eigenvalues is even dimensional and formed of pairs of eigenfunctions by Lemma~\ref{se.7}, we see from the statement of the theorem when $q\ne \frac{m}2$ that we get all of $\Pi_+^{\frac{m}2},$ so $\Pi_{+,\sm}^{\frac{m}2}=0.$
\end{proof}

\begin{corollary} In every degree, the product of positive small eigenvalues is polyhomogeneous in $\epsilon>0$.  Furthermore, for $q\leq(m-1)/2,$ the product of all positive small eigenvalues of $D_{\dR}^2$ in degree $q$ is asymptotic to 
\[(\frac{1}{c_{v/2-q}}\epsilon^{m-1-2q})^{\dim \cH^q_+(Z;F)}|\det((j_{q})_{\perp})|^2\cdot(\frac{1}{c_{v/2-(q-1)}}\epsilon^{m-1-2(q-1)})^{\dim \cH^{q-1}_+(Z;F)}|\det((j_{q-1})_{\perp})|^2\]
as $\epsilon\searrow 0,$ where the subscript $\perp$ for a map $d$ denotes the restriction $d_\perp: (\ker d)^{\perp}\to \im d$ of $d$ to the orthogonal complement of its kernel with respect to the $L^2$-inner product. We use the convention that $|\det(j_q)_{\perp}|=1$ when $(j_q)_{\perp}$ is the zero map.

For $q\geq(m+1)/2,$ the product is asymptotic to
\[(\frac{1}{c_{(q-1)-v/2}}\epsilon^{2(q-1)-v})^{\dim \cH^{q-1}_+(Z;F)}|\det((j_{q-1})_{\perp})|^2\cdot
(\frac{1}{c_{q-v/2}}\epsilon^{2q-v})^{\dim \cH^{q}_+(Z;F)}|\det((j_{q})_{\perp})|^2\]
as $\epsilon\searrow 0,$ where $j_q:= j_{m-q-1}: H^{m-q-1}(M_0;F^*)\to H^{m-q-1}(Z; F^*)$.

If $m$ is even and $q=\frac{m}2,$ then the product is asymptotic to
$$
    \left((\frac{1}{c_{1/2}}\epsilon)^{\dim \cH^{\frac{m}2-1}_+(Z;F)}|\det((j_{\frac{m}2-1})_{\perp})|^2 \right)\cdot \left((\frac{1}{c_{1/2}}\epsilon)^{\dim \cH^{\frac{m}2+1}_+(Z;F)}|\det((j_{\frac{m}2})_{\perp})|^2 \right)
$$
as $\epsilon\searrow 0,$ where $j_{\frac{m}2}:= j_{m-\frac{m}2-1}: H^{\frac{m}2-1}(M_0;F^*)\to H^{\frac{m}2-1}(Z;F^*)$.
\label{poly.1}\end{corollary}
\begin{proof}  By Theorem~\ref{se.11}, the product of positive small eigenvalues in degree $q$ is polyhomogeneous in $\epsilon$ since it is given by $\det(\Pi^q_+ D^2_{\dR}\Pi^q_+).$ In the asymptotic behavior of this product for $q\le \frac{m-1}2$, the first term comes from the eigenvalues corresponding to $\Pi^{q}_{+,\sm}$ and is computed using \eqref{se.10b}.  Here, we are relying on the fact that the map $(j_{q})_{\perp}$ maps an orthonormal basis by harmonic representatives of $\tH^q_+(M_0; F)$ to an orthogonal basis of harmonic representatives of $\tH^q_+(Z; F).$  Indeed, if $u_0$ and $v_0$ are orthogonal harmonic representatives of classes in $\tH^q_+(M_0; F)$ and $u_{\epsilon}$ and $v_{\epsilon}$ are extensions in the range of $\Pi_{\sma},$ then similarly to \eqref{se.10b}, we have that
\begin{multline*}
	    o(\epsilon^{m-1-2q})
	    = \langle v_{\epsilon}, \widetilde{\Delta}_{\epsilon}u_{\epsilon}\rangle_{L^2} 
	    = \langle \widetilde{d}_{\epsilon}v_{\epsilon}, \widetilde{d}_{\epsilon}u_{\epsilon}\rangle_{L^2} \\
	    = \left( \frac{1}{c_{v/2-q}} \langle v_+-v_-, u_+-u_- \rangle_{L^2}\right) \epsilon^{m-1-2q} + o(\epsilon^{m-1-2q}),
\end{multline*}
from which we see that $j_q(u_0)= u_+-u_-$ and $j_q(v_0)=v_+ -v_-$ are orthogonal as claimed.  The second term comes from the small eigenvalues corresponding to $\Pi^{q}_{+,b},$ which by Lemma~\ref{se.7} are the same as those corresponding to $\Pi^{q-1}_{+,\sm}.$ 
This proves the first statement for $F$ and the same argument also proves it for $F^*.$ The second statement in the theorem follows immediately by Poincar\'e duality.  For the last statement, it suffices to notice that from part (v) of Theorem~\ref{se.11}, the positive small eigenvalues in degree $\frac{m}2$ are the same as the small eigenvalues coming from $\Pi^{\frac{m}2-1}_{+,\sm}$ and $\Pi^{\frac{m}2+1}_{+,\sm}.$
 \end{proof}

We can now compute \eqref{eq:ATSmall}.
First let us define
\begin{equation*}
	a_q = \begin{cases}
	-\dim \tH^q_+(Z;F) \log ({c_{|v/2-q|}})  + 2 \log |\det (j_q)_\perp| & q \neq v/2 \\
	0 & q = v/2 \Mor q=-1
	\end{cases}
\end{equation*}
and note that, with $\Delta_q=\eth_{\dR}^2$ in degree $q$, taking the finite part in $\eps$ gives
\begin{equation*}
	\log\tau_{\tsmall}(\Delta_q) = a_{q-1} + a_q.
\end{equation*}
Thus we have
\begin{multline}\label{eq:FinalATsmall}
	-\frac12 \sum_{q=0}^m (-1)^q q \log \tau_{\sma}(\Delta_q)
	= -\frac12 \sum_{q=0}^m (-1)^q q(a_{q-1}+a_q)
	= \frac12 \sum_{q=0}^m (-1)^q a_q\\
	= \frac12 \sum_{\substack{0\leq q\leq m\\ q\neq v/2}} 
	(-1)^q\lrpar{  -\dim \tH^q_+(Z;F) \log ({c_{|v/2-q|}})  + 2 \log |\det (j_q)_\perp| }.
\end{multline}

\begin{remark}
If we assume that the metric $g_F$ is compatible with the flat connection on $F$ (which implies orthogonal holonomy), then we can use Poincar\'e duality to rewrite \eqref{eq:FinalATsmall} as a sum over $q\leq m/2,$
\begin{equation}\label{eq:OrthATsmall}
	\begin{cases}
	\displaystyle \sum_{\substack{q< v/2}} 
	(-1)^q\lrpar{  -\dim \tH^q_+(Z;F) \log ({c_{|v/2-q|}})  + 2 \log |\det (j_q)_\perp| }  & \Mif v \ev,\\
	0 & \Mif v \odd.
	\end{cases}
\end{equation}
\end{remark}

\subsection{Harmonic bases} \label{sec:HarmBases}
The second consequence of the analysis in \cite{ARS1} that we will use is about the asymptotics of harmonic forms as $\eps \to 0.$
Recall from \S \ref{subsec:AT} that when $m$ is odd, the analytically defined metric invariant quantity is
\begin{equation*}
	\bar{\lAT}(M,\{\mu_j^q\},F) = \lAT(M,g_{\ed},F,g_F)  - \log \left(  \Pi_{q=0}^{m} [\mu^q|\omega_{\epsilon}^q ]^{(-1)^q} \right)
\end{equation*}
where $\mu = \{ \mu_j^q \}$ is a fixed basis of $\tH^*(M;F),$  $\omega_{\epsilon}$ is an orthonormal basis of harmonic representatives with respect to the metrics $g_{\ehc}$ and $g_F$, and where $[\mu^q|\omega_{\epsilon}^q ] =|\det W^q|$ with $W^q$ the matrix such that $$\displaystyle \mu^q_i= \sum_j W^q_{ij} \omega^q_j.$$ 
In this section we compute the asymptotic expansion of $\log \left(  \Pi_{q=0}^{m} [\mu^q|\omega_{\epsilon}^q ]^{(-1)^q} \right)$ when $m$ is odd. We are interested in the coefficient of $\eps^0,$ as terms dependent on $\eps$ will cancel out with those in the expansion of $\lAT(M,g_{\ed},F,g_F).$  \\

To compute this contribution, we will make a specific choice for the basis $\mu.$  Namely, we let $\mu^q_{M_0}$ and $\mu^q_Z$ be bases of orthonormal harmonic representatives for $M_0$ and for $Z$ with respect to the metrics $g_0$ and $g_Z$ respectively and the metric $g_F;$ by an orthogonal transformation we can also assume without loss of generality that they are compatible with the decompositions
\begin{gather}\label{ff.7a}
    \IH_{\bm}^q(M_0;F) = L^2\cH^q_H(M_0;F) \oplus L^2\cH^q_{+}(M_0;F), \\
   \label{ff.7b} \tH^q(Z;F)= \cH^q_H(Z;F)\oplus \cH^q_+(Z;F). \end{gather}
Then take $\mu^q$ to be the subset of $(\mu^q_{M_0},\mu^{q-1}_Z)$ which is a basis compatible with the canonical decomposition
\begin{equation}
       \tH^q(M;F)= (\ker j_q) \oplus (\ker \pa_{q-1})^{\perp} =: \tH^q_{H}(M_0;F) \oplus \tH^{q-1}_H(Z;F) \quad \mbox{for}\; q\le \frac{m-1}2.
\label{ff.8}\end{equation}
Similarly, for $q>\frac{m-1}2,$ we take $\mu^q$ to be the subset of $(\mu^q_{M_0},\mu^{q}_Z)$ compatible with the canonical decomposition
\begin{equation}
       \tH^q(M)= \tH^q_{H}(M_0;F) \oplus \tH^{q}_H(Z;F), \quad  q> \frac{m-1}2.
\label{ff.9}\end{equation}
With these choices, the constant term in the asymptotic expansion of    $[\mu^q |\omega_{\epsilon}^q]$ comes from $\mu^{q-1}_Z$ if $q\le \frac{m-1}2$ and from $\mu^q_Z$ otherwise.  To be precise, let  $\alpha$ be a harmonic form on $Z$ with $[\alpha]\in \cH^{q-1}_H(Z;F)$ and $\|\alpha\|_{L^2}=1$ and let $\beta\in \cH^q(Z;F^*)$ be a harmonic form Poincar\'e dual to $\alpha,$ so that 
\begin{equation}
    \int_Z \alpha\wedge \beta=1.  
\label{ff.10}\end{equation}
If $q\le\frac{m-1}2,$ the element $[\nu]\in \tH^q(M;F)$ corresponding to the form $\alpha$ in the decomposition \eqref{ff.8} can be represented by a form $\nu$ with support in a tubular neighborhood $(-1,1)\times Z$ of $Z$ in $M$ such that
\begin{equation}
        \int_{(-1,1)\times Z} \nu\wedge\beta =1.  
\label{ff.11}\end{equation} 
On the other hand, by Theorem \ref{se.11}, the harmonic form $\omega_{\epsilon}$ with respect to $g_{\ehc}$ of $L^2$-norm equal to $1$ representing a positive multiple of the class as $[\nu]$ in $\tH^q(M;F)$ is  asymptotically of the form 
$$
          \omega_{\epsilon} \sim  \frac{1}{\sqrt{c_{\frac{m-1}2-(q-1)}}}\langle X\rangle^{q-1-\frac{m-1}2}\left(\rho^{-\frac{m-1}2} \frac{dx}{\rho}\wedge \rho^{q-1} \alpha\right)= \frac{\epsilon^{q-1-\frac{m-1}2}}{\sqrt{c_{\frac{m-1}2-(q-1)}}} \langle X\rangle^{2q-2-(m-1)} \frac{dX}{\langle X\rangle}\wedge\alpha
$$
in a neighborhood of $\bhs{sb}$.
Thus, from \eqref{ff.11}, we see that  asymptotically as $\epsilon$ tends to zero, $[\nu]\sim \gamma_{\epsilon} [\omega_{\epsilon}]$ with
$$
    \gamma_{\epsilon}^{-1} = \int_{-\infty}^{\infty} \frac{\epsilon^{q-1-\frac{m-1}2}}{\sqrt{c_{\frac{m-1}2-(q-1)}}} \langle X\rangle^{2q-2-(m-1)} \frac{dX}{\langle X\rangle}= \epsilon^{q-1-\frac{m-1}2}\sqrt{c_{\frac{m-1}2-(q-1)}}.$$
This implies that for $q\le\frac{m-1}2,$ 
\begin{equation}
      \log[\mu^q|\omega^q] = -\dim \cH^{q-1}_H(Z;F) \left(\frac12 \log c_{\frac{m-1}2-(q-1)} +(q-1-\frac{m-1}2)\log\epsilon \right)+ o(1)
\label{ff.12}\end{equation}
as $\epsilon$ tends to zero.  

When $q\ge\frac{m+1}2,$ the form $\nu$ representing the cohomology class $[\nu]$ corresponding to the form $\alpha\in\cH^q_H(Z;F)$ under the decomposition \eqref{ff.9} is such that  its restriction to $Z$ is $[\alpha],$ in other words,   
\begin{equation}
   \int_Z \nu\wedge \beta=1.  
\label{ff.13}\end{equation}
Another application of Theorem \ref{se.11} shows that the harmonic form $\omega_{\epsilon}$ with respect  $g_{\epsilon}$ of $L^2$-norm equal to $1$ representing a positive multiple of the class as $[\nu]$ in $\tH^q(M;F)$ is  asymptotically of the form 
\begin{equation}
          \omega_{\epsilon} \sim  \frac{1}{\sqrt{c_{q-\frac{m-1}2}}}\langle X\rangle^{\frac{m-1}2-q}\left(\rho^{-\frac{m-1}2} \rho^{q} \alpha\right)= \frac{\epsilon^{q-\frac{m-1}2}}{\sqrt{c_{q-\frac{m-1}2}}} \alpha
\label{ff.14}\end{equation}
as $\epsilon$ tends to zero.  From \eqref{ff.13}, we thus see that 
$$
      [\nu]\sim \frac{\sqrt{c_{q-\frac{m-1}2}}} {\epsilon^{q-\frac{m-1}2}} [\omega_{\epsilon}]
$$ 
as $\epsilon$ tends to zero.  Taking the logarithm, we obtain that for $q>\frac{m-1}2,$ 
\begin{equation}
  \log[\mu^q|\omega^q]= \dim \cH^{q}_H(Z;F)\left( \frac{1}2\log c_{q-\frac{m-1}2} +(\frac{m-1}2-q)\log\epsilon \right)+ o(1).
\label{ff.15}\end{equation}
Combining \eqref{ff.12} and \eqref{ff.15}, we see that when $m$ is odd,
\begin{equation}\label{ff.16}
	- \FP_{\eps=0} \log[\Pi_{q=0}^m [\mu^q|\omega^q]^{(-1)^q}]
	= -\frac12 \sum_{\substack{0\leq q \leq m \\ q \neq v/2}} (-1)^q \dim \cH^{q}_H(Z;F) \log c_{|v/2-q|}.
\end{equation}

\begin{remark}
If we assume that the metric $g_F$ is compatible with the flat connection on $F$ (which implies orthogonal holonomy), then we can use Poincar\'e duality to rewrite \eqref{ff.16} as a sum over $q<v/2$: 
\begin{equation}\label{eq:Orthff.16}
	\displaystyle -\sum_{q<v/2} (-1)^q \dim \tH^{q}_H(Z;F)\log c_{\frac{v}2-q}.
\end{equation}
\end{remark}

\section{Cusp degeneration and analytic torsion} \label{sec:AT}

Let $M$ be a closed manifold with a two-sided hypersurface $Z.$ We endow $X_s$ with an $\ehc$ metric $g_{\ehc}$ and a flat bundle $F\lra X_s$ with bundle metric $g_F,$ not necessarily compatible with the flat connection, and in this section we determine the limit as $\eps \to 0$ of analytic torsion.\\

In \cite[Theorem 11.2]{ARS1}, we have computed the constant term in the expansion as $\eps \to 0$ of the logarithm of analytic torsion:
\begin{multline}
	\FP_{\eps=0} \lAT(M,g_{\ehc}, F, g_F) = \lAT([M;Z],g_0,F,g_F) \\
	+ \lAT([-\pi/2,\pi/2], D_b, \cH^*(Z;F)) - \tfrac12\sum (-1)^q q \log\tau_{\tsmall}(\Delta_q).
\label{finfor.1}\end{multline}
We have now computed the last two terms, which leads to the following theorem.

\begin{theorem}\label{thm:MainAT}
Let $M$ be an odd-dimensional closed manifold, $Z$ a two-sided hypersurface in $M,$ and $F\lra X_s$ a flat vector bundle endowed with a bundle metric $g_F$ that is even at $\bhs{sb}(X_s),$ not necessarily compatible with the flat connection.  Assume that $F$ has trivial holonomy when restricted to $Z$.
Let $g_{\ehc}$ be an $\ehc$ metric that is product-type to order two, and let $\mu_{M_0}^k$ and $\mu_Z^k$ be orthonormal bases of harmonic representatives with respect to the metrics $g_0$ and $g_Z$ respectively, as well as $g_F;$ then let $\mu$ be the corresponding basis of $\tH^*(M;F)$ induced by $\mu_{M_0}$, $\mu_Z$ and the decompositions \eqref{ff.8} and \eqref{ff.9}. 
Then the analytic torsion satisfies
\begin{multline}\label{eq:FinalATFormula}
	\bar{\lAT}(M,\mu,F) 
	= \bar{\lAT}([M;Z],\mu_0,F) \\
	+ 
	\frac12\sum_{\substack{0\leq q\leq v \\ q\neq v/2}} (-1)^q
	\lrspar{ 2 \log |\det (j_q)_\perp|  + \dim\cH^q(Z;F) \lrpar{ (2q+1) \sign(2q-v) \log |v-2q| } }.
\end{multline}

If $g_F$ is compatible with the flat connection on $F$ (which implies orthonormal holonomy), then we have
\begin{multline*}
	\bar{\lAT}(M, \mu,F) =  \bar{\lAT}([M;Z],\mu_0,F)\\
	+\sum_{q=0}^{\frac{v}2-1} (-1)^q\big [2\log |\det((j_q)_{\perp})|+\dim\cH^q(Z;F)\left ( (v-2q)\log(v-2q)\right )\big ].
\end{multline*}
\end{theorem}

\begin{proof}
The term $\lAT([-\pi/2,\pi/2], D_b, \cH^*(Z;F))$ is computed in \eqref{eq:FinalATDb} to be
\begin{equation*}
	\frac12\sum_{\substack{0\leq q\leq v \\ q\neq v/2}} (-1)^qb_q
	\lrpar{ \log c_{|v/2-q|} + (2q+1) \sign(2q-v) \log |v-2q| },
\end{equation*}
the term $- \tfrac12\sum (-1)^q q \log\tau_{\tsmall}(\Delta_q)$ is computed in \eqref{eq:FinalATsmall} and is equal to
\begin{equation*}
	\frac12 \sum_{\substack{0\leq q\leq m\\ q\neq v/2}} 
	(-1)^q\lrpar{  -\dim \tH^q_+(Z;F) \log ({c_{|v/2-q|}})  + 2 \log |\det (j_q)_\perp| },
\end{equation*}
and the contribution from the harmonic basis is computed in \eqref{ff.16} to be
\begin{equation*}
	- \FP_{\eps=0} \log[\Pi_{q=0}^n [\mu^q|\omega^q]^{(-1)^q}]
	= \frac12 \sum_{\substack{0\leq q \leq m \\ q \neq v/2}} (-1)^q \lrpar{ - \dim \tH^{q}_H(Z;F) \log c_{|v/2-q|}}.
\end{equation*}
Adding these together and using the fact that
\begin{equation*}
	\dim \cH^q(Z;F)= \dim \cH^q_+(Z;F)+ \dim \cH^q_H(Z;F),
\end{equation*}
we see that
\begin{multline*}
	\bar{\lAT}(M,\mu,F) 
	= \bar{\lAT}([M;Z],\mu_0,F) \\
	+ 
	\frac12\sum_{\substack{0\leq q\leq v \\ q\neq v/2}} (-1)^q
	\lrpar{ 2 \log |\det (j_q)_\perp|  + (2q+1) b_q \sign(2q-v) \log |v-2q| }.
\end{multline*}

If $g_F$ is compatible with the flat connection on $F,$ then we can rewrite this as
\begin{multline*}
	\bar{\lAT}(M,\mu,F) 
	= \bar{\lAT}([M;Z],\mu_0,F) \\
	+ 
	\sum_{q < v/2} (-1)^q
	\lrpar{ 2 \log |\det (j_q)_\perp|  +(v-2q) b_q\log |v-2q| }.
\end{multline*}

\end{proof}

When $m$ is even and $g_F$ is a metric compatible with the flat connection, we can compute the analytic torsion of the manifold with cusp directly from \eqref{eq:FinalATFormula}.  

\begin{theorem} The analytic torsion of a manifold with cusp $(N_0,g_0)$ with link $Z$, of even dimension $m$, equipped with a flat Euclidean vector bundle $F$ satisfying the Witt condition and with trivial holonomy when restricted to $Z$, is given by
$$
\lAT(N_0,g_0,F)= \frac{m}{2} \sum_{q<\frac{m-1}2} (-1)^q\dim\cH^q(Z;F)\log(m-1-2q).
$$
\label{even.2}\end{theorem}
\begin{proof}
Applying \eqref{eq:FinalATFormula} to the case where $M_0= N_0 \sqcup N_0$ is the disjoint union of two copies of $(N_0,g_0),$ and using \eqref{eq:OrthATDb}, \eqref{eq:OrthATsmall} and the fact that $\lAT(M,g_{\epsilon},F)=0$,  we see that
$$
\lAT([M;Z],g_0,F)= m \sum_{q<\frac{m-1}2} (-1)^q
\dim\cH^q(Z;F)\log(m-1-2q),$$
from which the result follows.
\end{proof}

\section{Cusp degeneration and Reidemeister torsion} \label{sec:RT}

We assume as in the previous section that $F\to X_s$ is a flat vector bundle such that  $\tH^{\frac{m-1}2}(Z;F)=\{0\}$ (the Witt condition) and with holonomy inducing a unimodular representation $\alpha:\pi_1(M)\to \GL(k,\bbR)$.  For the interserction $R$-torsion of Dar to be defined, we need also to assume that the map $\alpha\circ \iota_*$ is trivial, where $\iota: Z\to M$ is the inclusion and $\iota_*: \pi_1(Z)\to \pi_1(M)$ is the induced map.  In other words, we need to assume $F$ is trivial when restricted to $Z$.  We also let $g_F$ be a choice of bundle metric for $F$.  Moreover, we will now assume that $m$ is odd.  To study the change of the $R$-torsion under a pinching surgery, we will make use of the long exact sequence \eqref{se.4}.  As a complex, we will denote this long exact sequence by $\cH_1.$  

Recall that
\begin{equation}
     \hM= \sm \bigcup ( \cC Z \sqcup \cC Z), \quad \sm\bigcap(\cC Z \sqcup \cC Z) = Z \sqcup Z,
\label{rt.1}\end{equation}
is the singular space associated to $\sm,$ where $\cC Z$ is the disjoint union of the cones of each connected components of $Z.$   To relate the $R$-torsion of $M$ with an appropriate intersection $R$-torsion on $\hM,$ we will need another exact sequence, namely the Mayer-Vietoris sequence obtained by writing $\hM$ as the union of $\sm$ with $\cC Z \sqcup \cC Z.$

The pseudomanifold $\hM$ has a natural stratification of depth one, with singular stratum given by a disjoint union of points.  Let $T$ be a choice of triangulation on $\hM$ compatible with this stratification and the decomposition \eqref{rt.1}.  Recall from \cite{ARS1} that we can then use $T$ and its first barycentric subdivision $T'$ to define the complex of cochains $\cR^{*}_{\bm}(\hM,\alpha)$.  This complex has natural restrictions to $\cC Z\sqcup\cC Z $ and  $Z \sqcup Z,$ so there is an induced Mayer-Vietoris short exact sequence of finite dimensional complexes
\begin{equation}
 \xymatrix{
 0 \ar[r] & \cR^*_{\bm}(\hM;F)\ar[r] & C^*_{T'}(M_0;F) \oplus \cR^*_{\bm}(\cC Z \sqcup \cC Z;F) \ar[r] & C^*_{T'}(Z \sqcup Z;F)\ar[r] & 0.
 }
 \label{rt.1a}\end{equation}
Here,  $C^*_{T'}(M_0;F)= C^*_{\widetilde{T}'}(\widetilde{\sm})\otimes_{\bbZ\pi_1(\sm)} \bbR^q$, where $\bbR^q$ is seen as a $\bbZ\pi_1(\sm)$-module via the representation $\alpha: \pi_1(M)\to \GL(k,\bbR)$ given by the holonomy of $F,$ $\widetilde{T}'$ is the lift of $T'$ to the universal cover $\widetilde{\sm}$ of $\sm$, and $C^*_{\widetilde{T}'}(\widetilde{\sm})$ is the group of cochains associated to the triangulation $\widetilde{T}'.$  Again, we are assuming $\alpha$ is trivial when restricted to $\pi_1(Z)$.    Similarly, we have that 
$$
C^*_{T'}(Z \sqcup Z;F)= \bigoplus_{i} \left[ C^*_{\widetilde{T}'}(Z_i)\otimes_{\bbZ\pi_1(Z_i)}\bbR^q \oplus C^*_{\widetilde{T}'}(Z_i)\otimes_{\bbZ\pi_1(Z_i)}\bbR^q\right],
$$
where $i$ labels the connected components of $Z.$  
   Now, the short exact sequence \eqref{rt.1a} induces a Mayer-Vietoris long exact sequence involving intersection cohomology
\begin{equation}
 \xymatrix{
 \cdots \ar[r] & \IH^q_{\bm}(\hM;F)\ar[r]^-{i_q'} & \tH^q(M_0;F) \oplus \IH^q_{\bm}(\cC Z \sqcup \cC Z;F) \ar[r]^-{j_q'} & \tH^q(Z \sqcup Z;F)\ar[r]^-{\pa_q'} & \cdots,
 }
\label{rt.2}\end{equation}
where $\IH^q_{\bm}(\hM;F)=\IH^q_{\bm}(\hM,\alpha)$ and $\IH^q_{\bm}(\cC Z \sqcup \cC Z;F)= \IH^q_{\bm}(\cC Z\sqcup \cC Z, \alpha).$  We will denote the long exact sequence \eqref{rt.2} by $\cH_2.$

\begin{theorem}
Given any bases for $\tH^q(M;F),$ $\IH^q_{\bm}(\hM;F),$ $\IH^q_{\bm}(\cC Z;F)$, $\tH^q(Z;F)$, and $\tH^q(M_0;F)$, we define corresponding bases for
$\IH^q_{\bm}(\cC Z\sqcup \cC Z;F)=\IH^q_{\bm}(\cC Z;F)\oplus \IH^q_{\bm}(\cC Z;F)$, and $\tH^q(Z \sqcup Z;F)= \tH^q(Z;F)\oplus \tH^q(Z;F)$ via the direct sum. Using these bases to define the $R$-torsions of $M,$ $\hM,$ $\sm,$ $\cC Z$, $\cC Z \sqcup \cC Z,$ $Z,$ $Z\sqcup Z,$ $\cH_1$ and $\cH_2$, we have
\begin{equation}\label{rt.3a}
\tau(M;F)=\frac {I\tau^{\bm}(\hM;F)\tau(Z;F)}{ I\tau^{\bm}(\cC Z;F)^2} \frac{\tau(\cH_2)}{\tau(\cH_1)}.
\end{equation}
\label{rt.3}\end{theorem}
\begin{proof}
By the formula of Milnor \cite{Milnor1966}, we have that
$$
     \tau(M_0;F)= \tau(M;F) \tau(Z;F) \tau (\cH_1) 
$$
and 
$$
 \tau(M_0;F)I\tau^{\bm}(\cC Z;F)^2= I\tau^{\bm}(\hM;F) \tau(Z;F)^2 \tau(\cH_2).
$$
Combining these two relations gives the result. Note that the direct sum assumption is used to write, for example, $\tau(Z\sqcup Z;F)=(\tau(Z;F))^2$.
\end{proof}

We now make a particular choice of bases for these spaces that allows a direct comparison with \eqref{eq:FinalATFormula} and also 
 makes some of the terms in \eqref{rt.3a} more explicit, in particular $\tau(\cH_2)\tau(\cH_1)^{-1}.$ Recall from section \ref{sec:HarmBases} the decompositions:
\begin{equation}\label{rt.4a}
\begin{gathered}
    \IH_{\bm}^q(M_0;F) = L^2\cH^q_H(M_0;F) \oplus L^2\cH^q_{+}(M_0;F), \\
    \tH^q(Z;F)= \cH^q_H(Z;F)\oplus \cH^q_+(Z;F);
    \end{gathered}\end{equation}
\begin{equation}
       \tH^q(M;F)=  \L^2\cH^q_{H}(M_0;F) \oplus \cH^{q-1}_H(Z;F) \quad \mbox{for}\; q\le \frac{m-1}2;
\label{rt.4b}\end{equation}
\begin{equation}
       \tH^q(M)= L^2\cH^q_{H}(M_0;F) \oplus \cH^{q}_H(Z;F), \quad  q> \frac{m-1}2.
\label{rt.4c}\end{equation} We use the same bases as in section~\ref{sec:HarmBases} for $\IH^q_{\bm}(\hM;F)$, $\tH^q(Z;F)$, and $\tH^q(M;F)$; namely, orthonormal bases $\mu^q_{M_0}$, $\mu^q_{Z}$ compatible with the decompositions \eqref{rt.4a} and with $\mu^q$ the basis induced from $\mu_{M_0}$, $\mu_Z$ and the decompositions \eqref{rt.4b} and \eqref{rt.4c}. 

For $q\le \frac{m-1}2,$ we use the canonical identification $\tH^q(M_0;F)=\IH^q_{\bm}(\hM;F)$ to get a corresponding basis for $\tH^q(M_0;F).$  Similarly, for $q\ge \frac{m+1}2,$ the canonical identification $\tH^q_c(M_0;F)= \IH^q_{\bm}(M;F)$ gives a corresponding basis for $\tH^q_c(M_0;F):=H^q(\sm, \pa \sm, F).$  We also need to make a choice of basis of $\IH^q_{\bm}(\cC Z;F);$  we will take the one induced by our choice of basis for $\tH^q(Z;F)$ and  the canonical identification
\begin{equation}
           \IH^q_{\bm}(\cC Z;F) = \left\{ \begin{array}{ll}
                \tH^{q}(Z;F), & q\le \frac{m-1}{2}, \\
            \{0\}, & q> \frac{m-1}{2}.    \end{array}   \right. 
\label{rt.5a}\end{equation}

 It remains to make a choice of basis for $\tH^q(M_0;F)$ when $q>\frac{m-1}2,$ but at the moment we can at least make a partial computation of $\tau(\cH_2)\tau(\cH_1)^{-1}.$  

\begin{lemma}
With the choice of bases made above, the contribution to $\tau(\cH_2)\tau(\cH_1)^{-1}$ coming from cohomology classes of degree $q\le \frac{m-1}2$ is given by
$$
      \prod_{q<\frac{m-1}2} (|\det (j_q: \cH^q_+(M_0;F)\to \cH^q_+(Z;F) )  |)^{(-1)^q}. 
$$
\label{rt.5}\end{lemma}   
\begin{proof}
First we compute the contribution to $\tau(\cH_2)$. To do this, it suffices to notice that the restriction of $j_q'$ to the second factor induces the identity map $j_q': \IH^q_{\bm}(\cC Z \sqcup \cC Z;F) \to \tH^q(Z\sqcup Z;F)$ with respect to our choice of bases, while $i_q'$ composed with the projection on the first factor gives the canonical identification $\IH^q_{\bm}(\hM;F)= \tH^q(M_0;F)$ (for $q\le \frac{m-1}2$) used to choose our basis for $\tH^q(M_0;F)$.

As for $\tau(\cH_1)$, the decompositions \eqref{rt.4a} and \eqref{rt.4b} are such that  in the long exact sequence $\cH_1,$
 $$
         \im(i_q)= \cH^q_H(M_0;F), \quad \im j_q= \cH^q_+(Z;F), \quad \im \pa_{q}= \cH^q_H(Z;F)\subset H^{q+1}(M;F).
 $$ 
 With our choice of bases, this means that the contribution to $\tau(\cH_1)$ coming from cohomology classes of degree $q\le \frac{m-1}2$ is given by
 $$
 \prod_{q<\frac{m-1}2} \left( ( \frac{|\det (i_q)_{\perp}||\det (\pa_q)_{\perp}|}{|\det (j_q)_{\perp}|}\right)^{(-1)^q}.
 $$ 
(Recall that $d_{\perp}$ is the restriction of a map $d$ to $(\ker d)^{\perp}$). Since $(i_q)_{\perp}: \cH^q_H(M_0;F)\to \cH^q_H(M_0;F) $ and $(\pa_q)_{\perp}:\cH^q_H(Z;F)\to \cH^q_H(Z;F)$ are the identity maps on $\cH^q_H(M_0;F)$ and $\cH^q_H(Z;F)$ for our choices of bases, the result follows.    
\end{proof}

In degree $q>\frac{m-1}2,$ we will take advantage of some cancellations occurring between $\tau(\cH_1)$ and $\tau(\cH_2)$ to compute $\tau(\cH_2)\tau(\cH_1)^{-1}$ directly.  First, notice that for $q>\frac{m-1}2,$ the long exact seqence $\cH_2$ corresponds to the relative long exact sequence associated to the pair $(\sm, \pa \sm),$
\begin{equation}
 \xymatrix{
 \cdots \ar[r] & \tH^q_c(M_0;F)\ar[r]^-{i_q'} & \tH^q(M_0;F)  \ar[r]^-{j_q'} & \tH^q(Z \sqcup Z;F)\ar[r]^-{\pa_q'} & \cdots,
 } \quad q\ge \frac{m+1}2,
\label{rt.7}\end{equation} 
 under the canonical identification $\tH^q_c(M_0;F)= \IH^q_{\bm}(\hM;F).$  This leads to the following commutative diagram between the long exact sequences $\cH_1$ and $\cH_2$ when $q\ge\frac{m+1}2:$
\begin{equation}
 \xymatrix{
 \cdots \ar[r] & \tH^q_c(M_0;F)\ar[r]^{i_q'}\ar[d]^{\hat{i}_q} & \tH^q(M_0;F)  \ar[r]^-{j_q'}\ar[d]^{\Id} & \tH^q(Z \sqcup Z;F)\ar[r]^{\pa_q'}\ar[d]^{\beta_q} & \tH^{q+1}_c(M_0;F)\ar[r]\ar[d]^{\alpha_{q+1}} & \cdots \\
  \cdots \ar[r] & \tH^q(M;F) \ar[r]^{i_q} & \tH^q(M_0;F) \ar[r]^{j_q} & \tH^q(Z;F) \ar[r]^-{\pa_q}  &H^{q+1}(M;F) \ar[r] & \cdots, } 
 \label{rt.8}\end{equation}
 where $\hat{i}_q: \tH^q_c(M_0;F)\to \tH^q(M;F)$ defined in \eqref{se.3c} is the standard push-forward map and the map $\beta_q$ is given by
 $$
       \begin{array}{lccl}
         \beta_q: & \tH^q(Z;F)\oplus \tH^q(Z;F) &\to & \tH^q(Z;F) \\
                        & (\mu_+,\mu_-) & \mapsto & \mu_+ -\mu_-,
       \end{array}
 $$
and the canonical identification $\tH^q(Z\sqcup Z;F)= \tH^q(Z;F)\oplus \tH^q(Z;F).$
 This definition suggests that we take a different orthonormal basis of harmonic forms on $\tH^q(Z;F)\oplus \tH^q(Z;F).$  Namely, if $\nu_1,\ldots, \nu_{i_q}$ is our chosen basis for $\tH^q(Z;F),$ then we take the basis
 $$
               (\frac{\nu_1}{\sqrt{2}},\frac{\nu_1}{\sqrt{2}}),\ldots, (\frac{\nu_{i_q}}{\sqrt{2}},\frac{\nu_{i_q}}{\sqrt{2}}), (\frac{\nu_1}{\sqrt{2}},-\frac{\nu_1}{\sqrt{2}}),\ldots,(\frac{\nu_{i_q}}{\sqrt{2}},-\frac{\nu_{i_q}}{\sqrt{2}}),
 $$
 for $\tH^q(Z\sqcup Z;F)=\tH^q(Z;F)\oplus \tH^q(Z;F).$  Since this change of basis is orthogonal, it has no effect on the torsion of $Z\sqcup Z$, and in particular \eqref{rt.3a} still holds if we compute the torsion of $Z\sqcup Z$ with respect to this new basis. We can now make the following simple observations.
 \begin{lemma} For $q\ge \frac{m+1}2$ the following assertions hold:
 \begin{enumerate}
 \item $\ker\hat{i}_q= L^2\cH^q_+(M_0;F)\subset \IH_{\bm}^q(\hM;F)= \tH^q_c(M_0;F)$;  
 \item $\im\hat{i}_q\cong L^2\cH^q_{H}(M_0;F) \subset \tH^q(M;F)$;
 \item If $\omega_1,\ldots,\omega_{\ell_q}$ is an orthonormal basis of $\cH^q_+(Z;F),$ then
 $$
                \pa_q'(\frac{\omega_1}{\sqrt{2}},\frac{\omega_1}{\sqrt{2}}), \ldots \pa_q'(\frac{\omega_{\ell_q}}{\sqrt{2}},\frac{\omega_{\ell_q}}{\sqrt{2}})
 $$
 is a basis of $\cH_+^{q+1}(M_0;F).$  This basis is, however, not necessarily orthonormal;  
 \item $\beta_q\circ j_q' \circ i_q(\cH^q_H(Z;F))=0$;
 \item  The composition
 $$
 \xymatrix{
       \cH^q_H(Z;F)\ar[r]^-{i_q} & \tH^q(M_0;F) \ar[r]^-{j_q'} & \tH^q(Z;F)\oplus \tH^q(Z;F) \ar[r]^-{\pr_d} & \tH^q(Z;F) \ar[r] & \cH^q_H(Z;F) 
  }     
 $$
 is the identity map, where $\pr_q$ is the map
 $$
        \begin{array}{lccl} 
        \pr_q: & \tH^q(Z;F)\oplus \tH^q(Z;F) & \to & \tH^q(Z;F) \\
                   & (\mu_1,\mu_2) & \mapsto &  \frac{\mu_1+\mu_2}{2}.
        \end{array}
 $$
 In particular, the map $i_q$ is injective when restricted to $\cH^q_H(Z;F)\subset \tH^q(M;F).$

\end{enumerate}
\label{des.1}\end{lemma}
\begin{proof}
The proof of (1) and (2) follows from the identification of $\tH^q_c(M_0;F)\cong \tH^q_{(2)}(M_0;F)$ when $q\ge \frac{m+1}2.$  For (3), it follows by noticing that, still under the identification  $\tH^q_c(M_0;F)\cong \tH^q_{(2)}(M_0;F),$ we have that  $\pa_q'\circ \iota_q= \hat{j}_q$ where 
$$
  \begin{array}{lccl} 
        \iota_q: &  \tH^q(Z;F) &\to &H^q(Z;F)\oplus \tH^q(Z;F)   \\
                   & \mu & \mapsto &  (\mu,\mu)
        \end{array}      
$$
is the diagonal inclusion.  For (4), this is by exactness of the bottom sequence in \eqref{rt.8}, since
$$
           \beta_q\circ j_q'\circ i_q= j_q\circ i_q=0.
$$
Finally, (5) follows from \eqref{se.3g} and the definition of the map $i_q.$

\end{proof}
 
 Therefore, removing the span of $(\frac{\nu_1}{\sqrt{2}},\frac{\nu_1}{\sqrt{2}}), \ldots (\frac{\nu_{i_q}}{\sqrt{2}},\frac{\nu_{i_q}}{\sqrt{2}})$ in $\tH^q(Z\sqcup Z;F),$ $\cH^q_+(M_0;F)$ in $\tH^q_c(M_0;F),$ $\cH^q_H(Z;F)$ in $\tH^q(M;F)$ and $i_q(\cH^q_H(Z;F))$ in $\tH^q(M_0;F),$ we obtain from \eqref{rt.8} the following commutative diagram of long exact sequences:
 \begin{equation}
  \xymatrix{
 \cdots \ar[r] & \tH^q_H(M_0;F)\ar[r]\ar[d]^{\Id} & \tH^q(M_0;F)/i_q(\cH^q_H(Z;F))  \ar[r]\ar[d]^{\Id} & \tH^q(Z;F)\ar[r]\ar[d]^{\Id} &  \cdots \\
  \cdots \ar[r] & \tH^q_H(M_0;F)\ar[r] & \tH^q(M_0;F)/i_q(\cH^q_H(Z;F))  \ar[r] & \tH^q(Z;F)\ar[r] &  \cdots. } 
 \label{rt.9}\end{equation}
In this diagram, the contribution of the top row to $\tau(\cH_1)^{-1}\tau(\cH_2)$ is cancelled by the contribution from the bottom row.  Therefore, in degree $q\ge \frac{m+1}2,$ the only contributions to $\tau(\cH_1)^{-1}\tau(\cH_2)$ come from 
\begin{itemize}
\item $\pa_q'$ when restricted to the span of $(\frac{\omega_1}{\sqrt{2}},\frac{\omega_1}{\sqrt{2}}), \ldots (\frac{\omega_{\ell_q}}{\sqrt{2}},\frac{\omega_{\ell_q}}{\sqrt{2}})$ in $\tH^q(Z\sqcup Z;F);$
\item $i_q: \cH^q_H(Z;F) \to i_q(\cH^q_H(Z;F));$ and
\item $j_q': i_q(\cH^q_H(Z;F))\to j_q'\circ  i_q (\cH^q_H(Z;F))\cong \cH^q_H(Z;F).$
\end{itemize}

To reach this conclusion, we have tacitly assumed that we have chosen a basis of $\tH^q(M_0;F)$ for $q\ge \frac{m+1}2$ which includes a basis of $i_q(\cH^q_{H}(Z;F)).$  We can go one step further and choose this basis so that the corresponding basis of $i_q(H^q_H(Z;F))$ is the image under $i_q$ of the chosen basis on $\cH^q_H(Z;F).$  With this choice, $i_q$ does not contribute to $\tau(\cH_1)^{-1}\tau(\cH_2)$ and we are left with the contributions of $\pa_q'$ and $j_q'.$  We are now ready to state the refinement of Theorem~\ref{rt.3}.

\begin{theorem}
Let $\mu^q_{M_0}$ and $\mu^q_Z$ be bases of $\IH^q_{\bm}(\hM;F)$ and $\tH^q(Z;F)$, orthonormal with respect to the metrics $g_0$ and $g_Z$ respectively, and compatible with the decompositions \eqref{rt.4a}. Let $\mu^q$ be a basis of $\tH^q(M;F)$ compatible with \eqref{rt.4b} and \eqref{rt.4c} and choose the basis for $\IH^q_{\bm}(\cC Z;F)$ induced by \eqref{rt.5a}.  Using these bases to define the corresponding $R$-torsions, we have the relation
\begin{equation*}
\tau(M;F)= \frac{I\tau^{\bm}(\hM;F)\tau(Z;F)}{I\tau^{\bm}(\cC Z;F)^2}\left( \prod_{q}|\det(j_q)_{\perp}|^{(-1)^q} \right)\left( \prod_{q>\frac{m-1}2} \sqrt{2}^{(-1)^{q+1}\dim H^q(Z;F)} \right).
\end{equation*}
\label{rt.10}\end{theorem}
 \begin{proof}
  By Theorem~\ref{rt.3}, it remains therefore to compute $\tau(\cH_1)^{-1}\tau(\cH_2).$  By Lemma~\ref{rt.5} and the discussion above, it remains to compute the contributions coming from $j'_q$ and $\pa_q'$ in degree $q\ge \frac{m+1}2.$    
 
 First, notice that with the choice of basis we have made for $\tH^q(M_0;F),$ $j'_q$ is almost an isometry; more precisely, $\frac{j_q'}{\sqrt{2}}$ is an isometry.  The contribution 
 of $j_q'$ to $\tau(\cH_1)^{-1}\tau(\cH_2)$ is therefore given by
 \begin{equation}
   \prod_{q>\frac{m-1}2} |\det (j_q': i_q(\cH^q_{H}(Z;F))\to j_q' (i_q(\cH^q_H(Z;F))|^{-(-1)^q}= \prod_{q>\frac{m-1}2} \sqrt{2}^{-(-1)^q\dim \cH^q_H(Z;F)}.
 \label{rt.11}\end{equation}
 On the other hand, identifying the span of $(\frac{\omega_1}{\sqrt{2}},\frac{\omega_1}{\sqrt{2}}), \ldots (\frac{\omega_{\ell_q}}{\sqrt{2}},\frac{\omega_{\ell_q}}{\sqrt{2}})$ isometrically with $\cH^q_+(Z;F),$ we see that the map $\sqrt{2}\pa_q': \cH^q_+(Z;F)\to L^2\cH^{q+1}_+(M_0;F)$ corresponds to $\hat{j}_q,$ hence to the adjoint of $j_{m-q-1}: L^2\cH^{m-q-1}_+(M_0;F^*)\to \cH^{m-q-1}_+(Z;F^*)$ where $F^*$ is the flat vector bundle dual of $F$.  Thus, the contribution of $\pa_q': \cH^q_+(Z;F)\to L^2\cH^{q+1}_+(M_0;F)$ to $\tau(\cH_1)^{-1}\tau(\cH_2)$ is given by
 \begin{equation}
   \prod_{q>\frac{m-1}2} (\sqrt{2}^{-\dim \cH^{q}_+(Z;F)}  |\det( (j_q)_{\perp}|   )^{(-1)^q}
 \label{rt.12}\end{equation}  
 with the convention again that $j_q:= j_{m-q-1}: H^{m-q-1}(M_0;F^*)\to H^{m-q-1}(Z;F^*)$ for $q>\frac{m-1}2$.  
 Combining \eqref{rt.11} and \eqref{rt.12} with Lemma~\ref{rt.5} and using the fact that $\dim \tH^q(Z;F)= \dim \cH^q_H(Z;F)+ \dim \cH^q_+(Z;F)$ gives the result. \end{proof}

\begin{remark}
If $F$ is a Euclidean flat vector bundle then $\tau(Z;F)=1$ by \cite[Proposition~1.19]{Cheeger1979}, so using Poincar\'e duality, the formula simplifies to
\begin{equation}
\tau(M;F)=\frac{I\tau^{\bm}(\hM;F)}{I\tau^{\bm}(\cC Z;F)^2} \left(  \prod_{q<\frac{m-1}{2}}|\det((j_q)_{\perp})|^{2(-1)^q}  \right) 2^{-\frac{\chi(Z;F)}{4}}.
\label{rt.10a}\end{equation}    \end{remark}

 \begin{remark}
 Even though we have made a specific choice of basis for $\tH^q(M_0;F)$ to prove Theorem~\ref{rt.10}, the final result is independent of such a choice.
 \end{remark}

\section{A Cheeger-M\"uller theorem for Witt representations on manifolds with cusps}\label{sec:CheegerMuller}

Let $(M,g)$ be a closed Riemannian manifold with a flat bundle $F \lra M$ corresponding to a unimodular representation $\alpha:\pi_1(M)\to \GL(k;\bbR)$. The Cheeger-M\"uller theorem \cite{Muller1993} gives an equality between the analytic torsion and the R-torsion for every choice of basis $\mu$ of the cohomology groups:
\begin{equation*}
	\bar\lAT(M,\mu,F) = \log\tau(M,\mu,F).
\end{equation*}
We have analyzed the behavior of both sides of this equation under analytic cusp surgery - the left-hand side in Theorem \ref{thm:MainAT} and the right-hand side in Theorem \ref{rt.10}. In this section we use these results to conclude a Cheeger-M\"uller theorem for Witt representations on manifolds with cusps.

\begin{theorem} Let $M$ be an odd-dimensional manifold, $Z \subseteq M$ a two sided hypersurface, $g_{\ehc}$ a cusp surgery metric which is product-type to order two, and $F \lra X_s$ a flat Witt vector bundle with holonomy inducing a unimodular representation $\alpha:\pi_1(M)\to \GL(k,\bbR)$ such that $\alpha$ is trivial when restricted to the image of $\pi_1(Z)$ in $\pi_1(M)$.  Let $g_F$ be a choice of bundle metric of $F$ and let $\mu^q_{M_0}$ and $\mu^q_Z$ be bases of $\mathrm{IH}^q_{\bar m}(\bar M_0;F)\cong L^2_{g_0}H^q(M_0;F)$ and $\tH^q(Z;F)$, consisting of harmonic forms which are orthonormal with respect to the metrics $g_F$, $g_0$ and $g_Z$. Choose the basis $\mu_{\cC Z}$ for $\IH^q_{\bm}(\cC Z;F)$ induced by \eqref{rt.5a}.  Using these bases to define the corresponding $R$-torsions and analytic torsion, we have the following formula:
\begin{multline}
	\bar \lAT([M;H], \mu_{M_0},F) =
	\log \lrpar{ \frac{I\tau^{\bm}(\hM_0, \mu_{M_0},F)\tau(Z;F)}{I\tau^{\bm}(\cC Z,\mu_{\cC Z},F)^2} } -\sum_{q>\frac{m-1}2} (-1)^q  \frac{\dim H^q(Z;F)}{2} \log 2\\
	- \sum_{\substack{0\leq q\leq m-1 \\ q\neq \frac{m-1}2}} (-1)^q \frac{\dim \tH^q(Z;F)}{2}\left[ |m-1-2q|\log|m-1-2q| \right]. 
\end{multline}
\label{ff.19}\end{theorem}
\begin{proof} Assume (by an orthogonal change of basis, if necessary) that $\mu_{M_0}$ and $\mu_Z$ respect the decompositions \eqref{ff.7a} and \eqref{ff.7b}, then pick the basis $\mu$ for $\tH^q(M;F)$ induced by \eqref{ff.8} and \eqref{ff.9}.  The theorem then follows immediately from Theorems \ref{thm:MainAT} and \ref{rt.10} together with the Cheeger-M\"uller theorem on $M$. \end{proof}

In particular, applying this result to the case where $M_0= N_0 \sqcup N_0$ is the disjoint union of two copies of $(N_0,g_0),$ a manifold with cusp with link $Z$, where $g_0$ is product-type to order two, we obtain the following:
\begin{corollary}
Let $F\lra N$ be a flat Witt bundle with unimodular holonomy endowed with a bundle metric $g_F$ that extends smoothly to the double of $N$ across $\pa N.$  Suppose that the holonomy of $F$ is trivial when restricted to $Z$.  
Let $\mu_{N_0}$ and $\mu_Z$ be bases of  $\mathrm{IH}^q_{\bar m}(\bar N_0;F)\cong L^2_{g_0}H^q(N_0;F)$ and $\tH^q(Z;F)$ respectively, consisting of harmonic forms orthonormal with respect to  $g_F$, $g_0$ and $g_Z$. The canonical identification \eqref{rt.5a} gives a basis $\mu_{\cC Z}$ for $\mathrm{IH}^q_{\bar m}(\cC Z).$  Using these bases to define the corresponding $R$-torsions and analytic torsion, we have the following formula:
\begin{multline}
	\bar \lAT(N_0, \mu_{N_0}, F) = 
	\log \lrpar{ \frac{I\tau^{\bm}(\bar N_0, \mu_{N_0},F)\tau(Z;F)^{\frac12}}{I\tau^{\bm}(\cC Z,\mu_{\cC Z},F)} }  -\sum_{q>\frac{m-1}2} (-1)^q  \frac{\dim H^q(Z;F)}{4} \log 2\\
	- \sum_{\substack{0\leq q\leq m-1 \\ q\neq \frac{m-1}2}} (-1)^q \frac{\dim \tH^q(Z;F)}{4}\left[ |m-1-2q|\log|m-1-2q| \right].\end{multline}
\label{ff.20}\end{corollary}  
Note that  $F$ extends to a flat Witt vector bundle on the double $M=N_0\cup_{\pa N_0}N_0$ and the hypothesis on $g_F$ means that it comes from a bundle metric on $M$ as well, so the previous theorem does indeed apply.

\section{Cusp degeneration and the boundary of Teichm\"uller space} \label{wolpert}
 

In addition to analyzing the analytic torsion, our results can also be used to analyze the behavior of families of hyperbolic metrics on surfaces which approach the boundary of Teichm\"uller space, giving a new perspective on results of Wolpert and of Burger \cite{wol87,wol90,wol07,bur}. In particular, we analyze the so-called `plumbing construction'. First we describe this construction: let $R_0$ be a hyperbolic surface with nodes $p_1,$ $\dots,$ $p_m$, where each node represents a pair of cusps at punctures $a_i$ and $b_i$ of $R_0\setminus\{p_1,\ldots,p_m\}.$ Let $U_i^j,$ $j=1,$ $2,$ be neighborhoods of $a_i$ and $b_i$ respectively.  Suppose without loss of generality that each $U_i^j$ is a disk of radius $\gamma_0$ for some fixed $\gamma_0>0$ and that we have local coordinates $z_i: U^1_i\to U$ and $w_i: U^2_i\to U$ with $z_i(a_i)=0$, $w_i(b_i)=0$ and such that the hyperbolic metric $g_0$ takes the forms
\begin{equation}
           \left(\frac{|dz_i|}{|z_i| \log |z_i|}\right)^2 \quad \mbox{and} \quad \left(\frac{|dw_i|}{|w_i| \log |w_i|}\right)^2
\label{wol.1}\end{equation}
in terms of these coordinates.   Then, for a sufficiently small choice of $\gamma_0,$ there exists an open set $V$ disjoint from each $U_i^j$ and a set of Beltrami differentials $\nu_i,$ $i=1,\ldots,3g-3-m,$ which are supported in $V$ and which span the tangent space of the boundary of Teichm\"uller space at $R_0.$ For $s$ sufficiently close to the origin in $\mathbb C^{3g-3-m},$ we can solve the Beltrami equation for $\nu(s)=\sum_{i=1}^m s_i\nu_i$ and obtain a family of  hyperbolic surfaces (with cusps) which we denote $R_{\nu(s)}.$ The hyperbolic metrics $g_s$ on these surfaces are smooth in $s$ and are conformal to  the hyperbolic metric $g_0$ on $R_0$ in each $U_i^j$ \cite{wol90,wol07}.

We now introduce a degeneration parameter $\sigma=(\sigma_1,\ldots,\sigma_m)\in\mathbb C^m$ describing the opening of the nodes. For each $\sigma$ sufficiently close to the origin, and each $i\in[1,m],$ we remove a pair of disks $D_{\sigma_i}^j$ of radius $|\sigma_i|$ around the $i$th node and identify the annuli $U_i^1\setminus D_{\sigma_i}^1$ and $U_i^2\setminus D_{\sigma_i}^2,$ in complex coordinates $(z,w),$ by $zw=\sigma_i.$ This produces  new Riemann surface $R_{\sigma,s}$ spanning a neighborhood of $R_0$ in Teichm\"uller space \cite{wol87}.  Each surface $R_{\sigma,s}$ can be equipped with a unique hyperbolic metric $g_{\sigma,s}$ in the conformal class specified by the complex structure.

To describe the behavior of $g_{\sigma,s}$ as $\sigma\to 0$,  notice that the local model describing the opening of the node, the degeneration  fixture 
$$
\mathcal P=\{(z,w,\tau): zw=\tau, |z|,|w|,|\tau|<1\},
$$
is a complex manifold fibering over the disk $D=\{|\tau|<1\}$ with fiber above $\tau$ naturally identified with the annulus $|\tau|<|z|<1$ for $\tau\ne 0$.  On each fibre, the unique complete hyperbolic metric in the conformal class specified by the complex structure is given by
$$
    g_{\cP,\tau}= \left( \frac{\pi}{\log |\tau|} \csc\left(\pi \frac{\log|z|}{\log|\tau|} \right)\left|\frac{dz}{z} \right| \right)^2= \left( \frac{\pi}{\log |\tau|} \csc\left(\pi \frac{\log|w|}{\log|\tau|} \right)\left|\frac{dw}{w} \right| \right)^2.
$$  
At $\tau=0$ this model degenerates to give two cusps as in \eqref{wol.1}. 
In fact, making the change of variables $x=-\frac{\pi}{\log|\tau|}\cot\left( \pi \frac{\log|z|}{\log|\tau|} \right)$, $\theta=\arg z$, we obtain
\begin{equation}
        g_{\cP,\tau}=\frac{dx^2}{x^2+\epsilon^2}+ (x^2+\epsilon^2)d\theta^2, \quad \mbox{with} \;\epsilon= \frac{-\pi}{\log|\tau|},
\label{wol.2}\end{equation}
which is precisely the degeneration model considered in the present paper.  

For each sufficiently small $\sigma$, we can use this model and construct an approximate hyperbolic metric $h_{\sigma,s}$ by gluing. More precisely, let  $g_{\mathcal P,\sigma}$ be a metric which on each $U^j_i\setminus D^1_{\sigma_i}$ is given by the metric $g_{\mathcal P,\sigma_i}$. Then let $\eta$ be a cutoff function on $R_0$ which is zero within a distance $\gamma_0/2$ of each node and identically 1 outside a distance $2\gamma_0$ of each node, and whose gradient has support in a union of annuli with inner radius $\gamma_0/2$ and outer radius $2\gamma_0$ about each node. Finally, as in \cite{wol07}, define a new metric $h_{\sigma,s}$ on $R_0$ by
\[h_{\sigma,s}=g_0^{\eta}g_{\mathcal P,\sigma}^{1-\eta}.\]
This family of metrics is smooth in $(\sigma,s)$ away from the nodes \cite{wol07}. The metrics $h_{\sigma,s}$ are \emph{not} necessarily exactly hyperbolic, as their curvature may not be identically $-1$ on the annuli where $\nabla\eta$ may be nonzero. However, their curvature may be computed directly; call it $K_{\sigma,s}.$ Then we must have $g_{\sigma,s}=e^{2\varphi_{\sigma,s}}h_{\sigma,s},$ where $\varphi_{\sigma,s}$ is the solution of the prescribed-curvature equation
\begin{equation}
-\Delta_{h_{\sigma,s}}\varphi_{\sigma,s}-K_{\sigma,s}=e^{2\varphi_{\sigma,s}},
\label{wol.3}\end{equation}
where $\Delta_{h_{\sigma,s}}$ is the Laplacian with non-negative spectrum corresponding to the metric $h_{\sigma,s}$.  As shown in \cite[p.293]{wol87}, the metric $h_{\sigma,s}$ is a good approximation of $g_{\sigma,s}$ in the sense that
\begin{equation}
    \lim_{(\sigma,s)\to 0} \frac{g_{\sigma,s}}{h_{\sigma,s}}=1.
\label{wol.4}\end{equation}
See also \cite{wol90} for an expansion at $(\sigma,s)=0$.  More recently, Melrose and Zhu \cite{Melrose-Zhu} have shown that the conformal factor $\varphi_{\sigma,s}$ is in fact log-smooth on the single surgery space $X_s$ associated to the family of metrics $h_{\sigma,s}$.

 In \cite{wol87}, Wolpert also obtains estimates for the small eigenvalues and for the determinants of the Laplacian of the metrics $g_{\sigma,s}$; see also \cite{bur} for a sharpening of the eigenvalue asymptotics.  Since the metrics $h_{\sigma,s}$ are exactly of the form \eqref{wol.2} in a fixed neighborhood of the nodes, our work may be used to recover and extend these results in many cases. Indeed, for any fixed $\tau\in\mathbb C^m$, we may directly apply our results to $h_{e^{-\pi/\eps}\tau,s},$ so using \eqref{wol.4} together with the prescribed curvature equation \eqref{wol.3} and its solution $\varphi_{e^{-\pi/\eps}\tau,s}$, we can derive corresponding results for $g_{e^{-\pi/\eps}\tau,s}.$

\subsection{Small eigenvalues} First we analyze the behavior of the small eigenvalues of $\Delta_{g_{\sigma,s}}.$ The curvature $K_{e^{-\pi/\eps}\tau,s}$ and the corresponding solution of the prescribed curvature equation are analyzed carefully in \cite{wol90}. From \cite[Section 3]{wol90}, we conclude that $||K_{e^{-\pi/\eps}\tau,s}+1||_{C^0}\to 0$ as $\eps\to 0,$ uniformly in $s.$ (In fact, $||K_{e^{-\pi/\eps}\tau,s}+1||_{C^0}=C\eps^2+\cO(\eps^4)$)). As a consequence, it is proved in \cite[Section 4]{wol90} that for any $\delta$ there exists an $\epsilon_0$ such that if $\eps<\epsilon_0,$ then for all sufficiently small $s$ and all points on the surface,
\[(1-\delta)h_{e^{-\pi/\eps}\tau,s}\leq g_{e^{-\pi/\eps}\tau,s}\leq(1+\delta)h_{e^{-\pi/\eps}\tau,s}.\]
By the well-known result of Dodziuk \cite[Prop. 3.3]{Dod82}, the quotients of the nonzero eigenvales of the Laplacians for $(R,g_{e^{-\pi/\eps}\tau,s})$ and $(R,h_{e^{-\pi/\eps}\tau,s})$ therefore approach 1 as $\epsilon\to 0,$ which allows us to apply our analysis of the small eigenvalues in section \ref{sec:Small} to conclude:

\begin{proposition} \label{lem:BurgerSmallEigen}
As $\epsilon$ goes to zero, the positive small eigenvalues $\lambda_{\eps,s}$ of $\Delta_{ g_{e^{-\pi/\eps}\tau,s}}$ satisfy
\[\lambda_{\eps,s}\sim c\epsilon+o(\epsilon),\]
where $c$ can be computed explicitly for each small eigenvalue using the methods of section \ref{sec:Small}.
\end{proposition}
As a particular case, consider the situation with $i=1.$ From the long exact sequence \eqref{se.4}, there is one positive small eigenvalue if the manifold $R_{\epsilon}$ becomes disconnected in the limit and zero if it does not. If it becomes disconnected and the volumes of the two connected components of the limit are $V_1$ and $V_2,$ we conclude from Lemma \ref{se.8} that the leading asymptotic of the single small eigenvalue is $\frac{V_1+V_2}{\pi V_1 V_2}\epsilon.$ This agrees with the result of Burger \cite{bur}, who computed these eigenvalue asymptotics, to the same accuracy and with specific values of $c,$ using methods involving a comparison with the graph Laplacian.

We can also describe the asymptotic behavior of eigenfunctions associated to small eigenvalues.

\begin{proposition}
Let $f_{\eps,s}$ be a real-valued eigenfunction of a small eigenvalue $\lambda_{\eps,s}$ of $\Delta_{g_{e^{-\pi/\eps}\tau,s}}$ which is continuous in $\eps>0$ and $s$ and with $L^2$-norm equal to $1$ for each $\eps$ and $s$.  Then $f_{\eps,s}$ extends to be continuous and bounded on $X_s\times V$, where $X_s$ is the single surgery space associated to $R_{\eps}$ and $V\subset \bbC^{3g-3-m}$ is the space of  deformations corresponding to the parameter $s$.   Moreover, for $s$ fixed, it restricts to a constant on each connected component of $\sm$, whereas near a connected component $\bhs{sb,i}$ of $\bs$,  it is of the form
\begin{equation}
  f_{\eps,s}= \frac{c^+_{i}-c^-_{i}}{\pi} \arctan\left(\frac{x_i}{\eps}\right) + \frac{c_{+,i}+c_{-,i}}2 + \mathcal{O}((\sqrt{x_i^2+\eps^2})^{\delta})
\label{ef.2}\end{equation}
for some small $\delta>0$, where $x_i$ is the coordinate in \eqref{wol.2} near the connected component $\bhs{sb,i}$ and $c^{\pm}_i$ is the constant value such that $f_{\eps,s}=c^{\pm}_i$ on the connected component of $\sm$ intersecting $\bhs{sb,i}$ with $\pm x_i\ge0$.  Finally, if there is only one positive small eigenvalue, then $f_{\eps,s}$ is polyhomogeneous on $X_s\times V$. 
\label{ef.1}\end{proposition}
\begin{proof}
By the result of Melrose-Zhu \cite{Melrose-Zhu}, the conformal factor $\varphi_{\eps,s}$ polyhomogeneous on $X_s\times V$, so $g_{e^{-\pi/\eps}\tau,s}$ is a polyhomogeneous product-type $\ehc$-metric, which means that  Lemma~\ref{se.8} and Theorem~\ref{se.11} do apply.  In particular, we have that 
$$
     f_{\eps,s}= \sum_{j=1}^{d} \mu_{j}(\eps,s) f_{j,\eps,s},
$$  
where $f_{j,\eps,s}$ are polyhomogeneous real-valued functions forming an orthonormal basis of the range of the projection $\Pi^0_{+,\sm}$ in Theorem~\ref{se.11}, and the $\mu_j(\eps,s)$ are bounded continuous function in $s$ and $\eps\ge 0$.  In particular, if $d=1$, then $f_{\eps,s}=\pm f_{1,\eps,s}$ is clearly polyhomogeneous.

To prove the result about the asymptotic behavior of $f_{\eps,s}$, it suffices thus to prove a corresponding result for the functions $f_{j,\eps,s}$.  This means that we can use Lemma~\ref{se.8} and its proof, in particular equation \eqref{se.10a}, to conclude that $\rho^{\frac12}f_{\eps,s}$, which is an eigenfunction for the conjugated Laplacian 
$\rho^{\frac12} (\Delta_{g_{e^{-\pi/\eps}\tau,s}})\rho^{-\frac12}$, is such that near $\bhs{sb,i}$,
\begin{equation}
  \rho^{\frac12}f_{\eps,s}= \eps^{\frac12}\left(1+\left(\frac{x_i}{\eps}\right)^2\right)^{\frac14}\left( \frac{c^+_i-c^-_i}{\pi}\arctan \left(\frac{x_i}{\eps}\right) + \frac{c_i^+ +c^-_i}2 \right) + \mathcal{O}(\rho^{\frac12+\delta})
\label{ef.3}\end{equation}   
for some $\delta>0$ small.  Hence, multiplying by $\rho^{-\frac12}$ gives \eqref{ef.2} as desired.  Since 
$$
\lim_{X\to\pm\infty } \arctan X= \pm \frac{\pi}2,
$$
we also see that $f_{\eps,s}$ is continuous and bounded on $X_s\times V$ as claimed.   
\end{proof}

\subsection{Determinant} We can also analyze the determinant of the Laplacian $\Delta_{g_{e^{-\pi/\eps}\tau,s}}.$   Observe that it is easy to show, by applying the maximum principle to the prescribed curvature equation, that whenever $\epsilon$ is small enough so that $||K_{e^{-\pi/\eps}\tau,s}+1||_{C^0}\leq 1/2,$ then there is a constant $C$ such that $||\varphi_{e^{-\pi/\eps}\tau,s}||_{C^0}\leq C.$ By substituting these bounds into the prescribed curvature equation, we also get a $C^0$ bound for $\Delta\varphi_{e^{-\pi/\eps}\tau,s}.$ Therefore, for sufficiently small $\eps,$ there is a universal constant $C$ such that
\[||\varphi_{e^{-\pi/\eps}\tau,s}||_{C^0}+||\Delta\varphi_{e^{-\pi/\eps}\tau,s}||_{C^0}\leq C.\]
We may now apply the Polyakov formula in the form from \cite{OPS88}:
\[\log\det\Delta_{g_{e^{-\pi/\eps}\tau,s}}=\log\det\Delta_{h_{e^{-\pi/\eps}\tau,s}}-\frac{1}{12\pi}\int_R(|\nabla\varphi_{e^{-\pi/\eps}\tau,s}|^2+ 2K_{e^{-\pi/\eps}\tau,s}) dh_{e^{-\pi/\eps}\tau,s}.\]
Using integration by parts and the bounds we have just proven, it is straightforward to show that for sufficiently small $\eps,$ there is a universal constant $C$ (possibly different from the one above) such that
\begin{equation}\label{quasicompare}|\log\det\Delta_{g_{e^{-\pi/\eps}\tau,s}}-\log\det\Delta_{h_{e^{-\pi/\eps}\tau,s}}|\leq C.\end{equation}
We now claim:
\begin{proposition} \label{lem:WolpertDet} 
As $\eps\to 0,$ $\log\det\Delta_{h_{e^{-\pi/\eps}\tau,s}}\to -\infty.$
\end{proposition}
\begin{proof} Since $h_{e^{-\pi/\eps}\tau,s}$ is a family of cusp surgery metrics, the log determinant in question has a polyhomogeneous expansion as $\eps\to 0.$ We need to investigate all the divergent terms in this expansion, so we need to be specific about the leading orders.  Thanks to Proposition~\ref{lem:BurgerSmallEigen}, we see from \cite[Theorem~11.2]{ARS1} that
\begin{equation}
\log\det\Delta_{h_{e^{-\pi/\eps}\tau,s}}= \frac{C_0}{\eps} \log \eps + C_1\eps^{-1}+C_2(\log\eps)^2+C_3\log\eps+C_4 + o(1)
\label{wol.5}\end{equation}
as $\eps\searrow 0$.  However, there is also a criterion in \cite[Theorem~11.2]{ARS1} to ensure that $C_0=0$.  This criterion does in fact apply in this setting, but not in the obvious way, since the dimension of the total space is not odd.  Indeed, it is well-known that the short time expansion of the trace of the heat kernel of the Laplacian on a circle is of the form
$$
        \Tr (e^{-t\Delta_{\bbS^1}})= c_{-1}t^{-1} + o(t^N)
$$ 
with $N\in \bbN$ as large as we want.  This can be deduced for instance from the fact that the Riemann zeta function $\zeta_{Riem}(s)$ has only a pole at $s=1$.  This means that the model operator in \cite[(7.2)]{ARS1} on the front face $\bhs{tff}$ of the heat space, when pushforward to $\bhs{tff}(\sE\sT)$, has no term of order zero at $\bhs{tf}(\sE\sT)$.  This is exactly the criterion of \cite[Theorem~11.2]{ARS1}, so that in fact  $C_0=0$ and 
\begin{equation}
\log\det\Delta_{h_{e^{-\pi/\eps}\tau,s}}=  C_1\eps^{-1}+C_2(\log\eps)^2+C_3\log\eps+C_4 + o(1).
\label{wol.5}\end{equation}
Thus, it suffices to show that $C_1<0$ to complete the proof.

However, from the definition of the determinant in \cite[equation (10.3)]{ARS1} and the renormalized push-forward theorem, we know that
\[C_1=-2\Rint_0^{\infty}A_{tff}\frac{d\sigma}{\sigma},\]
where $A_{tff}$ is the coefficient of the $\eps^{-1}$ term in the expansion of the renormalized trace at $\bhs{tff}.$ From \cite[Sec. 7]{ARS1},
\[A_{tff}=N_{tff}(A)=e^{-\sigma^2 \Delta_{S^1}}\frac1{\sqrt{4\pi}\sigma}\exp\lrpar{ -\frac{|\cdot|^2_{\Ed N\bhs{tf}} }{2\sigma^2}} \mu_{\Ephi N\bhs{sb} \btimes_Y H / H}.\]
Therefore, restricting to the diagonal and integrating yields
\[C_1=-\frac{1}{\sqrt{\pi}} \Rint_0^{\infty}\Tr (e^{-\sigma^2\Delta_{S^1}})\sigma^{-2}\ d\sigma=-\frac{1}{\sqrt{4\pi}} \Rint_0^{\infty} \Tr(e^{-t\Delta_{S^1}})t^{-3/2}\ dt.\] This renormalized integral may be evaluated and gives $C_1=-\frac{1}{\sqrt{\pi}}\Gamma(-1/2)\zeta_{Riem}(-1),$ which is negative. 
\end{proof}
Combining this lemma with \eqref{quasicompare}, we obtain a leading-order asymptotic formula for $\log\det\Delta_{g_{e^{-\pi/\eps}\tau,s}}$. In particular, we see that
\[\log\det\Delta_{g_{e^{-\pi/\eps}\tau,s}}\to -\infty\textrm{ as }\eps\to 0,\]
which agrees with the result of Wolpert \cite[Theorem 5.3]{wol87} when we use the description of the determinant in terms of the Selberg zeta function.

\bibliographystyle{amsalpha}

\bibliography{cuspsurgres}

\providecommand{\bysame}{\leavevmode\hbox to3em{\hrulefill}\thinspace}
\providecommand{\MR}{\relax\ifhmode\unskip\space\fi MR }
\providecommand{\MRhref}[2]{%
  \href{http://www.ams.org/mathscinet-getitem?mr=#1}{#2}
}
\providecommand{\href}[2]{#2}
\begin{thebibliography}{HHM04}

\bibitem[Alb09]{Albin:RenInt}
Pierre Albin, \emph{Renormalizing curvature integrals on
  {P}oincar\'e-{E}instein manifolds.}, Adv. Math. \textbf{221} (2009), no.~1,
  140--169.

\bibitem[AR13]{Albin-Rochon:ModSpace}
Pierre Albin and Fr{\'e}d{\'e}ric Rochon, \emph{Some index formulae on the
  moduli space of stable parabolic vector bundles}, J. Aust. Math. Soc.
  \textbf{94} (2013), no.~1, 1--37.

\bibitem[ARS14]{ARS1}
Pierre Albin, Fr{\'e}d{\'e}ric Rochon, and David Sher, \emph{Resolvent, heat
  kernel, and torsion under degeneration to fibered cusps}, arXiv:1410.8406,
  2014.

\bibitem[BMZ17]{BMZ}
Jean-Michel Bismut, Xiaonan Ma, and Weiping Zhang, \emph{Asymptotic torsion and
  {T}oeplitz operators}, J. Inst. Math. Jussieu \textbf{16} (2017), no.~2,
  223--349. \MR{3615411}

\bibitem[BSV16]{BSV}
Nicolas Bergeron, Mehmet~Haluk Seng\"un, and Akshay Venkatesh, \emph{Torsion
  homology growth and cycle complexity of arithmetic manifolds}, Duke Math. J.
  \textbf{165} (2016), no.~9, 1629--1693. \MR{3513571}

\bibitem[Bur88]{bur}
Marc Burger, \emph{Asymptotics of small eigenvalues of {R}iemann surfaces},
  Bull. Amer. Math. Soc. (N.S.) \textbf{18} (1988), no.~1, 39--40.

\bibitem[BV13]{BV}
Nicolas Bergeron and Akshay Venkatesh, \emph{The asymptotic growth of torsion
  homology for arithmetic groups}, J. Inst. Math. Jussieu \textbf{12} (2013),
  no.~2, 391--447.

\bibitem[BW80]{Borel-Wallach}
A.~Borel and N.~Wallach, \emph{Continuous cohomology, discrete subgroups, and
  representations of reductive groups}, Ann. of Math. Stud., vol.~94, Princeton
  University Press, 1980.

\bibitem[BZ92]{Bismut-Zhang}
Jean-Michel Bismut and Weiping Zhang, \emph{An extension of a theorem by
  {C}heeger and {M}\"uller}, Ast\'erisque (1992), no.~205, 235, With an
  appendix by Francois Laudenbach. \MR{1185803 (93j:58138)}

\bibitem[Che79]{Cheeger1979}
Jeff Cheeger, \emph{Analytic torsion and the heat equation}, Ann. of Math. (2)
  \textbf{109} (1979), no.~2, 259--322.

\bibitem[CV12]{Calegari-Venkatesh}
Frank Calegari and Akshay Venkatesh, \emph{A torsion {J}acquet-{L}anglands
  correspondence}, available online at arXiv:1212.3847, 2012.

\bibitem[Dar87]{Dar}
Aparna Dar, \emph{Intersection {$R$}-torsion and analytic torsion for
  pseudomanifolds}, Math. Z. \textbf{194} (1987), no.~2, 193--216.

\bibitem[DH14]{Dai-Huang}
Xianzhe Dai and Xiaoling Huang, \emph{The intersection {R}-torsion for finite
  cone}, available online at arXiv:1410.6110, 2014.

\bibitem[Dod82]{Dod82}
Jozef Dodziuk, \emph{Eigenvalues of the {L}aplacian on forms}, Proc. Amer.
  Math. Soc. \textbf{85} (1982), no.~3, 437--443.

\bibitem[Fri86]{Fried1986}
David Fried, \emph{Analytic torsion and closed geodesics on hyperbolic
  manifolds}, Invent. Math. \textbf{84} (1986), no.~3, 523--540. \MR{837526
  (87g:58118)}

\bibitem[GM80]{GM1980}
Mark Goresky and Robert MacPherson, \emph{Intersection homology theory},
  Topology \textbf{19} (1980), no.~2, 135--162.

\bibitem[GM83]{GM1983}
\bysame, \emph{Intersection homology. {II}}, Invent. Math. \textbf{72} (1983),
  no.~1, 77--129.

\bibitem[GS15]{Guillarmou-Sher}
Colin Guillarmou and David~A. Sher, \emph{Low energy resolvent for the {H}odge
  {L}aplacian: applications to {R}iesz transform, {S}obolev estimates, and
  analytic torsion}, Int. Math. Res. Not. IMRN (2015), no.~15, 6136--6210.
  \MR{3384474}

\bibitem[Har75]{Harder}
G.~Harder, \emph{On the cohomology of discrete arithmetically defined groups},
  Discrete subgroups of {L}ie groups and applications to moduli ({I}nternat.
  {C}olloq., {B}ombay, 1973), Oxford Univ. Press, Bombay, 1975, pp.~129--160.

\bibitem[HHM04]{hhm}
Tam{\'a}s Hausel, Eugenie Hunsicker, and Rafe Mazzeo, \emph{Hodge cohomology of
  gravitational instantons}, Duke Math. J. \textbf{122} (2004), no.~3,
  485--548.

\bibitem[HS10]{Hartmann-Spreafico2010}
L.~Hartmann and M.~Spreafico, \emph{The analytic torsion of a cone over a
  sphere}, J. Math. Pures Appl. (9) \textbf{93} (2010), no.~4, 408--435.

\bibitem[HS11]{Hartmann-Spreafico2011}
\bysame, \emph{The analytic torsion of a cone over an odd dimensional
  manifold}, J. Geom. Phys. \textbf{61} (2011), no.~3, 624--657.

\bibitem[Les13]{Lesch2013}
Matthias Lesch, \emph{A gluing formula for the analytic torsion on singular
  spaces}, Anal. PDE \textbf{6} (2013), no.~1, 221--256. \MR{3068545}

\bibitem[Mel93]{MelroseAPS}
Richard~B. Melrose, \emph{The {A}tiyah-{P}atodi-{S}inger index theorem},
  Research Notes in Mathematics, vol.~4, A K Peters Ltd., Wellesley, MA, 1993.

\bibitem[MFP14]{Menal-Ferrer-Porti:HigherDRTCuspedHyp3Mfds}
Pere Menal-Ferrer and Joan Porti, \emph{Higher-dimensional {R}eidemeister
  torsion invariants for cusped hyperbolic 3-manifolds}, J. Topol. \textbf{7}
  (2014), no.~1, 69--119.

\bibitem[Mil66]{Milnor1966}
J.~Milnor, \emph{Whitehead torsion}, Bull. Amer. Math. Soc. \textbf{72} (1966),
  358--426. \MR{0196736 (33 \#4922)}

\bibitem[MM13]{Marshall-Muller}
Simon Marshall and Werner M{\"u}ller, \emph{On the torsion in the cohomology of
  arithmetic hyperbolic 3-manifolds}, Duke Math. J. \textbf{162} (2013), no.~5,
  863--888.

\bibitem[MN96]{Melrose-Nistor}
Richard~B. Melrose and Victor Nistor, \emph{Homology of pseudodifferential
  operators {I}. {M}anifolds with boundary}, Available online at ArXiv
  /funct-an/9606005, 1996.

\bibitem[MP12]{Muller-Pfaff:ATComHypMfdsFinVol}
Werner M{\"u}ller and Jonathan Pfaff, \emph{Analytic torsion of complete
  hyperbolic manifolds of finite volume}, J. Funct. Anal. \textbf{263} (2012),
  no.~9, 2615--2675.

\bibitem[MP13a]{Muller-Pfaff:ATAsympBhvSeqHypMfdsFinVol}
\bysame, \emph{The analytic torsion and its asymptotic behaviour for sequences
  of hyperbolic manifolds of finite volume}, Available online at
  arXiv:1307.4914, 2013.

\bibitem[MP13b]{Muller-Pfaff:ATL2TorsionCmptLocSymSpaces}
\bysame, \emph{Analytic torsion and {$L^2$}-torsion of compact locally
  symmetric manifolds}, J. Differential Geom. \textbf{95} (2013), no.~1,
  71--119.

\bibitem[MP13c]{Muller-Pfaff:OnAsympRSATCmptHypMfds}
\bysame, \emph{On the asymptotics of the {R}ay-{S}inger analytic torsion for
  compact hyperbolic manifolds}, Int. Math. Res. Not. IMRN (2013), no.~13,
  2945--2983.

\bibitem[MP14]{Muller-Pfaff:GrowthTorsionCohoArithGps}
\bysame, \emph{On the growth of torsion in the cohomology of arithmetic
  groups}, Math. Ann. \textbf{359} (2014), no.~1-2, 537--555.

\bibitem[MR04]{Melrose-Rochon:Cusp}
Richard Melrose and Fr{\'e}d{\'e}ric Rochon, \emph{Families index for
  pseudodifferential operators on manifolds with boundary}, Int. Math. Res.
  Not. (2004), no.~22, 1115--1141.

\bibitem[M{\"u}l78]{Muller1978}
Werner M{\"u}ller, \emph{Analytic torsion and {$R$}-torsion of {R}iemannian
  manifolds}, Adv. in Math. \textbf{28} (1978), no.~3, 233--305.

\bibitem[M{\"u}l93]{Muller1993}
\bysame, \emph{Analytic torsion and {$R$}-torsion for unimodular
  representations}, J. Amer. Math. Soc. \textbf{6} (1993), no.~3, 721--753.

\bibitem[M{\"u}l12]{Muller:AsympRSTHyp3Mfds}
\bysame, \emph{The asymptotics of the {R}ay-{S}inger analytic torsion for
  hyperbolic {$3$}-manifolds}, Metric and {D}inferential Geometry. {T}he {J}eff
  {C}heeger {A}nniversary {V}olume, Progress in Math., vol. 297,
  Birkh{\"a}user, 2012, pp.~317--352.

\bibitem[MV12]{Mazzeo-Vertman}
Rafe Mazzeo and Boris Vertman, \emph{Analytic torsion on manifolds with edges},
  Adv. Math. \textbf{231} (2012), no.~2, 1000--1040.

\bibitem[MZ15]{Melrose-Zhu}
R.~B. Melrose and X.~Zhu, \emph{Resolution of the canonical fiber metric for a
  lefschetz fibration}, arXiv:1501.04124, 2015.

\bibitem[OPS88]{OPS88}
Brad Osgood, Ralph Phillips, and Peter Sarnak, \emph{Extremals of determinants
  of laplacians}, J. Funct. Anal. \textbf{80} (1988), no.~1, 148--211.

\bibitem[Par09]{Park}
Jinsung Park, \emph{Analytic torsion and {R}uelle zeta functions for hyperbolic
  manifolds with cusps}, J. Funct. Anal. \textbf{257} (2009), no.~6,
  1713--1758.

\bibitem[Pfa14a]{Pfaff:ATRTHyp3MFdsCusps}
Jonathan Pfaff, \emph{Analytic torsion versus {R}eidemeister torsion on
  hyperbolic 3-manifolds with cusps}, Math. Z. \textbf{277} (2014), no.~3-4,
  953--974.

\bibitem[Pfa14b]{Pfaff:ExpGrowthHomTorTowerCongSubgpsBianchi}
\bysame, \emph{Exponential growth of homological torsion for towers of
  congruence subgroups of {B}ianchi groups}, Ann. Global Anal. Geom.
  \textbf{45} (2014), no.~4, 267--285.

\bibitem[Pfa15]{Pfaff:SelbergZetaFunOddDHypMfdsFinVol}
\bysame, \emph{Selberg zeta functions on odd-dimensional hyperbolic manifolds
  of finite volume}, J. Reine Angew. Math. \textbf{703} (2015), 115--145.
  \MR{3353544}

\bibitem[Pfa17]{Pfaff:GluingFormATHypMfdsCusps}
\bysame, \emph{A {GLUING} {FORMULA} {FOR} {THE} {ANALYTIC} {TORSION} {ON}
  {HYPERBOLIC} {MANIFOLDS} {WITH} {CUSPS}}, J. Inst. Math. Jussieu \textbf{16}
  (2017), no.~4, 673--743. \MR{3680342}

\bibitem[Rai12]{Raimbault:Asymp}
Jean Raimbault, \emph{Asymptotics of analytic torsion for hyperbolic
  three--manifolds}, available online at arXiv:1212.3161, 2012.

\bibitem[Rai13]{Raimbault:ARHtorsion}
\bysame, \emph{Analytic, {R}eidemeister and homological torsion for congruence
  three--manifolds}, available online at arXiv:1307.2845, 2013.

\bibitem[RS71]{Ray-Singer}
D.~B. Ray and I.~M. Singer, \emph{{$R$}-torsion and the {L}aplacian on
  {R}iemannian manifolds}, Advances in Math. \textbf{7} (1971), 145--210.

\bibitem[She15]{Sher:ConicDeg}
David~A. Sher, \emph{Conic degeneration and the determinant of the
  {L}aplacian}, J. Anal. Math. \textbf{126} (2015), 175--226. \MR{3358031}

\bibitem[SS88]{Seeley-Singer}
R.~Seeley and I.~M. Singer, \emph{Extending {$\overline\partial$} to singular
  {R}iemann surfaces}, J. Geom. Phys. \textbf{5} (1988), no.~1, 121--136.

\bibitem[Vai01]{v}
Boris Vaillant, \emph{Index and spectral theory for manifolds with generalized
  fibred cusps}, available online at arXiv: math/0102072v1, 2001.

\bibitem[Ver09]{Vertman}
Boris Vertman, \emph{Analytic torsion of a bounded generalized cone}, Comm.
  Math. Phys. \textbf{290} (2009), no.~3, 813--860. \MR{2525641 (2010d:58032)}

\bibitem[Ver14]{Vertman:CheegerMuller}
\bysame, \emph{{C}heeger-{M}ueller {T}heorem on manifolds with cusps},
  available online at arXiv:1411.0615, 2014.

\bibitem[Wol87]{wol87}
Scott~A. Wolpert, \emph{Asymptotics of the spectrum and the {S}elberg zeta
  function on the space of {R}iemann surfaces}, Comm. Math. Phys. \textbf{112}
  (1987), no.~2, 283--315.

\bibitem[Wol90]{wol90}
\bysame, \emph{The hyperbolic metric and the geometry of the universal curve},
  J. Differential Geom. \textbf{31} (1990), no.~2, 417--472.

\bibitem[Wol10]{wol07}
\bysame, \emph{Families of {R}iemann surfaces and {W}eil-{P}etersson geometry},
  CBMS Regional Conference Series in Mathematics, vol. 113, Published for the
  Conference Board of the Mathematical Sciences, Washington, DC; by the
  American Mathematical Society, Providence, RI, 2010.

\end{thebibliography}

\end{document}